\documentclass[12pt]{amsart}
\usepackage{esvect,amssymb,stmaryrd,wasysym,amscd,amsthm,verbatim,nicefrac,amsmath,color,fancyhdr,mathrsfs,braket,graphicx, turnstile,hyperref,indentfirst,csquotes,url,enumitem,amsfonts, mathtools,lipsum}
\usepackage[letterpaper, left=2.5cm, right=2.5cm, top=2.5cm,
bottom=2.5cm,dvips]{geometry}
\hypersetup{
    colorlinks=true,
    linkcolor=blue,
    filecolor=magenta,      
    urlcolor=cyan,
    citecolor=blue,
    pdftitle={Minimising the Peak to $p$-average ratio in the vectorial case},
    pdfpagemode=FullScreen,
    }

\numberwithin{equation}{subsection}
\let\oldsection\section
\renewcommand{\section}{
  \renewcommand{\theequation}{\thesection.\arabic{equation}}
  \oldsection}
\let\oldsubsection\subsection
\renewcommand{\subsection}{
  \renewcommand{\theequation}{\thesubsection.\arabic{equation}}
  \oldsubsection}

\title[Semi-differentiability of generalised $\L^\infty$ envelopes]{Semi-differentiability of generalised $\L^\infty$ envelopes with applications to supremal functionals}

\author{Nikos Katzourakis}

\address{Department of Mathematics and Statistics, University of Reading, Whiteknights Campus, Pepper Lane, Reading RG6 6AX, UNITED KINGDOM}

\email{n.katzourakis@reading.ac.uk}

\makeatletter
\@namedef{subjclassname@2020}{\textup{2020} Mathematics Subject Classification}
\makeatother



\thanks{\!\!\!\!\!\!\!\texttt{The author has been partially financially supported through the EPSRC grant EP/X017109/1.}}

\def\R{\mathbb{R}}

\def\H{\mathrm{H}}
\def\E{\mathrm{E}}
\def\L{\mathrm{L}}

\def\D{\mathrm{D}}
\def\Om{\Omega}
\def\Argmax{\mathrm{Argmax}}

\def\X{\mathrm{X}}
\def\mX{\mathfrak{X}}
\def\F{\mathrm{F}}

\def\p{\mathrm{p}}

\def\c{\mathrm{c}}

\newcommand{\av}{-\hspace{-13pt}\displaystyle\int}

\def\weaklyconv{-\!\!\!\!\!-\!\!\!\!\rightharpoonup}

\def\weakstar{\,\overset{*_{\phantom{|}}}{{\smash{\weaklyconv }}\,}}

\newcommand{\bfX}{\mathbf{X}}

\newcommand{\C}{\mathrm{C}}

\renewcommand{\set}{\setminus}

\newcommand{\mL}{\mathcal{L}}
\newcommand{\mH}{\mathcal{H}}
\newcommand{\mF}{\mathcal{F}}

\newcommand{\N}{\mathbb{N}}

\newcommand{\K}{\mathcal{K}}

\newcommand{\mO}{\mathcal{O}}

\newcommand{\mD}{\mathcal{D}}
\newcommand{\mfD}{\mathfrak{D}}

\newcommand{\Argmin}{\mathrm{Argmin}}

\newcommand{\I}{\mathrm{I}}

\newcommand{\mB}{\mathbb{B}}
 
\newcommand{\noi}{\noindent}
\newcommand{\ms}{\medskip}

\newcommand{\al}{\alpha}
\newcommand{\be}{\beta}

\newcommand{\Ga}{\Gamma}
\newcommand{\de}{\delta}
\newcommand{\De}{\Delta}
\newcommand{\e}{\varepsilon}
\newcommand{\si}{\sigma}

\newcommand{\la}{\lambda}

\newcommand{\om}{\omega}

\newcommand{\ze}{\zeta}

\newcommand{\weak}{\, -\!\!\!\!\!-\!\!\!\!\rightharpoonup}

\newcommand{\larrow}{\longrightarrow}
\newcommand{\ot}{\otimes}

\renewcommand{\L}{\mathrm L}
\newcommand{\W}{\mathrm W}
\renewcommand{\p}{\partial}

\newcommand{\sub}{\subseteq}

\newcommand{\by}{\times}

\newcommand{\ess}{\mathrm{ess}}
\newcommand{\dist}{\mathrm{dist}}

\renewcommand{\div}{\mathrm{div}}
\newcommand{\supp}{\mathrm{supp}}

\newcommand{\beq}{\begin{equation}}

\newcommand{\eeq}{\end{equation}}

\newtheorem{theorem}{Theorem}
\newtheorem{corollary}[theorem]{Corollary}
\newtheorem{lemma}[theorem]{Lemma}
\newtheorem{proposition}[theorem]{Proposition}

\newtheorem{claim}[theorem]{Claim}
\newtheorem{example}[theorem]{Example}

\theoremstyle{definition}
\newtheorem{definition}[theorem]{Definition}
\newtheorem{remark}[theorem]{Remark}

\newcommand{\BPL}{\medskip \noindent \textbf{Proof of Lemma} }
\newcommand{\BPC}{\medskip \noindent \textbf{Proof of Claim} }

\newcommand{\BPP}{\medskip \noindent \textbf{Proof of Proposition} }
\newcommand{\BPT}{\medskip \noindent \textbf{Proof of Theorem} }

\begin{document}

\maketitle


\begin{abstract}
\!J.\,Danskin \cite{Da} proved in 1966 that the envelope of a family of continuous functions $\mathrm F : \mathbb R^n \times \mathrm K \longrightarrow \mathbb R$, parameterised by the points of a compact $\mathrm K \subseteq \mathbb R^m$, given by 
\[
f(u) := \max_{k \in \mathrm K} \mathrm F(u,k), \ \ \ \ f  \, : \, \mathbb R^n  \longrightarrow \mathbb R,
\]
is semi-differentiable on $\mathbb R^n$, if $\mathrm F$ is sufficiently regular. This result is of utmost importance in numerous applications, including game theory, economics, finance, and optimisation. Extensions have been proved towards almost every direction, but none permits to replace ``max over $\mathrm K$" with ``essential sup over $\mathrm K$". This is not a coincidence, as examples show that no direct generalisation can exist. By introducing some measure-theoretic apparatus, we establish the semi-differentiability of generalised $\mathrm L^\infty$ envelopes when $ \mathrm F$ is defined on the product of a Banach space with a measure space. As an application, we obtain a new foundational regularity result in the Calculus of Variations in $\mathrm L^\infty$, asserting that \emph{supremal functionals are semi-differentiable everywhere}, with an \emph{explicit formula for their semi-derivative}, despite generally being non-differentiable. Moreover, we establish a \emph{variational characterisation} of (absolute) minimisers of general level-convex  $\mathrm L^\infty$ functionals via their semi-differentials, revealing the \emph{genuine $\mathrm L^\infty$ counterpart of the Euler-Lagrange equations}. Conventional Aronsson equations and related divergence PDE involving measures, deducible from $\mathrm L^p$ approximations as $p\to \infty$, in general are \emph{not sufficient} for minimality in $\mathrm L^\infty$. Our results dislodge the trademark necessity for extrinsic $\mathrm L^p$-approximations as $p\to \infty$ to derive and study PDEs, reinterpreting the existing $\mathrm L^\infty$ theory. Most crucially, they provide a new powerful intrinsic $\mathrm L^\infty$ framework, evocative of the straightforward variational methods for integral functionals.
\end{abstract}

%
 
\tableofcontents

\section{Introduction}

\noi {\bf The context.} The starting point of this work is a problem first studied by J.\ Danskin in the 1960s \cite{Da}: given a continuous function $F : \R^n \by \K \larrow \R$, where $\K \sub \R^m$ is compact and $n,N\in \N$, consider the ``envelope" function $f : \R^n \larrow \R$ given by
\beq
\label{1.1}
f(u):= \max_{k\in \K} F(u,k). 
\eeq
Suppose $F$ is differentiable with respect to the first variable. When is the envelope $f$ differentiable? The non-linearity of the maximum operation implies that differentiability of $f$ fails even in the simplest cases. For example, let $F : \R \by [-1,1]\larrow \R$ be given by $F(u,k):=uk$. Then, 
\[
f(u)=\max_{k\in[-1,1]}uk=|u|, 
\]
and $f'(0)$ does not exist. However, as Danskin proved, the envelope $f$ is everywhere semi-differentiable in the Gateaux sense, and the following result can be established:

\smallskip

\noi {\bf Theorem}\,(Danskin 1966 \cite{Da}){\bf .} \emph{Let $F \in C(\R^n \by \K)$ be given, where $\K \sub \R^m$ is compact, and suppose the partial derivatives with respect to the first variable are also in $C(\R^n \by \K)$. Then, the envelope function \eqref{1.1} is Gateaux semi-differentiable everywhere on $\R^n$, and 
\beq
\label{1.2}
\left\{
\begin{split}
(\mfD f)^+_u(\phi) = \max_{k \in \K(u)} \big(\p F(\cdot, k)\big)_u(\phi),
\\
(\mfD f)^-_u(\phi) = \min_{k \in \K(u)} \big(\p F(\cdot, k)\big)_u(\phi),
\end{split}
\right.
\eeq
for all $u,\phi \in \R^n$.
}
\smallskip

\noi In the above, $\p F(\cdot, k) : \R^n \larrow (\R^n)^*$ is the partial differential with respect to the first variable, and $(\mfD f)^\pm_u(\phi)$ are the right/left semi-derivatives of $f$ at $u$ in the direction of $\phi$, namely
\[
(\mfD f)^\pm_u(\phi) := \lim_{t \to 0^\pm} \frac{f(u+t\phi)-f(u)}{t}, \ \ \ u\in\R^n.
\]
Also, $\K(u)$ is the \emph{``Argmax set" of $f$ as $u$}, namely the set of arguments realising the maximum in \eqref{1.1} (for further details on the notation see Section \ref{Section 2}):
\beq
\label{1.3}
\K(u) : = \Big\{ k\in \K : f(u)=F(u,k)\Big\}.
\eeq
Since $\K(u)$ generally is not a singleton set, the min/max in \eqref{1.2} may not coincide and therefore the right/left semi-derivatives may not match (e.g.\ in our example, $\K(0)=[-1,1]$). For the points $u$ for which they do and hence $\mH^0(\K(u))=1$, the envelope is Gateaux differentiable at this point. We note that $(\mfD f)^-_u(\phi)$ is actually superfluous and can be obtained from $(\mfD f)^+_u(\phi)$ with the direction reversal $\phi \leftrightarrow -\phi$, but we display it for emphatic purposes\footnote{Some authors, including Danskin himself, refer to the semi-derivative as just ``derivative", and the ``$+$" is not displayed even though the limit is only $1$-sided as $t\to0^+$, which adds to the confusion in the literature.}.

\smallskip

\noi {\bf The significance.} The above result, known as ``Danskin's theorem", is an all-important result in both theory and applications of mathematics. Danskin himself, who was a US Navy researcher during the Cold War, used it in min-max problems in game theory, specifically arising in defence military allocation problems, as well as in economics. Today there exist numerous other applications of Danskin's theorem in game theory, finance, economics, deterministic optimisation, stochastic optimisation, and much more, see e.g.\ \cite{BB, BR, Be, BS, Cl, Da, Da2, G, L, OT, R, TW, T, VS}. In particular, it is closely related to the so-called Envelope Theorem in the area of comparative statics. 

Due to its importance, there exist numerous extensions of Danskin's theorem towards {\it almost} every direction. Under appropriate additional assumptions, typical results of this type allow for general $F : \mX \by \K \larrow \R$, where $\mX$ is a Banach space or even a topological vector space, $\K$ is a compact or even a locally compact topological space, $u\mapsto F(u, \cdot)$ can be Fr\'echet differentiable on $\mX$ or convex, $(\p F)_u$ can be continuous or perhaps merely upper semi-continuous on $\mX \by \K$, etc (see e.g.\ the thorough historical review in \cite{BR} for precise results, and references therein). Therefore, under appropriate extra assumptions, one can still obtain that the envelope $f$, given by \eqref{1.1}, satisfies suitably interpreted versions of \eqref{1.2} (e.g.\ as sub-differentials in the convex case).

\smallskip 

\noi {\bf The limitations.} To the best of our knowledge, no extension of Danskin's theorem exists in the literature that allows to replace the maximum in \eqref{1.1} by the essential supremum with respect to some measure where $F(u,\cdot)$ is merely measurable, not even for the Lebesgue measure $\mL^m$ on $\K \sub \R^m$. In fact, it can easily be seen that such a direct extension is {\it not possible}. Firstly, it is unclear whether the measure-theoretic analogue of \eqref{1.1}, namely the \emph{generalised $\L^\infty$ envelope}
\beq
\label{1.4}
f(u) = \underset{\K}{\ess\sup} \, F(u,\cdot),
\eeq
even need be semi-differentiable on $\R^n$, but even if is somehow is, then the semi-differential cannot be expected to be given by the obvious analogue of \eqref{1.2}, which, for the right-hand side, would be
\beq
\label{1.5}
``\ (\mfD f)^+_u(\phi) =   \underset{k \in \K(u)}{\ess\sup}\,  \big(\p F(\cdot, k)\big)_u(\phi)\ ".
\eeq
Indeed, if it happens that the set $\K(u)$ arising from \eqref{1.3}-\eqref{1.4} is any $\mL^m$-nullset (for example a singleton or countable set), then the essential supremum in  \eqref{1.5} is undefined.

\smallskip 

\noi {\bf The $\L^\infty$ motivation.} Why would one care about such a measure-theoretic generalisation of Danskin's theorem? The inspiration arises from a completely different and very important field of mathematics, for which such an extension, if it existed, would have a major impact on all of its working philosophy, techniques and methodology, re-interpreting most existing results. The \emph{Calculus of Variations in $\L^\infty$} is the branch of the Calculus of Variations which is concerned with the study of variational problems for supremal functionals (also called $\L^\infty$ functionals), as well as with the connection with any extremality conditions, typically arising in the form of PDEs. This is the counterpart of the classical branch of the Calculus of Variations studying integral functionals, a classical but still extremely active field, which has its roots in the work of Euler in the 1600s. More precisely, the main object of study of the Calculus of Variations in $\L^\infty$ is the following.

\begin{definition}[Supremal functionals] \label{Definition1} \emph{Let $\Om \sub \R^n$ be a fixed open set, and for $k,n,N\in\N$, consider a Carath\'eodory function 
\[
\H \ :\ \Omega\times \Big(\mathbb R^N \times\mathbb R^{N \by n} \times \cdots  \by \mathbb R^{N \by n^{\ot k}}_{\mathrm s} \Big)  \longrightarrow \mathbb R.
\]
The $k$-th order supremal functional (arising from the ``supremand" $\H$) is given by
\beq
\label{1.6}
\left\{  \ \ \ \ 
\begin{split}
\E_\infty& \  : \ \W^{k,\infty}(\Om;\R^N) \by \mL(\Om) \larrow \R,
\\
&\E_\infty(u,\mO) := \underset{\mO}{\ess\sup}\, \H \big(\cdot, \D^{\vec{k}}u \big) .
\end{split}
\right. \ \ \ \
\eeq
Here, $\mL(\Om)$ symbolises the $\si$-algebra of Lebesgue-measurable subsets of $\Om$, and 
\[
\D^{\vec{k}}u := \big(u,\D u, \ldots, \D^k u\big)
\]
denotes the jet of derivatives up to $k$-th order of $\W^{k,\infty}$-Sobolev mappings $u : \Om\larrow \R^N$.}
\end{definition}
The Calculus of Variations in $\L^\infty$ studies variational problems for functionals of the type \eqref{1.6}, involving perhaps extra boundary conditions, pointwise, side, differential or functional constraints, etc. What is also of interest is to derive and study appropriate PDE conditions, serving as $\L^\infty$ analogues of the Euler-Lagrange equations, at which appropriately defined minimisers or ``critical points" satisfy
\[
``\ (\mfD \E_\infty)_u =0\ ".
\]

\noi {\bf The $\L^\infty$ caveats.}  As the quotation marks suggest, this is \emph{not actually possible}. Deferring a more detailed discussion on this field until Section \ref{Section 2}, we will only mention now that \emph{supremal functionals are actually non-differentiable in the Gateaux sense}, and this might happen \emph{even at absolute minimisers} (see Definition \ref{Definition21} and Example \ref{Example23}); the latter is one of the standing minimality notions in $\L^\infty$, rectifying the lack of locality in \eqref{1.6}, as it is not $\si$-additive with respect to the domain argument (Example \ref{Example20}). The failure of differentiability of \eqref{1.6} is not a matter of regularity or convexity properties of $\H$, and the situation does {\it not} improve even for $\C^\infty$ convex supremands (see Sections \ref{Section 2} and \ref{Section 8}). This is in direct contrast to the corresponding more classical case of integral functionals
\[
\begin{split}
\E \ & : \ \W^{k,p}(\Om;\R^N)  \larrow \R\ , \ \ \ \ \ \E(u) := \int_{\Om} \H \big(\cdot, \D^{\vec{k}}u \big) \, \mathrm d\mL^n,
\end{split}
\]
for $p<\infty$, where natural growth conditions and smoothness of $\H$ easily guarantee that $\E$ is Gateaux differentiable {\it everywhere}. Then, the vanishing of the Gateaux derivative
\[
(\mfD \E )_u(\phi) = \int_{\Om}\Big[ \H_{,\bfX} \big(\cdot, \D^{\vec{k}}u \big) \!: \! \D^{\vec{k}}\phi \Big] \, \mathrm d\mL^n,
\]
where $\phi \in \W^{k,p}(\Om;\R^N)$, leads to the weak formulation of the system of Euler-Lagrange PDEs, for a minimiser or a general critical point $u \in \W^{k,p}(\Om;\R^N)$. The lack of Gateaux differentiability of supremal functionals means that the {\it ``quintessence of the Calculus of Variations"} 
\[
\left.
\begin{array}{c}
\text{Global minimisers $u$:}  
\\
\E(u) \leq \E(\psi), \ \forall\, \psi
\end{array} \ 
\right\}
\overset{\text{Gateaux}}{\underset{(\Leftarrow (+) \text{\,convexity})}{\underset{\text{differentiability}}{\Leftarrow\!=\!=\!=\!=\!=\!=\!=\!\Rightarrow}}} 
\left\{\ 
\begin{array}{c}
\text{Euler-Lagrange PDE:}
\\
 (\mfD \E )_u(\psi-u) =0,\ \forall\,\psi,
\end{array}
\right.
\]
has \underline{no $\L^\infty$ counterpart}, and no such direct methodology and techniques exist for supremal functionals, even when we have smoothness and convexity of $\H$  in all variables jointly. 

\smallskip

\noi {\bf State-of-the-art in $\L^\infty$.}  Historically, the lack of differentiability of \eqref{1.6} has been partially circumvented by looking at differentiable approximations of the $\L^\infty$ functional, conventionally $\L^p$-approximations as $p\to\infty$. The hope is that by approximation one can discover limiting PDE conditions, aiming then at showing that they are actually associated to the $\L^\infty$ functional. There exist two main approaches in this regard, discussed in Section \ref{Section 2}. 

The first approach is due to Aronsson himself, the founder of the field (see \cite{A1, A2, A3, A4}), and leads to the celebrated {\it $\infty$-Laplacian} and the general {\it Aronsson equation}. These equations have attracted considerable interest by the community. In the scalar case, the theory of viscosity solutions has proven to be the appropriate framework to study them. Without any pretence of being exhaustive, we refer to \cite{AP, ACJS, AB, BEJ, BJ, BJW1, BJW2, BDP, BK, CDPP, C1, C2, CWY, DPZZ, ES, ESm, K5, KZZ, KZ, MWZ, PWZ, PZ, RZ, RZ2, Yu}. In the vectorial and higher order case, Aronsson-type systems require distinct novel approaches as these systems are generally fully nonlinear, non-elliptic and with discontinuous coefficients  \cite{AK, AyK1, AyK2, CKM2, CKP, DK, K1, K4, K6, K7, K8, K9, K10, KP, KS, PP, SS}.  The other approach, which is more recent, leads to a family of divergence equations involving measures, parameterised by the subdomains of $\Om$ (see \cite{EY} for $k=N=1$ and \cite{K1, KM1, KM2, KM3} for extensions). 

Except for very {\it special} cases (e.g.\ $N=k=1$, pure $(x,\D u(x))$ dependence, convexity and uniform coercivity of $\H$, see e.g.\ Example \ref{Example36} taken from \cite{Yu}), neither approach can lead to a general characterisation of absolute minimisers under natural level-convexity assumptions of $\H$, and these PDEs are only necessary conditions. For instance, in the vectorial case, there might even exist non-minimising smooth solutions to the model $\infty$-Laplace system \cite{KS}. Level-convexity is the natural notion of ``quasiconvexity" for supremal functionals, and is much weaker than convexity (see Section \ref{Section 2} for further details, and \cite{RZ2} for a comprehensive account). More crucially, a {\it deeper understanding of the regularity of supremal functionals and its connections to minimality still eludes us}. In particular, it remains a mystery how it is possible for absolute minimisers of supremal functionals to satisfy necessary PDE conditions, despite being non-differentiable and (in a sense) non-local.

\smallskip

\noi {\bf Some $\L^\infty$ history \& raison d'$\mathbf{\hat e}$tre.} Unlike the integral Calculus of Variations which is a classical field, the supremal counterpart is a much more modern field. Perhaps the main justification for this gap is that supremal variational problems {\it do not} arise via least action or conservation principles in Physics. Notwithstanding, except for the intrinsic mathematical interest, supremal problems are very important in various applications for which minimising the maximum energy/action/cost/error is more important than minimising the integral average. Minimising in $\L^\infty$ excludes at the outset large pointwise energy ``spikes" of small average. Examples include fracture, optimal shape design, PDE-constrained optimisation, inverse problems, data assimilation, machine learning, image processing, etc (see e.g.\ \cite{BEJ, BJ, BK, BP, CK1, CK2, CKM, ETT, GNP,  K1, K3, K3.5, SS}). However, $\L^\infty$ problems have also an intrinsic mathematical interest, as the space $\L^\infty$ is connected to many important problems. Aronsson's motivation was the problem of optimal Lipschitz extensions, whilst $\L^\infty$ problems relate also to quasi-conformal mappings \cite{CR, K6} and curvature minimisation problems connected to the Yamabe and Nirenberg problems in Geometry  \cite{MS, S}. The first order scalar theory is nowadays a very well-developed field  (see e.g.\ \cite{C1, K5}). The first vector-valued contributions arose in \cite{K10}, and the first higher order contributions arose in \cite{KP} (scalar) and in \cite{CKM2} (vectorial). The field is extremely active, but still largely under development.

\smallskip

\noi {\bf Synopsis of main results.} This work is motivated by the current incomplete understanding of the regularity of supremal functionals in the Calculus of Variations in $\L^\infty$. In the absence of any intrinsic techniques materialising an ``$\L^\infty$-quintessence of the Calculus of Variations" (which actually cannot really exist in the integral format, since Gateaux differentiability fails), the current state of affairs in $\L^\infty$ is severely constricted by the use of extrinsic ad-hoc approximation-based methodologies.

To this end, our first main result of independent interest concerns a {\it semi-differentiability result for generalised $\L^\infty$ envelopes}, which can be seen as a measure-theoretic generalisation of Danskin's theorem. The missing ingredient which hindered such an extension until now was an appropriate measure-theoretic interpretation of the notion of ``{\it maximum over an Argmax set}", which turns out to be different from the expected one shown in \eqref{1.5}. This applies to functions defined on the product of a Banach space with a measure space, and requires only natural hypotheses, and is the content of our first main result, Theorem \ref{Theorem2}. By applying this result to the Calculus of Variations in $\L^\infty$, in Theorem \ref{Theorem3}  we establish that general $k$-th order supremal functionals as in Definition \ref{Definition1} are everywhere Gateaux semi-differentiable, with an explicit formula for the semi-differential. Further, in Theorem \ref{Theorem5} we establish that for general level-convex supremal functionals, (absolute) minimality is necessary and sufficient for the satisfaction of the Partial Differential Inequality (PDI) given by the non-negativity of the semi-differential. Finally, in Theorem \ref{Theorem26} we show the necessity of the $k$-th order Aronsson equation for the satisfaction of the semi-differential PDI, for simplicity only for classical solutions and in the scalar case. Sufficiency is excluded by Example \ref{Example36}.

\smallskip

\noi {\bf Significance and impact.}  Theorem \ref{Theorem2} has potential far-reaching applications in a variety of fields which utilise Danskin's theorem and the Envelope theorem, like game theory, finance, economics, deterministic and stochastic optimisation, etc. Theorem \ref{Theorem3} solves a long-standing open problem in the Calculus of Variations in $\L^\infty$ regarding the regularity of supremal functionals. The new PDI arises from the following representation of the semi-differential (expressed here for simplicity only for mappings $u,\phi \in \C^k(\overline{\Om};\R^N)$, for $\phi \not \equiv 0$):
\[
(\mfD \E_\infty)^+_u(\phi) \,= \sup_{\Argmax \{  \H(\cdot, \D^{\vec{k}}u)\, : \, \supp(\phi) \}} \! \Big[ \H_{,\bfX} \big(\cdot, \D^{\vec{k}}u \big) \!: \! \D^{\vec{k}}\phi \Big].
\]
This result demystifies the situation regarding many of the properties of \eqref{1.6}, including for instance why certain PDEs are possible to arise by $\L^p$ approximation as $p\to\infty$, giving a new regularity assertion and new intrinsic machinery for $\L^\infty$ problems. Theorem \ref{Theorem5} allows the ``quintessence of the Calculus of Variations" to be made available in $\L^\infty$, merely by replacing differentiability by semi-differentiability and PDE by PDI, under natural assumptions, without requiring convexity, for any order and in any dimension:
\[
\left.
\begin{array}{c}
\text{Absolute minimisers $u$:}  
\\
\E_\infty(u,\mO) \!\leq \! \E_\infty(\psi,\mO), \ \forall\, \psi, \, \forall\, \mO \! \Subset \Om
\end{array} 
\right\}
\overset{\text{Gateaux semi-}}{\underset{    \underset{(\Leftarrow (+) \text{\,level-convexity})}{\text{differentiability}}   }{\Leftarrow\!=\!=\!=\!=\!=\!=\!=\!\Rightarrow}} 
\left\{
\begin{array}{c}
\text{{\it (New)} PDI:}\ \
\\
 (\mfD \E_\infty )^+_u(\psi-u) \geq 0,\ \forall\,\psi.
\end{array} 
\right. \ \
\]
Hence, our results offer a new strong intrinsic toolbox, as well as insights into a long-standing obscure situation, rendering the trademark need of extrinsic $\L^p$ approximations as $p\to\infty$, obsolete. In particular, Aronsson-type PDEs, long regarded until now the ``PDE analogues of the Euler-Lagrange equation in $\L^\infty$", are merely necessary conditions deducible from the semi-differential PDI (they are sufficient only in the scalar first order case, under convexity, coercivity and practically pure gradient dependence). The same applies to the more modern approach of parametric families of PDEs for measures arising as necessary conditions, which are also even less understood.

\smallskip

\noi {\bf The generalised measure-theoretic Danskin theorem.} Now we state our first principal result, and discuss additional results and extensions established in Section \ref{Section 3}. Let $(\mX,\|\cdot\|)$ be a Banach space with topological dual $(\mX^*,\|\cdot\|_*)$, and let $(\Om,\mathfrak L,\la)$ be a measure space. Let
\beq
\label{1.7}
\F \ : \ \mX \by \Om \larrow \R
\eeq
be a function, which satisfies that
\beq
\label{1.8}
\left\{
\begin{array}{ll}
\F(u,\cdot) \in \L^\infty(\Om,\la), &\text{for all }  u\in \X, \smallskip
\\
\F(\cdot,x) \in \C^1(\mX), &\la\text{-a.e.}\ x\in \Om,
\end{array}
\right.
\eeq
where the $\C^1$ regularity in \eqref{1.8} is meant in the standard Fr\'echet sense, and consider for $\la$-a.e.\ $x\in \Om$, the partial (Fr\'echet) differential operator
\beq
\label{1.9}
\p \F (\cdot,x) \ : \ \mX \larrow \mX^*, \ \ u\mapsto \big(\p \F (\cdot,x)\big)_u.
\eeq
Our first main result is the following.

\begin{theorem}[Semi-differentiability of the generalised $\L^\infty_\la$ envelope] \label{Theorem2} Suppose $\F$, and $\p \F (\cdot,x)$ are given by \eqref{1.7} and  \eqref{1.9}, where $(\mX,\|\cdot\|)$ is a Banach space and $(\Om,\mathfrak L,\la)$ is a measure space. Suppose $\F$ satisfies \eqref{1.8}, and additionally that 
\begin{align}
\label{1.10}
& \p \F (\cdot,x) \text{ is strongly-weakly* continuous, $\la$-essentially uniformly in}\ x\in \Om .
\\\
\label{1.11}
& \F (\cdot,x) \text{ is (strongly) continuous,  \hspace{36pt} $\la$-essentially uniformly in}\ x\in \Om .
\end{align}
Consider the $\L^\infty$-envelope of the function $\F$ with respect to $\la$-a.e.\ $x\in \Om$:
\beq
\label{1.12}
f(u) := \la \text{\!-} \underset{x\in \Om}{\ess\sup}\, \F(u,x), \ \ \ f \ : \ \mX \larrow \R.
\eeq
Then, $f$ is directionally (Gateaux) semi-differentiable everywhere on $\mX$, and the (right) semi-differential is given by
\beq
\label{1.13}
\begin{split}
(\mfD f)^+_u(\phi) = \lim_{\e \to 0^+} \! \bigg(\la \text{\!-} \underset{x\in \Om_\e(u)}{\ess\sup}\, \big(\p \F (\cdot,x)\big)_u(\phi) \bigg),
\end{split}
\eeq
for all $u,\phi \in \mX$.  In \eqref{1.13}, for any $\e>0$, $\Om_\e(u)$ is the ``approximate argmax set", given by
\beq
\label{1.14}
 \Om_\e(u) := \bigg\{ \hat x\in \Om \ : \ \F(u,\hat x) \geq   \la \text{\!-} \underset{x\in \Om}{\ess\sup}\, \F(u,x)-\e\bigg\}.
 \eeq
\end{theorem}

\begin{remark} Note that in \eqref{1.14}, $\Om_\e(u) = \big\{ \F(u,\cdot) \geq f(u)-\e\big\}$ is the $\mathfrak L$-measurable superlevel set containing all points $\e$-close to the $\la$-essential supremum.
Since the approximate argmax sets $(\Om_\e(u))_{\e>0}$ are monotone in $\e$, it follows that the limit in \eqref{1.13} is actually an infimum over $\e>0$. By juxtaposing \eqref{1.13} with \eqref{1.2}$^{+}$, one can retrospectively interpret the conclusion of Theorem \ref{Theorem2} by saying that the maximum over the argmax set in the classical theorem is replaced by an asymptotic approximate version of it, evocative of the way that ill-defined pointwise values in $\L^\infty$ are often replaced by the ``essential limsup"  \cite{K3.5} (also called ``local Lipschitz constant" if $k=N=1$, see \cite{C1}). However, the essential limsup involves shrinking families of open balls, whilst here we have the approximate argmax sets $(\Om_\e(u))_{\e>0}$, which are merely $\mathfrak L$-measurable. Unless we have continuity and compactness (in which case we reduce to the classical extensions of Danskin's theorem, see below), in general the approximate argmax sets cannot be replaced by open sets. However, due to the inclusions
\[
 \Om_\e(u) \sub \bigg\{ \hat x\in \Om \ : \ \F(u,\hat x) >  \la \text{\hspace{-1pt}-} \underset{x\in \Om}{\ess\sup}\, \F(u,x)-2\e\bigg\} \sub \Om_{2\e}(u),
 \]
the sets $(\Om_\e(u))_{\e>0}$ can be replaced in \eqref{1.13} by the version in which the inequality appearing in \eqref{1.14} is strict. In addition to the conclusion of Theorem \ref{Theorem2} above, we also have precise a priori estimates on the difference quotients, see Corollary \ref{Corollary19A}. 
\end{remark}

\begin{remark}[On the assumptions] (i) The precise meaning of assumptions \eqref{1.10} and \eqref{1.11} is respectively (see also \eqref{1.15}-\eqref{1.16} that follow)
\[
\begin{split}
\label{1.15}
\forall\,\phi \in \mX,\ \ &\la \text{\hspace{-1pt}-} \underset{x\in \Om}{\ess\sup}\, \Big| \Big[\big(\p\F(\cdot,x)\big)_w 
\! - \! \big(\p\F(\cdot,x)\big)_v\Big](\phi)\Big| \larrow 0, \ \text{ as }w\to v \text{ in }\mX,
\\
& \la \text{\hspace{-1pt}-} \underset{x\in \Om}{\ess\sup}\, \Big| \F(w,x)  - \F(v,x) \Big| \larrow 0, \hspace{68pt} \text{ as }w\to v \text{ in }\mX.
\end{split}
\]
(ii) It appears that assumption \eqref{1.8} on the continuous Fr\'echet differentiability of $\F(\cdot,x)$ for $\la$-a.e.\ $x\in \Om$ could be relaxed to a slightly weaker assumption involving its continuous Gateaux differentiability, but this would create additional difficulties in the proof, as the relevant integrals appearing would have to be interpreted in the weaker Gelfand-Pettis sense (see e.g.\ \cite{Ta, V}). Since supremal functionals always satisfy \eqref{1.8} under minimal hypotheses on the supremand (see Theorem \ref{Theorem3}), we will not pursue this line of generalisation.
\end{remark}

Note that a change of variables gives the identity $(\mfD f)_u^- = - (\mfD f)_u^+(-\, \cdot)$,
which implies that \eqref{1.13} is equivalent to the following expression for the left semi-differential:
\[
(\mfD f)^-_u(\phi) = \lim_{\e \to 0^+} \!\bigg(\la \text{-} \underset{x\in \Om_\e(u)}{\ess\inf}\, \big(\p \F (\cdot,x)\big)_u(\phi) \bigg).
\]
We will therefore refrain from discussing $(\mfD f)_u^-$, unless explicitly required. If in addition to \eqref{1.11} $\F$ is locally Lipschitz continuous, $\la$-essentially uniformly in $x\in \Om$, in Proposition \ref{Proposition11} it is shown that $f$ is also \emph{Hadamard semi-differentiable and Penot semi-differentiable} (i.e.\ M-semi-differentiable), with the same representation \eqref{1.13} for the stronger semi-differential. This makes the general Hadamard semi-differential Calculus (see e.g.\ \cite{De1, De2, De3}) available for generalised $\L^\infty$ envelopes. In particular, for any fixed $u\in\mX$, $(\mfD f)^+_u : \mX \larrow \R$ is a positively $1$-homogeneous continuous convex functional, the chain rule holds true, and $f\mapsto (\mfD f)^+_u (\phi)$ is a linear functional (Remark \ref{Remark12}). 

By comparing \eqref{1.13}-\eqref{1.14} with their classical pointwise counterparts  \eqref{1.2}$^+$-\eqref{1.3}, the natural question of independent interest arises whether \eqref{1.13} can be recast in a form that resembles \eqref{1.2}$^+$, by interpreting it as a ``measure-theoretic maximum over the argmax set". As we show in Section \ref{Section 4}, this is indeed possible. By utilising the new measure-theoretic concepts of \emph{approximate supremum} ``ap-sup" and of the \emph{measure-theoretic argmax set} ``$\la$-Argmax", introduced in Section \ref{Section 4} (specifically Definitions \ref{Definition13}, \ref{Definition15} and \ref{Definition18}), we can reformulate \eqref{1.13} in a more familiar fashion, that corresponds directly to classical pointwise Danskin-type theorems (Corollary \ref{Corollary19}):
\beq
\label{1.17}
(\mfD f)^+_u(\phi) = \underset{ \la \text{-} \Argmax\{\F(u,\cdot) \,:\, \Om \}}{\textrm{ap-sup}}\, \big(\p \F (\cdot,x)\big)_u(\phi). 
\eeq
However, the elegant formulation \eqref{1.17} requires an additional structure on the measure space $(\Om,\mathfrak L, \la)$, which needs to be a Borel regular metric measure space, satisfying and differentiation property (namely every locally integrable function satisfies the Lebesgue differentiation theorem), and also that $\supp (\la)=\Om$ (namely $\la$ assigns positive values to every non-empty open set in $\Om$, see Sections \ref{Section 2} and \ref{Section 4} for details). This additional required structure has no bearing on applications to the Calculus of Variations in $\L^\infty$, and hence the formulation \eqref{1.17} is always available. As we show in Section \ref{Section 4} (Proposition \ref{Proposition16}), the generalised notions of approximate supremum and essential argmax set reduce to the standard pointwise concepts if we have continuity for $\F,\p\F$ and compactness for $\Om$. In that case, \eqref{1.17} reduces to the conclusion of the classical pointwise (infinite-dimensional) Danskin theorem.

\smallskip

\noi {\bf Semi-differentiability of supremal functionals.} Our second principal result concerns the application of Theorem \ref{Theorem2} to the regularity question in the Calculus of Variations in $\L^\infty$, which motivated this work in the first place. For brevity, let us symbolise
\[
\R^{\vec M}:= \mathbb R^N \times\mathbb R^{N \by n} \times \cdots  \by \mathbb R^{N \by n^{\ot k}}_{\mathrm s}
\]
when $n,k,N\in\N$, and recall also the jet notation $\D^{\vec{k}}u = \big(u,\D u, \ldots, \D^k u\big)$.

\begin{theorem}[Semi-differentiability of supremal functionals] \label{Theorem3} Suppose that $\Om \sub \R^n$ is an open set, and let $\H : \Om \times \R^{\vec M} \longrightarrow \R$ be a Carath\'eodory function, where $n,k,N\in\N$. Consider the supremal functional $\E_\infty : \W^{k,\infty}(\Om;\R^N) \by \mL(\Om) \larrow \R$ given by \eqref{1.6}, namely
\[
\E_\infty(u,\mO) = \underset{\mO}{\ess\sup}\, \H \big(\cdot, \D^{\vec{k}}u \big) .
\]
Suppose that $\H(x,\cdot)$ is in $\C^1\big( \R^{\vec M} \big)$ essentially uniformly in $x\in\Om$, namely for any $R>0$ there exists an increasing modulus of continuity $\om_R \in \C(0,\infty)$ with $\om_R(0^+)=0$ such that
\beq
\label{1.18}
\left\{ \ \
\begin{array}{r}
\underset{\Om}{\ess\sup} \Big| \H(\cdot,\bfX'') -  \H(\cdot,\bfX')\Big| \leq \om_R\big( |\bfX''-\bfX'|\big) , \smallskip
\\ 
\underset{\Om}{\ess\sup} \Big| \H_{,\bfX}(\cdot,\bfX'') -  \H_{,\bfX}(\cdot,\bfX')\Big| \leq \om_R \big( |\bfX''-\bfX' |\big) ,
\end{array}
\right.
\eeq
for all $ \bfX'',\bfX' \in \mB_R \sub \R^{\vec M}$. We also suppose that
\begin{align}
\label{1.19}
\big|  \H(\cdot , \mathbf 0) \big| \in \L^\infty(\Om).
\end{align}
Then, $\E_\infty(\cdot,\Om)$ is (Gateaux) semi-differentiable everywhere on $\W^{k,\infty}(\Om;\R^N)$, and its (right) semi-differential can be represented as
\beq
\label{1.20}
\big(\mfD \E_\infty(\cdot,\Om)\big)^+_u(\phi) = \lim_{\e \to 0^+}  \! \Bigg( \underset{  \big\{ \H(\cdot, \D^{\vec{k}}u) \geq \E_\infty(u,\Om)-\e \big\} }{\ess\sup}\, \Big[
\H_{,\bfX} \big(\cdot, \D^{\vec{k}}u \big) \!: \! \D^{\vec{k}}\phi \Big] \Bigg),
\eeq
for any $u,\phi \in \W^{k,\infty}(\Om;\R^N)$.
\end{theorem}
The assumptions \eqref{1.18}-\eqref{1.19} are satisfied in all cases of interest (for example when $\H$ and its derivative $\H_{,\bfX}$ are continuous on $\smash{\overline{\Om} \by \R^{\vec M}}$). Note we could have perhaps required that $\big|  \H(\cdot , \mathbf X_0) \big| \in \L^\infty(\Om)$ for \emph{some} $\smash{\bfX_0 \in  \R^{\vec M}}$  instead of \eqref{1.19}, but this is only illusively more general, as \eqref{1.18}-\eqref{1.19} together imply $\big|  \H(\cdot , \bfX_0) \big| \in \L^\infty(\Om)$ for all $\smash{\bfX_0 \in \R^{\vec M}}$. 

If we additionally assume that $|\H_{,\bfX}(\cdot,\mathbf 0)| \in \L^\infty(\Om)$, in Proposition \ref{Proposition22} we show that  $\E_\infty(\cdot,\Om)$ is also Hadamard semi-differentiable and Penot semi-differentiable, with the corresponding stronger semi-differentials represented again by \eqref{1.20}. Here again this extra assumption is equivalent to the requirement $|\H_{,\bfX}(\cdot,\bfX_0)| \in \L^\infty(\Om)$ for \emph{some} point $\smash{\bfX_0 \in \R^{\vec M}}$. Further, in view of the measure-theoretic results of Section \ref{Section 4}, \eqref{1.20} can be concisely written as
\beq
\label{1.21}
\big(\mfD \E_\infty(\cdot,\Om)\big)^+_u(\phi) = \underset{\mL^n\text{-}\Argmax \{  \H(\cdot, \D^{\vec{k}}u)\, : \, \Om \}}{\mathrm{ap}\text{-}\!\sup} \! \Big[ \H_{,\bfX} \big(\cdot, \D^{\vec{k}}u \big) \!: \! \D^{\vec{k}}\phi \Big],
\eeq
for any $u,\phi \in \W^{k,\infty}(\Om;\R^N)$. 

Theorem \ref{Theorem2} is clearly valid on any open subset $\mO\subseteq \Om$ as well, therefore $\E_\infty(\cdot,\mO)$ is semi-differentiable everywhere on $\W^{k,\infty}(\mO;\R^N)$, with corresponding expressions for $(\mfD \E_\infty(\cdot,\mO))^+$ being given by \eqref{1.20}-\eqref{1.21} with $\mO$ in the place of $\Om$. That being said, a trademark hindrance of the Calculus of Variations in $\L^\infty$, mentioned earlier briefly but discussed in more detail in Section \ref{Section 2}, is that the standing notion of minimality for \eqref{1.6} involves variations on all subdomains $\mO\subseteq \Om$, or at least on compactly contained $\mO\Subset \Om$. The reason is that $\E_\infty(u,\cdot)$ is not an outer measure in any reasonable $\si$-sub-algebra of $\mL(\Om)$. 

In the light of the above comments, it appears that to characterise absolute minimality, one would need to consider a family of semi-differential operators parameterised by the open subdomains of $\Om$, namely $\smash{\big\{(\mfD \E_\infty(\cdot,\mO))^+ : \mO\subseteq \Om \big\}}$, rather than a single semi-differential operator ``$(\mfD \E_\infty)^+$" (this parameterisation appears for instance in the case in the derivation of PDEs involving measures via $\L^p$ as $p\to \infty$, see Section \ref{Section 3}). However, contrary to first appearances, this is not necessary. As shown in Proposition \ref{Proposition20} (Section \ref{Section 2}), absolute minimisers can be characterised by considering arbitrary variations $\phi \in \W^{k,\infty}_0(\Om;\R^N)$, as long as one selects as subdomains \emph{only} the sets $\{\phi \neq 0\}$ and compares energies for $\E_\infty(\cdot, \{\phi \neq 0\})$, instead of parameterising the notion by taking all $\phi \in \W^{k,\infty}_0(\mO;\R^N)$ for all open $\mO\subseteq \Om$ (Definition \ref{Definition21}). Namely, we can {\it discard compactly contained variations in the notion of absolute minimality}. This property does not seem to have previously appeared in the literature (but see the recent contribution \cite{CKM1}, in which we study the peculiar features of compactly supported variations). Anyhow, it turns out to be an exceptionally useful feature of supremal functionals, leading to the following subdomain-free concept of semi-differential:
\begin{definition}[Total semi-differential of supremal functionals]
Consider the supremal functional \eqref{1.6}. We define its (total) semi-differential operator
\[
\ \ \ \ \ \ \ (\mfD \E_\infty)^+ \ :\   \W^{k,\infty}(\Om;\R^N) \larrow \R^{ \W^{k,\infty}_0(\Om;\R^N)},  \ \ \ u\mapsto (\mfD \E_\infty)^+_u,
\]
by setting
\beq
\label{1.22}
\ \ \  (\mfD \E_\infty)^+_u(\phi)  := \left\{
 \begin{array}{ll}
\!\! \big(\mfD \E_\infty(\cdot,\{\phi\neq0\}\big)^+_u(\phi), & \phi \neq 0,
 \\ 
 \!\!0, & \phi =0, \phantom{a^\big[}
 \end{array}
 \right.
\eeq
for $\phi \in \W^{k,\infty}_0(\Om;\R^N)$.
\end{definition}
The right hand side of $(\mfD \E_\infty)^+$ is well-defined by Theorem \ref{Theorem3}. By utilising Proposition \ref{Proposition20}, the (total) semi-differential allows to recast the conclusion of Theorem \ref{Theorem3} as
\beq
\label{1.23}
(\mfD \E_\infty)^+_u(\phi) = \underset{\mL^n\text{-}\Argmax \{  \H(\cdot, \D^{\vec{k}}u)\, : \, \{\phi\neq0 \}\}}{\mathrm{ap}\text{-\!}\sup} \! \Big[ \H_{,\bfX} \big(\cdot, \D^{\vec{k}}u \big) \!: \! \D^{\vec{k}}\phi \Big].
\eeq
for any $u \in \W^{k,\infty}(\Om;\R^N)$ and arbitrary variations $\phi \in \W^{k,\infty}_0(\Om;\R^N)^\varobslash$, because \eqref{1.20}-\eqref{1.21} actually hold on any open $\mO\subseteq \Om$. Here we have used the notation
\[
\mathfrak W^\varobslash := \mathfrak W \set\{0\},
\]
when $\mathfrak W$ is any vector space. The above leads to the single semi-differential operator \eqref{1.22} being a candidate to characterise variationally (absolute) minimisers, without needing a family of operators. In that case, the methodology and working philosophy of variational principles for integral functionals would carry over to the supremal case.

\smallskip

\noi {\bf Characterisation of minimality in $\L^\infty$.} Our final principal result completes the picture, by showing that indeed non-negativity of the (total) semi-differential operator characterises absolute minimisers. Even though the necessity direction requires no assumptions additional to \eqref{1.18}-\eqref{1.19} imposed in Theorem \ref{Theorem3}, unsurprisingly, the  sufficiency direction requires stronger hypotheses. Specifically, $\H$ must be level-convex, namely $\H(x,\cdot)$ must have convex sublevel sets, and also it must be non-degenerate away from any of its global minima in $\smash{\R^{\vec M}}$. These assumptions are weaker than those required in the integral case for the sufficiency of the Euler-Lagrange equations, and weaker also than the assumptions for the Aronsson equation to be sufficient if $N=k=1$, which require convexity and uniform coercivity in $\D u$, and at most dependence on $x$, see Example \ref{Example36} and \cite{Yu}. If however $N\geq 2$, then not even the $\infty$-Laplace model PDE system suffices for minimality, see \cite{KS} and Section \ref{Section 2}.

\begin{theorem}[Variational characterisation of absolute minimisers via semi-differential PDI]  \label{Theorem5} In the context of Theorem \ref{Theorem3} and under the same assumptions \eqref{1.18}-\eqref{1.19}, consider the following statements concerning a fixed $u \in \W^{k,\infty}(\Om;\R^N)$:
\begin{itemize}

\item[\emph{(i)}] The mapping $u$ is an absolute minimiser of the supremal functional $\E_\infty$ given by  \eqref{1.6}, namely
\[
\ \ \ \ \ \ \E_\infty(u,\mO)\leq\E_\infty(u+\psi,\mO), \text{ for all }\mO \subseteq \Om \text{ open and all } \psi \in \W^{k,\infty}_0(\mO;\R^N).
\]
\item[\emph{(ii)}] The total semi-differential of \eqref{1.6}, expressed by \eqref{1.22}, is non-negative when evaluated at $u$: 
\[
\ \ \ \ \ \ (\mfD \E_\infty)^+_u\geq 0.
\]
By \eqref{1.23}, this means
\[
\ \ \ \ \ \ \underset{\mL^n\text{-}\Argmax \{  \H(\cdot, \D^{\vec{k}}u)\, : \, \{\phi\neq0 \}\}}{\mathrm{ap}\text{-\!}\sup} \! \Big[ \H_{,\bfX} \big(\cdot, \D^{\vec{k}}u \big) \!: \! \D^{\vec{k}}\phi \Big] \geq 0,
\]
for any $\phi \in \W^{k,\infty}_0(\Om;\R^N)^\varobslash$.
\end{itemize}
Then, we have that \emph{(i)} implies \emph{(ii)}. Conversely, \emph{(ii)} implies \emph{(i)} if in addition the supremand $\H$ is non-degenerate and level-convex, namely:
\begin{align}
\text{For a.e.\ $x\in \Om$, $\H_{,\bfX}(x,\cdot)\neq \mathbf 0$ on $\R^{\vec M} \set  \Argmin \big\{\H (x,\cdot) : \R^{\vec M} \big\}$}.
\label{1.24}
\\
\text{For a.e.\ $x\in \Om$ and any $t\in\R$, $\big\{ \H(x,\cdot) < t\big\}$ is convex in $\R^{\vec M} $.}
\label{1.25}
\end{align}
\end{theorem}
In the above result, that hypothesis \eqref{1.24} needs to be assumed separately, as it does not follow from the level-convexity requirement (this is not the case for convexity, which implies \eqref{1.24} under $\C^1$ regularity). Further, no assumption on the argmin set of $\H (x,\cdot)$ is imposed, nor any coercivity is assumed, hence it may well be either empty or very large. Further, due to the monotonicity of the approximate argmax sets in $\e>0$ (recall \eqref{1.13}-\eqref{1.14} and \eqref{1.20}-\eqref{1.22}), we may rewrite the inequality in Theorem \ref{Theorem5}(ii) equivalently as
\beq
\label{1.26}
\underset{\big\{ \! \H(\cdot, \D^{\vec{k}}u) \geq \E_\infty(u,\{\phi\neq0 \}) -\e \! \big\} \cap \{\phi\neq0 \}}{\ess\sup}  \Big[ \H_{,\bfX} \big(\cdot, \D^{\vec{k}}u \big) \!: \! \D^{\vec{k}}\phi \Big] \geq 0,
\eeq
for any $\phi \in \W^{k,\infty}_0(\Om;\R^N)^\varobslash$ and $\e>0$. This reformulation might be more useful for technical applications of semi-differentiability in $\L^\infty$. In the case that $\H(x,\cdot)$ happens to be convex in $\smash{\R^{\vec M}}$, the proof of sufficiency in Theorem \ref{Theorem5} is much simpler, so we provide after the main proof in Section \ref{Section 7} a second direct proof under the convexity assumption. 

The characterisation of Theorem \ref{Theorem5} goes beyond the class of absolute minimisers. In Theorem \ref{TheoremXX} we state a corresponding version asserting that minimality of $u$ for $\E_\infty(\cdot,\mO)$ with respect to any closed admissible class $\mathfrak C \sub \W^{k,\infty}(\mO;\R^N)$ is equivalent to ${(\mfD \E_\infty(\cdot,\mO))^+_u\geq 0}$ on $\mathfrak C$, for any $\mO \sub \Om$. This version is potentially useful for the study of constrained and other related problems.

\smallskip

\noi {\bf Relation to the Aronsson equations.} As already discussed, our results imply that the non-negativity of the semi-differential is the appropriate differential characterisation of absolute minimality for supremal functionals, with Aronsson-type (systems of) PDE being a necessary only condition, except for very special cases. For the sake of illustration, in Section \ref{Section 8} we provide a direct proof that the semi-differential PDI implies the satisfaction of an Aronsson system of PDEs (Theorem \ref{Theorem26}). For simplicity, and in order to avoid technical nuances arising from concepts of generalised solutions, we establish this in the smooth case only. Particular instances of this result in the first and second order (scalar and vectorial) case have previously appeared in the works \cite{DK, K10, KP}. Their proofs were more technical; the proof presented here is streamlined. We refrain from discussing the vector-valued case, which presents several additional difficulties, see Section \ref{Section 2}.

As shown in Example \ref{Example36}, in general the Aronsson equation cannot characterise minimality in $\L^\infty$, and in the vectorial case as mentioned not even the $\infty$-Laplacian can (\cite{KS}). However, in the special cases in which this is indeed true (which happens for example if $\H$ is $\C^2$, depends only on $(x,\D u)$, is convex in the $\D u$-variable and uniformly coercive, see \cite[Theorem 1, p.\ 155]{Yu}), the \emph{semi-differential PDI provides a new characterisation of viscosity solutions to the Aronsson equation}. We state this as a corollary in the case of the $\infty$-Laplace PDE.

\begin{corollary}[Characterisation of $\infty$-Harmonic functions via the $\L^\infty$ semi-differential PDI] Let $\Om \sub \R^n$ be open, where $n\in\N$. Then, a function $u\in \W^{1,\infty}(\Om)$ is $\infty$-Harmonic, namely it is a viscosity solution to the $\infty$-Laplacian
\[
\De_\infty u : = \D u \ot \D u \!:\! \D^2u = 0, \ \ \text{ in }\Om,
\]
if and only if it satisfies
\[
\underset{ \big\{ |\D u| \geq \| \D u \|_{\L^\infty(\{\phi\neq0 \})} -\e \! \big\} \cap \{\phi\neq0 \}}{\ess\sup}  \big[  \D u  \cdot \D \phi \big] \geq 0,
\]
for any $\phi \in \W^{1,\infty}_0(\Om)^\varobslash$ and all $\e>0$. 
\end{corollary}

The above characterisation could potentially be useful for regularity considerations, which is perhaps the main open problem in the scalar case (specifically, whether $\infty$-Harmonic functions are $\C^1$ for $n\geq 3$; it is known for $n=2$, see \cite{ES, ESm}). 


\ms

\section{Preliminaries and the Calculus of Variations in $\L^\infty$}
\label{Section 2}

\noi {\bf Generalities.} Our general measure theory, functional space and PDE notation is either self-explanatory, or a convex combination of standard symbolisations (e.g.\ in \cite{D, EG, KV}). We nonetheless explicitly discuss some of our conventions for the benefit of the reader. Let $n,k,N \in \mathbb N$, and let us also fix an open domain $\Omega \subseteq \mathbb R^n$. Given a mapping $u : \R^n \supseteq \Om \larrow \R^N$, the derivatives (of first, second, and $k$-th order) will be denoted by
\[
\left\{ \ \ \   \begin{split}
    \D u=\big(\D_i u_\al\big)_{i\in\{1,...,n\}}^{\al\in\{1,...,N\}}&\ :\ \Omega\larrow \R^{N \by n},\\
    \D^2u=\left(\D^2_{ij}u_\al\right)_{i,j\in\{1,...,n\}}^{\al\in\{1,...,N\}}&\ :\ \Omega \larrow  \R^{N \by n^{\ot 2}}_{\mathrm s},\\
  \vdots  \ \ \ \ \ \   \ \ \ \ \ \ & \ \ \ \ \ \ \  \ \ \ \vdots
     \\
    \D^k u=\left(\D^k_{i_1...i_k} u_\al\right)_{i_1,\ldots,i_k \in\{1,...,n\}}^{\al\in\{1,...,N\}}&\ : \ \Omega \larrow  \R^{N \by n^{\ot k}}_{\mathrm s},\ \ \ \ \ \ \ \ \ 
  \end{split}
  \right.
\]
and they are valued into their respective (symmetric) tensor spaces, defined as 
\[
  \R^{N \by n^{\ot k}}_{\mathrm s}\! :=\!\bigg\{ \! {\bf X}\in \R^N \! \ot \underbrace{\R^n\otimes\cdots\otimes\R^n\!\!}_{k \text{ times}}\ : {\bf X}_{\al i_1\cdots i_k}= {\bf X}_{\al \sigma(i_1\cdots i_k)}, \sigma\text{ permutation on }(i_1,\ldots,i_k)\! \bigg\}.
\]
The notation ``$(\cdot)^{\!\ot k}$" will be used $k$-fold symmetric tensor product in $\R^n$: 
\[
a^{\!\ot k}:= \underset{k\text{-times}}{\underbrace{a\ot \cdots \ot a}}\, \in \R^{n^k}_{\mathrm s}.
\]
For brevity, we will use the following compact notation for the ``$k$-th order jet" of $u$:
\[
\D^{\vec{k}}u := \big(u, \D u, \D^2u,...,\D^k u \big) \ : \ \Om \larrow  \mathbb R^N \times\mathbb R^{N \by n} \times \cdots  \by \mathbb R^{N \by n^{\ot k}}_{\mathrm s} . 
\]
We also set
\[
\R^{\vec M} :=  \mathbb R^N \times\mathbb R^{N \by n} \times \cdots  \by \mathbb R^{N \by n^{\ot k}}_{\mathrm s},
\]
and symbolise elements in the space $\R^{\vec M}$ as
\[
 \bfX \equiv \big(\bfX_0, \bfX_1,\ldots, \bfX_k \big), \ \ \ \bfX_j \in \mathbb R^{N \by n^{\ot j}}_{\mathrm s},\ j=0,\ldots,k.
\]
Given a function $\H :\Omega\times  \R^{\vec M}  \longrightarrow \mathbb R$, we will symbolise the arguments of $\H$ throughout this work as $\H(x, \textbf X)$, and the derivative with respect to the second variable will be denoted by $\H_{,\bfX}(x, \textbf X)$. Such a function $\H$ is called Carath\'eodory if $\H(\cdot, \textbf X)$ is Lebesgue measurable for all $\bfX \in \smash{\R^{\vec M}}$ and $\H(x, \cdot)$ is continuous for a.e.\ $x\in\Om$. The Euclidean (Frobenius) inner products on $\smash{\R^{N \by n^{\ot k}}_{\mathrm s}}$ and on $\smash{\R^{\vec M}} $ will be denoted by $\bfX_k\! :\!  \mathbf Y_k$ and $\bfX \! :\! \mathbf Y$ respectively,  whilst the Euclidean norm will be denoted by $| \cdot |$ in all dimensions. The $n$-dimensional Lebesgue and  the $s$-dimensional Hausdorff measure on $\R^n$ will be denoted by $\mL^n$ and $\mH^s$ respectively. Open balls of radius $\rho$ and centre $x$ in a metric space $(\Om, d)$ will be denoted by $\mB_\rho(x)$. If the metric space is also a vector space, we will write $\mB_\rho(0)\equiv \mB_\rho$. The disjoint union of sets will be denoted by ``$\biguplus$". Further, for any set $X\neq\emptyset$, we will symbolise the set of all function $f : X\larrow \R$ by
\[
\R^X := \big\{ f \, |\, f : X \larrow \R \big\}.
\]
Finally, for any vector space $\mathfrak W$, we will use the notation 
\[
\mathfrak W^\varobslash := \mathfrak W \set\{0\}
\]
to symbolise the punctured space with the origin deleted.
\smallskip

\noi {\bf The spaces $\W^{k,\infty}(\Om;\R^N)$ and $\W^{k,\infty}_0(\Om;\R^N)$.} In this work we will use systematically the $k$-th order Sobolev spaces $\W^{k,\infty}(\Om;\R^N)$ and $\W^{k,\infty}_0(\Om;\R^N)$, where the domain $\Om \sub \R^n$ is \emph{not} assumed to be bounded. Since these spaces are perhaps not very common, we discuss now some of their properties. $\W^{k,\infty}(\Om;\R^N)$ is a dual Banach space which can be isometrically embedded into a cartesian product of copies of $\L^\infty(\Om)$ via the map $\smash{u \mapsto  \D^{\vec k}u}$. For convergence purposes, we will specify the norm
\[
\| u \|_{\W^{k,\infty}(\Om;\R^N)} := \big\| \D^{\vec k}u \big\|_{\L^\infty(\Om)}.
\]
$\W^{k,\infty}_0(\Om;\R^N)$ is defined as the closure in the weak*-$\W^{k,\infty}$ topology of the subspace $\C^{\infty}_\c(\Om;\R^N)$  $\sub \W^{k,\infty}(\Om;\R^N)$. This means $ u \in \W^{k,\infty}_0(\Om;\R^N)$ if and only if there exists a sequence $(u_j)_1^\infty \sub \C^{\infty}_\c(\Om;\R^N)$ such that $u_j \larrow u$ in $\W^{k-1,\infty}(\Om;\R^N)$ and also $\D^k u_j \weakstar \D^k u$ in $\L^\infty(\Om;\R^{N\by n^{\ot k}}_{\mathrm s})$, as $j\to\infty$. By Morrey's theorem, we have the inclusion  $\W^{k,\infty}(\Om;\R^N) \sub \C^{k-1}(\Om;\R^N)$, thus the $(k-1)$-order jet
\[
\D^{\vv{k-1\,}}u =\big(u,\D u, \ldots, \D^{k-1}u \big) 
\]
is continuous on $\Om$, when $u\in \W^{k,\infty}(\Om;\R^N)$. We also have the inclusion of spaces
\beq
\label{2.0}
\W^{k,\infty}_0(\Om;\R^N) \sub \C^{k-1}_0(\Om;\R^N), 
\eeq
where, for any $\ell \in \N$, 
\[
\C^{\ell}_0(\Om;\R^N) := \Big\{ u\in \C^{\ell}(\Om;\R^N) \, : \, \forall\, \e>0,\ \big\{ |\D^{\vec \ell}u | \geq \e \big\}\text{ is compact in }\Om\Big\}.
\]
The inclusion \eqref{2.0} can be seen as follows: fix  $ u \in \W^{k,\infty}_0(\Om;\R^N)$, and let $(u_j)_1^\infty \sub \C^{\infty}_\c(\Om;\R^N)$ be such that $u_j \larrow u$ in $\W^{k-1,\infty}(\Om;\R^N)$ as $j\to\infty$. For any $\e>0$, we choose $j\geq j(\e)$ such that
\[
\e > \big\| \D^{\vv{k-1\,}}u_j -\D^{\vv{k-1\,}}u \big\|_{\L^\infty(\Om)}  \geq \big\| \D^{\vv{k-1\,}}u \big\|_{\L^\infty(\Om \set \supp(u_j))} = \sup_{\Om \set \supp(u_j))} \big| \D^{\vv{k-1\,}}u \big|.
\]
(For the last equality above we used Remark \ref{Remark17}.) This implies that 
\[
\Om \set \supp(u_j) \sub \big\{ \big| \D^{\vv{k-1\,}}u \big| < \e \big\}, 
\]
which is equivalent to $\big\{ \big| \D^{\vv{k-1\,}}u \big| \geq \e \big\} \sub \supp(u_j),$ therefore establishing that $u\in \C^{k-1}_0(\Om;\R^N)$. It is easy to see that for any point $x\in\p\Om$, the $(k-1)$-order jet has a vanishing Lebesgue value at this point. Indeed, fix $x\in\p\Om$, $\e>0$, and choose $j\geq j(\e)$ as above. Since $ \supp(u_j)$ is compact in $\Om$, we can find $\rho=\rho(\e)>0$ small such that $\mB_\rho(x) \sub \Om \set \supp(u_j) $. This yields
\[
\av_{\mB_\rho(x)\cap \Om}  \big| \D^{\vv{k-1\,}}u \big| \, \mathrm d \mL^n \leq \big\| \D^{\vv{k-1\,}}u \big\|_{\L^\infty(\mB_\rho(x)\cap \Om)} <\e.
\]
In particular, the extension by zero on $\R^n\set \Om$ provides a canonical inclusion $ \W^{k,\infty}_0(\Om;\R^N) \sub \W^{k,\infty}_0(\R^n;\R^N)$. Let us finally record for subsequent use the following useful property.
\begin{claim}
 \label{Claim9} 
Fix $\Om \sub \R^n$ open and $\psi \in \W^{k,\infty}(\Om;\R^N)^\varobslash$. Then, we have that $\psi \in \W^{k,\infty}_0(\Om;\R^N)$ if and only if $\psi \in \W^{k,\infty}_0(\{\psi \neq0\};\R^N)$. In either case, we have $\D^{\vec k} \psi =0$ a.e.\ on $\Om \set \{\psi \neq 0\}$.
\end{claim}

\noi {\bf Proof of Claim} \ref{Claim9}. Fix any $\psi \in \W^{k,\infty}_0(\Om;\R^N)^\varobslash$. Then, $\{\psi \neq 0\}$ is a non-empty open subset of $\Om$. Extending $\psi$ by zero from $\Om$ to $\R^n$, we have $\psi \in \W^{k,\infty}_0(\R^n;\R^N)$, and that $\psi \equiv 0$ on $\R^n \set \{\psi \neq 0\}$. By Sobolev space theory, we have $\D^{\vec k} \psi =0$ a.e.\ on $\R^n \set \{\psi \neq 0\}$. Morrey's theorem implies that the $(k-1)$-order jet of $\psi$ is continuous on $\R^n$, and therefore
\[
\D^{\vv{k-1\,}}\psi  \equiv 0 \ \text{ on } \ \R^n \set \{\psi \neq 0\}.
\]
The above implies that $\psi \in \W^{k,\infty}_0(\{\psi \neq0\};\R^N)$. The reverse implication is obvious. \qed

\ms

\noi {\bf Semi-differentiability.} Let now $\mX$ be a topological vector space, and $f : \mX \larrow \R$ a given functional. The upper/lower/left/right Gateaux semi-derivatives of $f$ at $u\in \mX$ in the direction of $\phi \in \mX$ (namely the four Dini numbers), are defined by
\[
\begin{split}
(\underline{\mfD} f)^\pm_u(\phi) := \liminf_{t \to 0^\pm} \frac{f(u+t\phi)-f(u)}{t}, \ \ \ (\underline{\mfD} f)^\pm_u \ : \ \mX \larrow [-\infty,+\infty],
\\
(\overline{\mfD}  f)^\pm_u(\phi) := \limsup_{t \to 0^\pm} \frac{f(u+t\phi)-f(u)}{t}, \ \ \ (\overline{\mfD} f)^\pm_u \ : \ \mX \larrow [-\infty,+\infty].
\end{split}
\]
Clearly, the functionals $(\underline{\mfD} f)^+,(\overline{\mfD} f)^+,(\underline{\mfD} f)^-,(\overline{\mfD} f)^-$ exist always, but may not coincide. If  $(\underline{\mfD} f)_u^+=(\overline{\mfD} f)_u^+$, we will denote the common value by  $({\mfD} f)_u^+$ and call it the right semi-derivative (or semi-differential) at $u$. Similarly, if  $(\underline{\mfD} f)_u^-=(\overline{\mfD} f)_u^-$, we will denote the common value by  $({\mfD} f)_u^-$ and call it the left semi-derivative (or semi-differential) at $u$. If further $({\mfD} f)_u^-=({\mfD} f)_u^+$, then $f$ is Gateaux differentiable at $u$. In general, we only have the identity $({\mfD} f)_u^-=-({\mfD} f)_u^+(- \, \cdot)$. Let us also recall for later use two related but stronger notions of semi-differential from the literature (see e.g.\ \cite{De1, De2, De3}). $f$ is called (right) {\it Penot-semi-differentiable (also known as $M$-semi-differentiable) at $u$ in the direction $\phi$} if
\beq
\label{2.1}
({\mfD} f)^+_u(\phi) = \underset{\psi \to \phi \text{ in }\mX}{\lim_{t \to 0^+}} \frac{f(u+t\psi)-f(u)}{t}.
\eeq
$f$ is called (right) {\it Hadamard-semi-differentiable at $u$ in the direction $\phi$} if
\beq
\label{2.2}
({\mfD} f)^+_u(\phi) = \lim_{t \to 0^+} \frac{f(u(t))-f(u(0))}{t},
\eeq
for any path $u : [0,t_0) \larrow X$, satisfying that $u(0)=u$, $u'(0^+)=\phi$, $t_0>0$. (Note that both these notions are different from the Clarke semi-derivative, in which the limit is taken in $u$, rather than in $\phi$, see \cite{Cl}). The meaning of \eqref{2.1}-\eqref{2.2} is that the Gateaux semi-derivative exists in the stronger sense stated by the respective right hand side limit. 

\smallskip

\noi {\bf Metric measure spaces.} For the material in Section \ref{Section 4}, we will need to use metric measure spaces with some additional structure, which we discuss now. More details on these spaces can be found e.g.\ in \cite{F, Ma,Pr,R}. Let \emph{$(\Om,d,\mathfrak L,\la)$ be a Borel regular metric measure space}. This means $d$ is a is metric (distance) on $\Om$ and $(\Om,\mathfrak L,\la)$ forms a measure space, such that the $\si$-algebra $\mathfrak L$ contains the Borel $\si$-algebra $\mathcal B(\Om)$, generated by the metric topology of $\Om$, and $\la$ has the Borel regularity property (every subset of $\Om$ is contained into a Borel set with the same measure). We will further need the \emph{differentiation property} to be valid, namely the Lebesgue-Besicovitch differentiation result holds true: for any $f\in \L^1_{\mathrm{loc}}(\Om,\la)$, we have
\beq
\label{2.5A}
\ \ \ \ \ \ f(x) = \lim_{\rho\to0} \frac{1}{\la(\mB_\rho(x))} \int_{\mB_\rho(x)} f\, \mathrm d \la, \ \ \ \la \text{-a.e. }x\in\Om. 
\eeq
We understand also that the \emph{differentiation property subsumes  the property that the measure is supported everywhere on $\Om$}, namely $\supp(\la)=\Om$. This is equivalent to the requirement that non-empty open sets in $\Om$ (and hence balls) have positive $\la$-measure, so in particular the fraction in \eqref{2.5A} is always well-defined. 

The archetypal example of a metric measure space with these properties is the Euclidean space endowed with an everywhere-supported Radon measure, but there exist simple sufficient conditions guaranteeing the differentiation property on general metric spaces without any linear structure (eg.\ Vitali measures, doubling measures, see \cite{R} and \cite{Ma,Pr}). Equality \eqref{2.5A} has some important consequences. By choosing $f=\chi_A$, the characteristic function of a set $A \in \mathfrak L$, we have
\beq
\label{2.6A}
\lim_{\rho\to0} \mathcal D_\rho (\la,A,x) = 
\left\{ 
\begin{array}{ll}
 1, & \la \text{-a.e. }x\in A, \phantom{\Big]}
\\
0, &  \la \text{-a.e. }x\in \Om \set A,
\end{array} \ \ \ \ \ \mD_\rho (\la,A,x) := \frac{\la(\mB_\rho(x)) \cap A)}{\la(\mB_\rho(x))} .
\right.
\eeq
This allows to define the \emph{measure-theoretic closure, interior and boundary of any $A \in \mathfrak L$}: 
\beq
\label{2.7A}\left\{ \ \ \ \ 
\begin{split}
\overline{A}^\la &:= \Big\{x\in \Om :  \limsup_{\rho\to0} \mD_\rho (\la,A,x)>0\Big\},
\\
A^{\la \circ} &:= \Big\{x\in \Om :  \lim_{\rho\to0} \mD_\rho (\la,A,x)  =1 \Big\},
\\
\p^\la A &:= \overline{A}^\la \set A^{\la \circ}.
\end{split}
\right.
\eeq
By \eqref{2.6A}-\eqref{2.7A}, every set $A \in \mathfrak L$ differs by its measure-theoretic interior and closure by at most a $\la$-nullset, and the measure-theoretic boundary is a $\la$-nullset as well. For convenience we might on occasion use the modifier ``$\la$-" instead of saying ``measure-theoretic", which also removes any ambiguity on the measure used. The relations of $\overline{A}$$^\la$, $A^{\la \circ} $ to their topological counterparts are as follows
\beq
\label{2.8A}
\overline{A}^\la \sub \overline{A}, \ \ \ A^{\circ} \sub A^{\la \circ},
\eeq
with the inclusions in general being strict. To see \eqref{2.8A}, fix $x\in A^{\circ} \sub A$. Then, there exists a ball $\mB_\rho(x) \sub A$, which implies that $\la (\mB_\rho(x) \cap A) =\la (\mB_\rho(x))$, whence $\mD_\rho(\la,A,x)=1$ and thus $x\in A^{\la \circ}$. If $x\in\Om\set \overline{A}$, there exists a ball $\mB_\rho(x) \sub \Om\set \overline{A} \sub \Om \set A$, which implies $\la (\mB_\rho(x) \cap A) =0$, whence $\mD_\rho(\la,A,x)=0$ and thus $x\in \Om \set \overline{A}$$^\la$. 

\begin{remark}[Supremum vs essential supremum] \label{Remark17} If the measure $\la$ is supported everywhere on $\Om$, then the supremum and the essential supremum coincide for continuous functions. Due to its significant role herein, we do provide a simple proof of this fact. Since $f \leq \sup_\Om f$ everywhere on $\Om$, we readily have
\[
\la \text{-} \underset{\Om}{\ess \sup} \, f  \leq \sup_\Om f.
\]
Conversely, by the definition of the essential supremum, we have 
\[
\la \text{-} \underset{\Om}{\ess \sup} \, f  \geq t, \ \ \ \forall\, t\in \R : \ \la(\{f>t\})>0.
\]
We set $T(\de):=\sup_\Om f -\de$, for $\de>0$. If $\de$ is small enough, the continuity of $f$ implies that the open set $\{f > {\sup}_\Om f -\de\}$ is non-empty in $\Om$. Then, since $\supp (\la) =\Om$, we have that $\la\big(\{f > {\sup}_\Om f -\de\}\big) >0$. This implies for the choice of $t=T(\de)$ that
\[
\la \text{-} \underset{\Om}{\ess \sup} \, f  \geq T(\de)= \sup_\Om f -\de.
\]
We conclude by letting $\de\to 0^+$.
\end{remark}

\noi {\bf Approximate limits, approximate continuity.} Continuing in the same context of Borel regular metric measures spaces, let $f : \Om \larrow \R$ be $\la$-measurable. For any $x\in \Om$, the \emph{approximate limsup of $f$ at $x$} is defined as
\beq
\label{2.9A}
\la\textrm{-ap}\limsup_{y\to x} f(y) := \sup\Bigg\{ t\in \R \, :\,  \limsup_{\rho\to0}  \frac{\la \big(\mB_\rho(x)) \cap \{f \geq t\} \big) }{\la(\mB_\rho(x))}>0\Bigg\}.
\eeq
Similarly, 
the \emph{approximate limit of $f$ at $x$} is defined as
\beq
\label{2.10A}
\la\textrm{-ap}\lim_{y\to x} f(y) :=L \ \ \iff\ \  \forall \, \e>0, \ \ \ \lim_{\rho\to0}  \frac{\la \big(\mB_\rho(x)) \cap \{|f-L| \geq \e \} \big) }{\la(\mB_\rho(x))}=0.
\eeq
Approximate limits are unique, if they exist. $f$ is called \emph{approximately continuous at $x$} if
\[
\la\textrm{-ap}\lim_{y\to x} f(y) = f(x).
\]
In the Euclidean space (see e.g.\ \cite{EG,F}), a.e.\ approximate continuity is equivalent to measurability of a function.

\smallskip

\noi {\bf Elements of the Calculus of Variations in $\L^\infty$.} We collect some rudiments on the properties of supremal functionals, related ideas and some new results and observations, only inasmuch as they are needed for the subsequent developments in Sections \ref{Section 6} and \ref{Section 7}. Let $n,N,k\in \N$ and let $\Om\sub \R^n$ be an open set, and consider the supremal functional \eqref{1.6}  arising from a Carath\'eodory function $\H : \Om \by \smash{\R^{\vec M}} \larrow \R$, namely
\[
\left\{  \ \ \ \ 
\begin{split}
\E_\infty & \  : \ \W^{k,\infty}(\Om;\R^N) \by \mL(\Om) \larrow \R,
\\
&\E_\infty(u,\mO) := \underset{\mO}{\ess\sup}\, \H \big(\cdot, \D^{\vec{k}}u \big) .
\end{split}
\right. \ \ \ \
\]
As mentioned in the introduction, supremal functionals where first studied by Aronsson in the 1960s (in the case of $k=N=1$), motivated by the problem of optimising Lipschitz extensions, due to the relation of the $\L^\infty$ norm of the gradient functional 
\beq
\label{5.1}
\I_\infty \ : \ \W^{1,\infty}(\Om;\R^N) \by \mL(\Om) \larrow \R , \ \ \ \ \I_\infty(u,\mO) := \underset{\mO}{\ess \sup} \,|\D u|,
\eeq
to the Lipschitz constant (when $N=1$). The functional $\I_\infty$ is the archetypal model of $\L^\infty$ functionals, and has been studied extensively. The scalar case is largely well understood. The vectorial case, however, is much more complicated, and the theory is still evolving. Despite the energetic activity in the area, many important problems remain open.

The emergence of the second, subdomain variable in the definition of supremal functionals is essential. Unlike integral functionals, a major hindrance is that global minimisers generally fail to minimise on subdomains with respect to their own boundary values. The latter locality property is automatic for integral functionals. A further issue closely related to the lack of locality is that global minimisers of supremal functionals are generally not unique either, not even for the simplest strictly convex supremands.

\begin{example} \label{Example20} Consider \eqref{5.1} for $n=1$ and $\Om=(-2,0)\cup(0,2)$. Let us take
\[
u^+(x):=\max\{1-|x+1|,x\}, \ \ \ u^0(x):=\max\{0,x\}, \ \ \ u^-(x):=|x+1| -1.
\]
Then, all three functions $u^\pm,u^0$ are global minimisers of $\I_\infty(\cdot, \Om)$ in the space $\big\{ u\in \W^{1,\infty}(\Om) : u(-2)=0, u(0)=0, u(2)=2 \big\}$, but neither $u^+$ nor $u^-$ minimise $\I_\infty(\cdot, \mO)$ when $\mO =(-2,0)$ for their own boundary values (i.e.\ for $u(-2)=0,u(0)=0$). Instead, the unique minimiser of $\I_\infty(\cdot, \mO)$ for these values is $u^0$ only.
\end{example}
The above feature of supremal functionals has led to the following trademark minimality concept, which restores the desired locality by building it into the minimality notion:
\begin{definition}[Absolute minimisers]
\label{Definition21}
Consider the supremal functional \eqref{1.6}, and a given mapping $u\in \W^{k,\infty}(\mO;\R^N)$. We say that $u$ is an absolute minimiser of $\E_\infty$ on $\Om$ when
\[
\ \ \E_\infty(u,\mO) \leq \E_\infty(u+\phi,\mO), \ \ \ \ \forall\,\, \mO \sub \Om \text{ open, }\ \forall\,\, \phi \in \W^{k,\infty}_0(\mO;\R^N).
\]
\end{definition}

\begin{remark} (i) Definition \ref{Definition21} is rather standard in the area, but it creates more problems that it actually solves, as it throws out the direct method of the Calculus of Variations \cite{D}. Absolute minimisers in general cannot be obtained in this way because minimising sequences may converge subsequentially to \emph{any} of the global minimisers. Alternative methods are needed, and the standing idea is to approximate $\L^\infty$ functionals by the corresponding integral $\L^p$ functionals. By studying the limit as $p\to\infty$, the expectation is that, since integral functionals have the desired locality properties, the approximation method will select in the limit the correct object, namely a global absolute minimiser (more on this method below). 

\smallskip

\noi (ii) Definition \ref{Definition21} is at slight variance to other ``localised" versions of this notion in the literature, which often require the minimality condition to be satisfied only on compactly contained $\mO\Subset \Om$, rather than all open $\mO \sub \Om $. The benefit of our slightly stronger version is that it allows to combine into a single definition absolute plus global minimisers, as typically the former are constructed for given boundary data by selecting one of the latter. In the scalar case for a bounded domain plus structural assumptions, absolute minimisers are unique (see e.g.\ \cite{ACJS, BJW1, C1}), hence the two notions coincide, but in general this is not so clear. Anyhow, for the results established herein, it makes no difference in the proofs which version of absolute minimality is used, and Definition \ref{Definition21} is preferable as it is more suitable for our purposes.
\end{remark}

The result below reformulates Definition \ref{Definition21} in a very useful fashion, in that it allows the two-parameter family of subdomains and variations to be reduced to a single parameter family of variations only, where the relevant open subsets are their own non-zero sets. In other words, Proposition \ref{Proposition20} below says that we \emph{may discard compactly supported variations on subdomains altogether and consider only those vanishing on the boundary, as they do not contribute anything to the minimality notion in $\L^\infty$}. This observation appears to be new in the literature. Interestingly though, in the very recent work \cite{CKM1} compactly supported variations are isolated and studied in some detail, demonstrating some of their striking features.

\begin{proposition}[Equivalent formulations for absolute minimisers] \label{Proposition20} Let $u \in \W^{k,\infty}(\Om ;\R^N)$ be given. Consider the following statements for the supremal functional \eqref{1.6}:
\begin{enumerate}
\item For any open $\mO \sub \Om$ and any $\phi \in \W^{k,\infty}_0(\mO;\R^N)$, we have
\[
\E_\infty(u,\mO) \leq \E_\infty(u+\phi,\mO). 
\]
\item For any $\psi \in \W^{k,\infty}_0(\Om;\R^N)^\varobslash$, we have
\[
\E_\infty\big(u,\{\psi \neq 0\}\big) \leq \E_\infty\big(u+\psi,\{\psi \neq 0\}\big) .
\]
\item For any $\psi \in \W^{k,\infty}_0(\Om;\R^N)^\varobslash$, we have
\[
\E_\infty\big(u, \supp(\psi)\cap\Om\big) \leq \E_\infty\big(u+\psi, \supp(\psi)\cap\Om \big) .
\]
\end{enumerate}
Then (1) and (2) are equivalent, whilst they both imply (3). Conversely, (3) implies (1) and (2) if either $u,\psi \in \C^k(\overline{\Om};\R^N)$, or we restrict all classes of variations $\phi, \psi \in \W^{k,\infty}_0(\Om;\R^N)^\varobslash$ only to those which satisfy $\mL^n(\p\{\phi \neq 0\})=\mL^n(\p\{\psi \neq 0\})=0$.
\end{proposition}

Proposition \ref{Proposition20} has also a localised version involving compactly contained subdomains. The statement is given right below for the sake of completeness, and the proof is identical to that of Proposition \ref{Proposition20}.

\begin{corollary}[Equivalent formulations for absolute minimisers, localised version] \label{Corollary22} Let $u \in \W^{k,\infty}(\Om ;\R^N)$ be given. Consider the following statements for the supremal functional \eqref{1.6}:
\begin{enumerate}
\item For any open $\mO \Subset \Om$ and any $\phi \in \W^{k,\infty}_0(\mO;\R^N)$, we have
\[
\E_\infty(u,\mO) \leq \E_\infty(u+\phi,\mO). 
\]
\item For any $\psi \in \W^{k,\infty}_c(\Om;\R^N)^\varobslash$, we have
\[
\E_\infty\big(u,\{\psi \neq 0\}\big) \leq \E_\infty\big(u+\psi,\{\psi \neq 0\}\big) .
\]
\item For any $\psi \in \W^{k,\infty}_c(\Om;\R^N)^\varobslash$, we have
\[
\E_\infty\big(u, \supp(\psi)\big) \leq \E_\infty\big(u+\psi, \supp(\psi) \big) .
\]
\end{enumerate}
Then (1) and (2) are equivalent, whilst they both imply (3). Conversely, (3) implies (1) and (2) if either $u,\psi \in \C^k(\overline{\Om};\R^N)$, or we restrict all classes of variations $\phi, \psi \in \W^{k,\infty}_c(\Om;\R^N)^\varobslash$ only to those which satisfy $\mL^n(\p\{\phi \neq 0\})=\mL^n(\p\{\psi \neq 0\})=0$.
\end{corollary}

\BPP \ref{Proposition20}. It suffices to prove the implications for non-zero test functions only.

\noi \underline{\emph{(1)} $\Rightarrow$ \emph{(2)}}: Assume \emph{(1)} and fix $\psi \in \W^{k,\infty}_0(\Om;\R^N)^\varobslash$. Then, $\{\psi \neq 0\} \sub \Om$ is a non-empty open subset, and by Claim \ref{Claim9} we have  $\smash{\psi \in \W^{k,\infty}_0\big(\{\psi \neq 0\};\R^N\big)}$. We conclude by applying \emph{(1)} to $\mO:= \{\psi \neq 0\}$ for $\phi := \psi \in \W^{k,\infty}_0(\mO;\R^N)$, to deduce \emph{(2)}.

\smallskip

\noi \underline{\emph{(2)} $\Rightarrow$ \emph{(1)}}: Assume \emph{(2)}, and fix an open set $\mO \subseteq \Om$ and $\phi \in \W^{k,\infty}_0(\mO;\R^N)^\varobslash$. By extending $\phi$ by zero on $\Om\set  \mO$, we have $\smash{\phi \in \W^{k,\infty}_0(\Om;\R^N)}$. Then we have  $\smash{\D^{\vec k} \phi} =0$, a.e.\ on $\mO \set \{\phi\neq 0\}$ $\sub \Om$, because of Claim \ref{Claim9}. Therefore, we obtain
\[
\begin{split}
\E_\infty\big(u,\mO \big) &=  \underset{\mO}{\ess\sup}\, \H\big( \cdot, \D^{\vec k} u \big)
\\
&= \max\bigg\{  \underset{\mO  \set \{\phi\neq 0\}}{\ess\sup}\, \H\big( \cdot, \D^{\vec k} u \big) , \  \underset{ \{\phi\neq 0\}}{\ess\sup}\, \H\big( \cdot, \D^{\vec k} u \big) \bigg\}
\\
&= \max\bigg\{  \underset{\mO  \set \{\phi\neq 0\}}{\ess\sup}\, \H\big( \cdot, \D^{\vec k} u+ \D^{\vec k} \phi \big) , \ \underset{ \{\phi\neq 0\}}{\ess\sup}\, \H\big( \cdot, \D^{\vec k} u \big) \bigg\}
\\
&\overset{\emph{(2)}}{\leq}  \max\bigg\{  \underset{\mO  \set \{\phi\neq 0\}}{\ess\sup}\, \H\big( \cdot, \D^{\vec k} u+ \D^{\vec k} \phi \big) , \ \underset{ \{\phi\neq 0\}}{\ess\sup}\, \H\big( \cdot, \D^{\vec k} u+ \D^{\vec k} \phi \big) \bigg\}
\\
&=  \underset{\mO}{\ess\sup}\, \H\big( \cdot, \D^{\vec k} u+ \D^{\vec k} \phi \big)
\\
& = \E_\infty\big(u+\phi,\mO\big),
\end{split}
\]
for any $\mO \sub \Om$ open and any $\phi \in \W^{k,\infty}_0(\mO;\R^N)^\varobslash$. Therefore, \emph{(1)} ensues.

\smallskip

\noi \underline{\emph{(2)} $\Rightarrow$ \emph{(3)}}: Assume \emph{(2)} and fix  $\smash{\phi \in \W^{k,\infty}_0(\Om;\R^N)} ^\varobslash$. Then, $\{\phi \neq 0\}$ is open in $\Om$, and 
\[
\supp(\phi) = \{\phi \neq 0\} \biguplus \p \{\phi \neq 0\} 
\] 
where $\Om\cap(\p \{\phi \neq 0\} ) \sub \{\phi = 0\} \sub {\Om} $. If $\mL^n\big(  \Om\cap (\p \{\phi \neq 0\} )\big) =0$, then we are done, because
\[
\underset{ \Om\cap \supp(\phi) }{\ess\sup}\, f= \, \underset{ \{\phi \neq 0\} }{\ess\sup}\, f,
\]
for any $f\in \L^\infty(\Om)$. If instead $\mL^n\big(  \Om\cap (\p \{\phi \neq 0\} )\big)  >0$, then by Sobolev function theory we have that $\smash{\D^{\vec k} \phi =0}$ a.e.\ on $ \Om\cap (\p \{\phi \neq 0\})$, because $\Om\cap (\p \{\phi \neq 0\}) \sub \{\phi = 0\} \sub \Om$ and additionally $\phi \in \W^{k,\infty}_0(\Om;\R^N)$. 
These facts allow us to write
\[
\begin{split}
\E_\infty\big(u, \Om\cap \supp(\phi)\big) &=  \underset{\Om\cap \supp(\phi)}{\ess\sup}\, \H\big( \cdot, \D^{\vec k} u \big)
\\
&= \max\bigg\{  \underset{\Om\cap(\p\{\phi\neq 0\} )}{\ess\sup}\, \H\big( \cdot, \D^{\vec k} u \big) , \  \underset{ \{\phi\neq 0\}}{\ess\sup}\, \H\big( \cdot, \D^{\vec k} u \big) \bigg\}
\\
&= \max\bigg\{  \underset{ \Om\cap(\p\{\phi\neq 0\} ) }{\ess\sup}\, \H\big( \cdot, \D^{\vec k} u+ \D^{\vec k} \phi \big) , \ \underset{ \{\phi\neq 0\}}{\ess\sup}\, \H\big( \cdot, \D^{\vec k} u \big) \bigg\},
\end{split}
\]
which yields the estimate
\[
\begin{split}
\E_\infty\big(u, \Om\cap \supp(\phi)\big) &\overset{\emph{(2)}}{\leq}  \max\bigg\{  \underset{ \Om\cap(\p\{\phi\neq 0\} ) }{\ess\sup}\, \H\big( \cdot, \D^{\vec k} u+ \D^{\vec k} \phi \big) , \ \underset{ \{\phi\neq 0\}}{\ess\sup}\, \H\big( \cdot, \D^{\vec k} u+ \D^{\vec k} \phi \big) \bigg\}
\\
&=  \underset{ \Om\cap \supp(\phi) }{\ess\sup}\, \H\big( \cdot, \D^{\vec k} u+ \D^{\vec k} \phi \big),
\\
&\leq \E_\infty\big(u+\phi, \Om\cap\supp(\phi)\big),
\end{split}
\]
for any $\phi \in \W^{k,\infty}_0(\Om;\R^N)^\varobslash$. Hence, \emph{(3)} ensues.

\smallskip

\noi \underline{\emph{(3)} $\Rightarrow$ \emph{(2)}}: Under the additional assumed hypotheses, by Remark \ref{Remark17} and standard measure-theory, we have the identity 
\[
 \smash{\E_\infty\big(v, \Om\cap\supp(\psi)\big) =  \E_\infty\big(v, \{\psi\neq 0\}\big)}, 
 \]
 for $v \in\{ u,u+\psi\}$. The conclusion  ensues. \qed
\ms

The primary convexity notion that renders supremal functionals weakly* sequentially lower semi-continuous is that of level-convexity with respect to the leading term, see \cite{BJW2, GM, P1, P2, RZ2}. This notion, which requires convex sublevel sets, is much weaker than convexity. Similarly to the vectorial case for integral functionals, the notion which is not only sufficient but also necessary in the vectorial supremal case is slightly weaker than level-convexity, an implicit notion due to Barron-Jensen-Wang which we call \emph{BJW-convexity}, introduced in \cite{BJW2}. The corresponding leading notion for integral functionals which characterises sequential lower semi-continuity in the vectorial case (under extra bounds), is of course the celebrated Morrey's quasiconvexity notion \cite{D}. It was later proved that BJW-convexity is indeed strictly weaker than level-convexity when $N\geq2$ (see \cite{RZ2}). Notwithstanding, convexity notions and weak* sequential lower semi-continuity considerations are less consequential for supremal functionals when compared to integral functionals, because the direct method does not yield \emph{absolute} minimisers, only global minimisers.

One further central hindrance to the development of the field from the start has been the lack of Gateaux differentiability of supremal functionals. Unlike the integral case, whereat regularity of the integrand and growth bounds easily imply differentiability of the functional by using the chain rule, in the supremal case the situation does not improve regardless of any stringent assumptions imposed on the supremand. The following example is rather striking, as it shows that \emph{the model  functional \eqref{5.1} is Gateaux non-differentiable at any absolute or global minimiser even for $n=N=1$}. 

\begin{example}[Non-differentiability of supremal functionals]\label{Example23} Let $\Om \sub \R$ be any open bounded interval, and consider the supremal functional \eqref{5.1} for $N=1$.
In this case, it is easy to confirm that absolute minimisers of $\I_\infty$ are unique (given endpoint boundary data), coincide with global minimisers and consist of the class of affine functions on $\R$. Fix any affine function $A : \R \larrow \R$ with $A(x):=ax+b$, $a>0$, $b\in \R$, which is an absolute minimiser of $\I_\infty$ on $\Om$, and also a global minimiser on $\Om$ for its own boundary values on $\p\Om$. By Remark \ref{Remark17} we have
\[
\I_\infty (u,\mO) = \max_{\overline{\Om}} | u'|^2, \ \ \text{ for all } u\in \C^1(\overline{\Om}).
\]
Fix $\phi \in (\C^0_0 \cap \C^1)(\overline{\Om})$. By minimality we have
\[
\I_\infty (A,\Om) \leq \I_\infty (A+t\phi,\Om), \ \ \text{ for all }t\neq 0, 
\]
which implies
\[
\liminf_{t\to0^+}\frac{\I_\infty (A+t\phi,\Om) - \I_\infty (A,\Om)}{t} \geq 0 \geq \limsup_{t\to0^-}\frac{\I_\infty (A+t\phi,\Om) - \E_\infty (A,\Om)}{t}.
\]
By Danskin's theorem applied to $\F : \R \by \overline{\Om} \larrow \R$ given by $ \F(t,x) := |a+t\phi'(x)|^2$, we have that the left and right semi-differentials exist. Since $|A'|=A' \equiv a>0$ on $\overline{\Om}$, we have that $\Argmax\{|A'| :  \overline{\Om}\} = \overline{\Om}$. This implies
\[
\left\{ \ \ 
\begin{split}
\big( \mfD \I_\infty(\cdot,\Om)\big)^+_A (\phi) &=\max_{\Argmax\{|A'| \,: \, \overline{\Om}\}} [A'\, \phi'] = a \max_{\overline{\Om}} [  \phi']  \geq 0,
\\
\big( \mfD \I_\infty(\cdot,\Om)\big)^-_A (\phi) &=\min_{\Argmax\{|A'| \,:\, \overline{\Om}\}} [A'\, \phi' ] =a \min_{\overline{\Om}} [  \phi'] \leq 0.
\end{split} 
\right. \ \ \  \
\]
Since $\phi =0$ on $\p\Om$, we have $\int_\Om \phi'=0$. If $\phi \not\equiv 0$, there must exist points $x^\pm \in \Om$ such that $\phi'(x^+)>0>\phi'(x^-)$. Thus, the left/right semi-derivatives do not match:
\[
\big( \mfD \I_\infty(\cdot,\Om)\big)^+_A (\phi) >\big( \mfD \I_\infty(\cdot,\Om)\big)^-_A (\phi), \ \ \text{ for \emph{all} }\phi \in (\C^0_0 \cap \C)^1(\overline{\Om})\set\{0\}.
\]
Therefore $\I_\infty(\cdot,\Om)$ is not Gateaux differentiable at any affine (absolute and global) minimiser, along any direction. By the same reasoning, it follows that $\I_\infty(\cdot,\mO)$ is not Gateaux differentiable at any minimiser for any open subdomain $\mO \sub \Om$ and along any direction. 
\end{example}

The lack of differentiability of \eqref{1.6} has been partially circumvented in the literature by considering approximations by differentiable functionals, typically the corresponding $\L^p$ functionals for $p<\infty$. The idea, exemplified below in the case of the scalar-valued first order supremal functionals (depending only on $(x,\D u)$), is to approximate by the corresponding $\L^p$-functionals, and use the following rectangle to discover a PDE by passing to the limit. 
 \[
 \begin{split}
&\ \ {\E_p(u,\Om) := \left(\, \av_\Om \H(\cdot, \D u)^p\, \mathrm d \mL^n \! \right)^{\!\!1/p} }\overset{p\to\infty}{-\!\!\!-\!\!\!-\!\!\!-\!\!\!-\!\!\!-\!\!\!-\!\!\!\larrow}\ \| \H(\cdot, \D u) \|_{\L^\infty(\Om)} = \E_\infty(u,\Om)
\\ 
&\begin{array}{c}
 \hspace{80pt}\big\uparrow \hspace{200pt} {\big\uparrow} \hspace{100pt}
\\
 \vspace{-0.5em}
 \hspace{79.34pt} \bigg| \hspace{204pt}  {\vdots} \hspace{100pt}
 \\
 \hspace{80pt}\big\downarrow \hspace{200pt} {\big\downarrow} \hspace{100pt}
 \end{array}
\\
& \hspace{14pt} \div\big(\H(\cdot, \D u)^{p-1} \H_{,X}(\cdot, \D u)\big) =0\ \  \overset{p\to\infty}{-\!\!\!-\!\!\!-\!\!\!-\!\!\!-\!\!\!-\!\!\!-\!\!\!\larrow} \ \ \ \ \ \ \ \ \ {\text{ ???}}  
\end{split}
\]
The expectation is that the PDE you can discover via this limiting process, subsequently it could perhaps be connected directly to the supremal functional, as one cannot take variations of $\E_\infty$ to derive it. Despite $\E_\infty$ being non-differentiable, it is remarkable that this process does indeed lead to PDEs that can a posteriori be related to the functional. The method however is not universal, and depending on how one passes to the limit as $p\to \infty$, it yields more than one PDE, and of different types. Below are the relevant formal arguments:

\smallskip
\noi {\bf Approach 1.} This idea dates back to Aronsson \cite{A3,A4}. By distributing the divergence and normalising, we obtain
\[
\D\big(\H(\cdot, \D u)\big)\H_{,X}(\cdot, \D u) + \frac{\H(\cdot, \D u)\big) \div \big(\H_{,X}(\cdot, \D u)\big)}{p-1} =0 , \ \ \text{ in }\Om,
\]
from where we discover the Aronsson equation as $p\to\infty$ (see also Theorem \ref{Theorem26} in  Section \ref{Section 8}). This PDE has been studied quite extensively in the literature, in the context of viscosity solutions (\cite{BEJ, BJ, BJW1, C1, CWY, DPZZ, KZZ, MWZ, PWZ, Yu}). Satisfaction of this PDE indeed can be proved to be necessary for absolute minimality. However, except for special cases and under strong convexity assumptions (see \cite[Theorem 1, p.\ 155]{Yu}, and Example \ref{Example36} in Section \ref{Section 8}), the Aronsson equation obtained via this process is not sufficient for minimality. 

\smallskip

\noi {\bf Approach 2.} This approach is more recent and is due to Evans-Yu in the scalar case \cite{EY} (and due to the author with Moser in the vectorial and higher order case \cite{KM1, KM2, KM3}, see also \cite{AyK1, CK1, CK2, K1, K3, K3.5}). For any $\mO \sub \Om$, rescale the PDE as
\[
\div\bigg(\dfrac{ \H(\cdot, \D u)^{p-1}  }{\E_p(u,\mO)^{p-1}} \H_{,X}(\cdot, \D u) \bigg) =0,  \ \ \text{ in }\Om.
\]
Then, by passing to limits, one obtains a family of divergence PDEs for measures, parameterised by the topology of $\Om$:
\[
\ \ \ \div\big(\si_\mO  \H_{,X}(\cdot, \D u) \big) =0, \ \ \si_\mO \in \mathcal M^+(\overline{\mO}), \ \ \mO \sub \Om \text{ open}.
\]
The main issue here is that one does not have a single PDE, but rather a parametric family. Further, these measures typically are supported on the boundary $\p\mO$, charging the ``argmax" set $\{\H(\cdot,\D u) : \overline{\mO}\}$. Thus, the PDE in that case is trivial inside the domain. This approach, that clearly has inherent limitations, is far less understood than Aronsson equations.

\begin{remark}[The vectorial case] Now we discuss briefly the vectorial case. Since it presents several additional layers of complexity, we will only consider the archetypal case of the $\infty$-Laplacian. It was first derived from $\L^p$-approximations  in \cite{K10}. For maps $u : \R^n \supseteq \Om \larrow \R^N$, it reads
\[
\De_\infty u := \Big( \D u \ot \D u + |\D u|^2[\![ \D u ]\!]^\bot\! \ot \I \Big) \!:\! \D^2 u =0 \ \ \text{ in }\Om,
\]
where for any matrix $\X \in \R^{N\by n}$, $[\![ \X ]\!]^\bot := \mathrm{Proj}_{\mathrm{R}(\X : \R^n \to \R^N)^\bot}$ is the orthogonal projection on the orthogonal complement of its range. $\De_\infty u =0$ is the simplest non-trivial Aronsson system, but it has discontinuous coefficients even for smooth solutions, which have interfaces whereon the dimension of the tangent space to $u(\Om)$ changes \cite{K7, K8, K9, K10}. The system actually consists of two differential operators orthogonal to each other, the tangential one $\D u \D(|\D u|^2)=0$, and the normal one $[\![ \D u ]\!]^\bot\De u=0$, which expresses that $\De u$ is tangential to $u(\Om)$. Further, it appears that $\infty$-Harmonic maps \textbf{cannot} be characterised through absolute minimisers, as classical solutions to $\De_\infty u =0$ are characterised in \cite{K8} through different classes of variations, specifically rank-one absolute minimals and variations normal to $u(\Om)$ but free on the boundary (see also  \cite{AyK1, AyK2} for extensions). In \cite{KS} it was shown that the Dirichlet problem admits smooth non-minimising solutions, therefore, similarly to the situation in Example \ref{Example36}, $\De_\infty $ is clearly not sufficient for (absolute) minimality. Finally, viscosity solutions do not apply to this non-divergence non-monotone system, and one needs alternative novel approaches to study it rigorously, as solutions are typically non-classical (\cite{CKP, K4}). 
\end{remark}


\section{Semi-differentiability of generalised $\L^\infty$ envelopes}
\label{Section 3}

In this section we establish our first main result, Theorem \ref{Theorem2}, as well as some supplementary results, all in the generality of arbitrary Banach spaces and measure spaces. We begin with the proof of  Theorem \ref{Theorem2}, which consists of Lemmas \ref{Lemma6}  and \ref{Lemma7}. These lemmas establish as two halves an upper bound and a lower bound for the upper/lower semi-derivatives, which together imply the desired conclusion \eqref{1.13}-\eqref{1.14}. Let $(\mX,\|\cdot\|)$ be a Banach space with dual space $(\mX^*,\|\cdot\|_*)$, and let $(\Om,\mathfrak L,\la)$ be a measure space. Let $\F : \mX \by \Om \larrow \R$ be a function which satisfies \eqref{1.8}, namely $\F(u,\cdot) \in \L^\infty(\Om,\la)$ for all $u\in \X$, and $\F(\cdot,x) \in \C^1(\mX)$ for $\la$-a.e.\ $x\in \Om$. Consider the  $\L^\infty$ $\la$-envelope of $\F$ with respect to $x\in \Om$, defined in \eqref{1.12}:
\[
f(u) = \la \text{-} \underset{x\in \Om}{\ess\sup}\, \F(u,x), \ \ \ f \ : \ \mX \larrow \R.
\]
For any $\e>0$, recall also the approximate argmax sets, defined in \eqref{1.14}:
\[
 \Om_\e(u) = \Big\{  \ \F(u,\cdot) \geq   \la \text{-} \underset{x\in \Om}{\ess\sup}\, \F(u,x)-\e\Big\}.
 \]
 For convenience, let us restate assumptions \eqref{1.10}-\eqref{1.11} of Theorem \ref{Theorem2} in a more explicit fashion: for any $v,\phi \in \mX$, there exist increasing moduli of continuity $\om_{v,\phi}, \om_v \in \C(0,\infty)$ with $\om_{v,\phi}(0^+)=\om_v (0^+)=0$ (depending on $\{v,\phi\}$ and on $v$ respectively), such that for all $w\in\mX$
\begin{align}
\label{1.15}
 \la \text{-} \underset{x\in \Om}{\ess\sup}\, \Big| \Big[\big(\p\F(\cdot,x)\big)_w 
\! - \! \big(\p\F(\cdot,x)\big)_v\Big](\phi)\Big|  & \leq \om_{v,\phi} \big( \| w-v\|\big), 
\\
\label{1.16}
 \la \text{-} \underset{x\in \Om}{\ess\sup}\, \Big| \F(w,x)  - \F(v,x) \Big| & \leq \om_{v} \big( \| w-v\| \big).
\end{align}

\begin{lemma} \label{Lemma6} Suppose that $\F$ satisfies \eqref{1.8}, and additionally $\p \F (\cdot,x) : \mX \larrow \mX^*$ is strongly-weakly* continuous, $\la$-essentially uniformly in $x\in \Om$, i.e.\ \eqref{1.15} holds. Then, for any $u,\phi \in \mX$
\beq
\label{3.1}
\begin{split}
(\underline{\mfD} f)^+_u(\phi) \geq \lim_{\e \to 0^+} \! \bigg(\la \text{-\hspace{1pt}} \underset{x\in \Om_\e(u)}{\ess\sup}\, \big(\p \F (\cdot,x)\big)_u(\phi) \bigg).
\end{split}
\eeq  
\end{lemma}
Note that the right hand side limit exists always, by virtue of the properties of $\F$ and the monotonicity of the set-mapping $\e \mapsto \Om_\e(u)$.

\BPL \ref{Lemma6}. Fix $u,\phi \in \mX$, $t>0$ and $\e>0$. By utilising \eqref{1.13}-\eqref{1.14}, for $\la$-a.e.\ $x \in \Om_\e(u)$, we may estimate
\[
\begin{split}
f(u+t\phi) - f(u) &=  \la \text{-} \underset{x\in \Om}{\ess\sup}\, \F(u+t\phi ,x) -  \la \text{-} \underset{x\in \Om}{\ess\sup}\, \F(u,x)
\\
&\geq  \la \text{-} \underset{x\in \Om}{\ess\sup}\, \F(u+t\phi ,x) - \F(u,x) -\e
\\
&\geq  \F(u+t\phi ,x) - \F(u, x) -\e.
\end{split}
\]
By the Fr\'echet differentiability of $\F(\cdot,x)$, the fundamental theorem of Calculus yields
\[
\begin{split}
\frac{f(u+t\phi) - f(u)}{t} &\geq   \int_0^1 \big(\p\F(\cdot,x)\big)_{u+\la t\phi }(\phi) \,\mathrm d \la -\frac{\e}{t}
\\
&\geq   \int_0^1 \big(\p\F(\cdot,x)\big)_{u}(\phi) \,\mathrm d \la    -\frac{\e}{t}
\\
& \ \ \ \ -   \la \text{-} \underset{\bar x\in \Om}{\ess\sup}\,   \bigg| \!\int_0^1 \Big[ \big(\p\F(\cdot, \bar x)\big)_{u+\la t\phi } -\big(\p\F(\cdot, \bar x)\big)_{u }\Big] (\phi) \,\mathrm d \la \bigg|.
\end{split}
\]
 for $\la$-a.e.\ $x \in \Om_\e(u)$. Hence, by taking the $\la$-ess\,sup with respect to $x\in \Om_\e(u)$, we have the estimate
\[
\begin{split}
\frac{f(u+t\phi) - f(u)}{t} & \geq   \la \text{-} \underset{x\in \Om_\e(u)}{\ess\sup}\,  \big(\p\F(\cdot,x)\big)_{u}(\phi)   -\frac{\e}{t}
\\
\ \ \ \  \ \ - \bigg( & \int_0^1 \mathrm d \la \bigg) \! \sup_{w\in\mX : \|w-u\|\leq t\|\phi\|} \!\! \bigg( \la \text{-} \underset{\bar x\in \Om}{\ess\sup}\, \Big| \big(\p\F(\cdot,\bar x)\big)_{w}(\phi)\! -\! \big(\p\F(\cdot,\bar x)\big)_{u }(\phi)\Big|  \bigg).
\end{split}
\]
In view of assumption \eqref{1.15}, the above estimate implies that there exists an increasing modulus of continuity $\om_{u,\phi} \in \C(0,\infty)$ with $\om_{u,\phi}(0^+)=0$ such that
\[
\begin{split}
\frac{f(u+t\phi) - f(u)}{t} \geq    \la \text{-} \underset{x\in \Om_\e(u)}{\ess\sup}\, \big(\p\F(\cdot,x)\big)_{u}(\phi)   -\frac{\e}{t} - \om_{u,\phi} \big( t\| \phi \| \big).
\end{split}
\]
Since the last term above tends to zero at $t\to0^+$, by letting first $\e \to 0^+$ and subsequently $t \to 0^+$, we conclude with
\[
\begin{split}
\liminf_{t\to0^+}\frac{f(u+t\phi) - f(u)}{t}  \geq \lim_{\e \to 0^+} \! \bigg(\la \text{-} \underset{x\in \Om_\e(u)}{\ess\sup}\,   \big(\p\F(\cdot,x)\big)_{u}(\phi)\bigg).
\end{split}
\]
This establish the claimed inequality. \qed
\ms

Now we demonstrate the upper bound for the semi-differential.

\begin{lemma} \label{Lemma7}  Suppose that $\F$ satisfies \eqref{1.8}, and additionally $\p \F (\cdot,x) : \mX \larrow \mX^*$ is strongly-weakly* continuous, $\la$-essentially uniformly in $x\in \Om$, that is \eqref{1.15} holds. Suppose also that $\F (\cdot,x) : \mX \larrow \R$ is continuous, $\la$-essentially uniformly in $x\in \Om$, namely \eqref{1.16} is satisfied. Then, for all $u,\phi \in \mX$ we have
\beq
\label{3.2}
\begin{split}
(\overline{\mfD} f)^+_u(\phi) \leq \lim_{\e \to 0^+} \! \bigg(\la \text{-\hspace{1pt}} \underset{x\in \Om_\e(u)}{\ess\sup}\, \big(\p \F (\cdot,x)\big)_u(\phi) \bigg).
\end{split}
\eeq 
\end{lemma}

\BPL \ref{Lemma7}. Let us fix $u,\phi \in \mX$ and $t>0$. We begin by noting that by assumption \eqref{1.16}, there exists an increasing modulus of continuity $\om_{u} \in \C(0,\infty)$ with $\om_{u} (0^+)=0$, such that 
\beq
\label{3.3}
\begin{split}
\la \text{-} \underset{\Om}{\ess\sup} \Big| \F(u+t\phi,\cdot) -\F(u,\cdot)\Big| &\leq \om_u \big(\|u+t\phi - u\| \big) = \om_u \big(t\|\phi\| \big) ,
\end{split}
\eeq
Now we establish an inclusion on the approximate argmax sets, stating that the set of a perturbation of an element is contained into an enlarged set of the element itself.

\begin{claim}\label{Claim8} For $u,\phi,\om_u, t$ as above and for any $\de>0$, we have
\[
\begin{split}
\Om_\de(u+t\phi)  \sub \Om_{\de+2\om_u(t\|\phi \|)}(u)  .
 \end{split}
 \]
\end{claim}

\BPC \ref{Claim8}. By definition \eqref{1.14} applied to $u+t\phi$, for any $x\in \Om_{\de}(u+t\phi)$, we have
\[
\begin{split}
\F(u+t\phi,x) &\geq \la \text{-} \underset{\Om}{\ess\sup}\, \F(u+t\phi,\cdot) -\de.
\end{split}
\]
Fix such a point $x\in \Om_{\de}(u+t\phi)$. By utilising \eqref{3.3} and the above estimate, we have
\[
\begin{split}
\F(u,x) &\geq \F(u+t\phi,x) -\om_u \big(t\|\phi\| \big)
\\
&\geq \la \text{-} \underset{\Om}{\ess\sup}\, \F(u+t\phi,\cdot) -\de-\om_u \big(t\|\phi\| \big)
\\
& \geq \la \text{-} \underset{\Om}{\ess\sup}\, \F(u,\cdot) -\big( \de + 2\om_u \big(t\|\phi\| \big)\big).
\end{split}
\]
In view of definition \eqref{1.14}, we deduce that $x\in \Om_{\de + 2 \om_u (t\|\phi\| )}(u)$, yielding the claimed inclusion $\Om_\de(u+t\phi)  \sub \Om_{\de+2\om_u(t\|\phi \|)}(u) $. 
\qed
\ms

\noi Now we complete the proof of Lemma \ref{Lemma7}. For any fixed $u,\phi \in \mX$ and $t>0$, let us select any $\e,\de>0$ such that
\beq
\label{3.4}
\de \leq 2\om_u(t\|\phi \|), \ \ \ \e \geq 4\om_u(t\|\phi \|).
\eeq
In view of \eqref{1.12} and \eqref{1.14}, for $\la$-a.e.\ $x \in \Om_\de(u+t\phi)$, we may estimate
\[
\begin{split}
f(u+t\phi) -f(u)& = \la \text{-} \underset{\Om}{\ess\sup}\, \F(u+t\phi,\cdot) - \la \text{-} \underset{\Om}{\ess\sup}\, \F(u,\cdot) 
\\
&\leq \F(u+t\phi,x) +\de - \la \text{-} \underset{\Om}{\ess\sup}\, \F(u,\cdot) 
\\
&\leq \F(u+t\phi,x)  - \F(u,x) +\de.
\end{split}
\]
By the Fr\'echet differentiability of $\F(\cdot,x)$ and assumption \eqref{1.16}, the fundamental theorem of Calculus applied to the above estimate, yields
\[
\begin{split}
f(u+t\phi) -f(u) &= t \int_0^1 \big(\p\F(\cdot,x)\big)_{u+\la t\phi}(\phi)\, \mathrm d \la +\de
 \\
 &\leq  \la \text{-} \underset{x\in \Om^\de(u+t\phi)}{\ess\sup}\, \bigg[ \int_0^1 \big(\p\F(\cdot,x)\big)_{u+\la t\phi}(\phi)\, \mathrm d \la \bigg]t +\de.
\end{split}
\]
Note now that by Claim \ref{Claim8} and \eqref{3.4}, we have the inclusions 
\beq
\label{3.5}
 \Om^\de(u+t\phi) \sub \Om_{\de+2\om_u(t\|\phi \|)}(u)  \sub \Om_{4\om_u(t\|\phi \|)}(u)  \sub \Om_{\e}(u).
\eeq
Therefore, by \eqref{3.5}, we may further estimate
\[
\begin{split}
\frac{f(u+t\phi) -f(u)}{t} &\leq  \la \text{-} \underset{x\in \Om_\e (u)}{\ess\sup}\, \bigg[ \int_0^1 \big(\p\F(\cdot,x)\big)_{u+\la t \phi}(\phi) \mathrm d \la \bigg] + \frac{\de}{t}
\\
&\leq  \la \text{-} \underset{x\in \Om_\e (u)}{\ess\sup}\, \bigg[ \int_0^1 \big(\p\F(\cdot,x)\big)_{u}(\phi) \mathrm d \la \bigg] + \frac{\de}{t}
 \\
 &\ \ \ + \sup_{w\in \mX : \|w-u\| \leq t\|\phi\|} \! \Bigg( \! \la \text{-} \underset{x\in \Om_\e(u)}{\ess\sup}\, \bigg|  \big(\p\F(\cdot,x)\big)_{w}(\phi) -\big(\p\F(\cdot,x)\big)_{u}(\phi) \bigg|\Bigg)
 \\
 &\leq  \la \text{-} \underset{x\in \Om_\e(u)}{\ess\sup}\, \big(\p\F(\cdot,x)\big)_{u}(\phi) + \frac{\de}{t}
  \\
 &\ \ \ + \sup_{w\in \mX : \|w-u\| \leq t\|\phi\|} \! \Bigg( \! \la \text{-} \underset{x\in \Om}{\ess\sup}\, \bigg|  \big(\p\F(\cdot,x)\big)_{w}(\phi) -\big(\p\F(\cdot,x)\big)_{u}(\phi) \bigg|\Bigg).
\end{split}
\]
By assumption \eqref{1.15}, there exists an increasing modulus of continuity $\om_{u,\phi} \in \C(0,\infty)$ with $\om_{u,\phi}(0^+)=0$ such that 
\[
\begin{split}
\frac{f(u+t\phi) -f(u)}{t} &\leq  \la \text{-} \underset{x\in \Om_\e(u)}{\ess\sup}\, \big(\p\F(\cdot,x)\big)_{u}(\phi) + \frac{\de}{t} + \om_{u,\phi} \big(t\| \phi\| \big).
\end{split}
\]
By \eqref{3.4}, we may let $\de\to0^+$ and subsequently $t\to0^+$, to deduce
\[
\begin{split}
\limsup_{t\to0^+} \frac{f(u+t\phi) -f(u)}{t} &\leq   \la \text{-} \underset{x\in \Om_\e(u)}{\ess\sup}\, \big(\p\F(\cdot,x)\big)_{u}(\phi) ,
\end{split}
\]
for any $\e>0$. By finally letting $\e\to0^+$, we see that Lemma \ref{Lemma7} has been established. \qed
\ms

The proof of Theorem \ref{Theorem2} is complete, as it is a consequence of Lemmas \ref{Lemma6} and  \ref{Lemma7}. \qed
\ms

An inspection of the arguments in the proofs of Lemmas \ref{Lemma6} and  \ref{Lemma7} yields that we have additionally obtained the following a priori bounds:

\begin{corollary}[Estimates on the difference quotients of the $\L^\infty$ envelope] \label{Corollary19A} In the setting of Theorem \ref{Theorem2} and under the same assumptions, additionally to \eqref{1.13} we also have the estimates 
\[
\left\{ \ \ \ 
\begin{split}
\frac{f(u+t\phi) - f(u)}{t} &\geq    \la \text{\!-} \underset{x\in \Om_\e(u)}{\ess\sup}\, \big(\p\F(\cdot,x)\big)_{u}(\phi)   -\frac{\e}{t} - \om_{u,\phi} \big( t\| \phi \| \big),
\\
\frac{f(u+t\phi) -f(u)}{t} &\leq  \la \text{\!-} \underset{x\in \Om_\epsilon (u)}{\ess\sup}\, \big(\p\F(\cdot,x)\big)_{u}(\phi) + \om_{u,\phi} \big(t\| \phi\| \big),
\end{split}
\right. \ \ \ \ \ 
\]
for any $\e,\epsilon,t>0$ with $\epsilon \geq 4\om_{u} \big(t\| \phi\| \big)$. Here $\om_{u,\phi},\om_{u}  \in \C(0,\infty)$ are the moduli of continuity from assumptions \eqref{1.15}-\eqref{1.16}.
\end{corollary}

Theorem \ref{Theorem2} allows to characterise the set of Gateaux differentiability points $u\in \mX$ of the generalised $\L^\infty$ envelope \eqref{1.12}. Given any $g \in \L^\infty(\Om,\la)$ and $A\in \mathfrak L$, we define the \emph{$\la$-essential oscillation of $g$ over $A$} as:
\beq
\label{3.6}
 {\la \text{-} \underset{A}{\ess\, \mathrm{osc}}\, g :=  \la \text{-} \underset{A}{\ess\sup}\, g - \la \text{-} \underset{A}{\ess\inf}\, g. } 
\eeq

\begin{corollary}[The set of Gateaux differentiability points of the generalised $\L^\infty$ envelope] In the setting of Theorem \ref{Theorem2} and under the same assumptions, the set of Gateaux differentiability points of the envelope function \eqref{1.12} is given by
\[
\mathcal D(f)=\bigcap_{\phi \in \mX} \bigg\{u\in \mX :  \lim_{\e \to 0^+} \! \bigg( \la \text{-} \underset{\Om_\e(u)}{\ess\,\mathrm{osc}}\, \big(\p\F(\cdot,x)\big)_{u}(\phi)\bigg)=0\bigg\},
\]
where the essential oscillation is given by \eqref{3.6}.
\end{corollary}
The result above is a direct consequence of Theorem \ref{Theorem2}. It suffices only to note that by a direction reversal in \eqref{1.13} the left semi-derivative is given by
\[
({\mfD} f)^-_u(\phi) = \lim_{\e \to 0^+} \! \bigg(\la \text{-\hspace{1pt}} \underset{x\in \Om_\e(u)}{\ess\inf}\, \big(\p \F (\cdot,x)\big)_u(\phi) \bigg),
\]
and $f$ is Gateaux differentiable when $({\mfD} f)^+_u =({\mfD} f)^-_u$ (recall the discussion about semi-differentials in Section \ref{Section 2}). Now we establish that if $\F$ satisfies some additional regularity, the conclusion of Theorem \ref{Theorem2} can be strengthened and the envelope \eqref{1.12} is also differentiable in the sense of Hadamard and in the sense of Penot (also referred to as $M$-semi-differentiability, recall \eqref{2.1}-\eqref{2.2} in Section \ref{Section 2}), with the same expression \eqref{1.13}-\eqref{1.14} for the semi-differential. This is significant, because it makes the general theory of semi-differential Calculus available in this context (see e.g.\ \cite{Cl, De1, De2, De3}), for example linearity, continuity and the chain rule.

\begin{proposition}[Semi-differentiability of the envelope in the Hadamard and Penot sense]
\label{Proposition11}
In the context of Theorem \ref{Theorem2} and under the same assumptions, suppose additionally that $\F : \mX \by \Om \larrow \R$ is (locally) Lipschitz continuous in $u$, $\la$-essentially uniformly in $x\in\Om$, namely for any $R>0$, there exists $L(R)>0$ such that
\beq
\label{3.10}
\la\text{-\hspace{1pt}}\underset{\Om}{\ess\sup}\, \Big| \F(u,\cdot) -  \F(v,\cdot)\Big| \leq L(R) \| u - v \|,
\eeq
when $u,v\in \mB_R \sub \mX$. Then, the envelope \eqref{1.12} is also Hadamard semi-differentiable as well as Penot semi-differentiable (recall \eqref{2.1}-\eqref{2.2}), and the corresponding stronger semi-differentials are  given by the representation formula \eqref{1.13}-\eqref{1.14}. Moreover, for any fixed $u\in \mX$, the semi-differential $(\mfD f)^+_u : \mX \larrow \R$ is a positively $1$-homogeneous, continuous and convex functional. 
\end{proposition}

As a consequence of the general theory of Hadamard semi-differential Calculus (see e.g.\ \cite{De1, De2, De3}), we have the following properties in addition those in Proposition \ref{Proposition11}: 

\begin{remark}[Linearity and chain rule]\label{Remark12} For any fixed $u,\phi \in \mX$, $f \mapsto (\mfD f)^+_u(\phi)$ is a linear functional. Further, the chain rule holds
\[
\big(\mfD (G\circ f)\big)^+_u (\phi) = (\mfD G)^+_{f(u)}  \big(  (\mfD f)^+_u(\phi)\big),
\]
 for any Hadamard semi-differentiable map $G : \mathfrak Y \larrow \mX$, where $( \mathfrak Y,\|\cdot\|)$ is a Banach space.
\end{remark}

\BPP \ref{Proposition11}. We first show that the envelope $f$ satisfies the same Lipschitz regularity property as $\F$. Indeed, by \eqref{3.10}, for any $R>0$ and $u,v\in \mB_R \sub \mX$, we have
\[
 \F(u,\cdot)  \leq  \F(v,\cdot) + L(R) \| u - v \|,
\]
for some $L(R)>0$, which implies
\[
\la\text{-\hspace{1pt}}\underset{\Om}{\ess\sup}\,  \F(u,\cdot)  \leq  \la\text{-\hspace{1pt}}\underset{\Om}{\ess\sup}\, \F(v,\cdot) + L(R) \| u - v \|.
\]
By swapping the roles of $u,v$ in the above, \eqref{1.12} allows us to infer that 
\beq
\label{3.11}
 \big| f(u) -  f(v)\big| \leq L(R) \| u - v \|,
\eeq
any $R>0$ and $u,v\in \mB_R$. Under \eqref{3.11}, the general theory of semi-differential Calculus allows to infer all the conclusions, but we can also easily prove them directly: for any $u,\phi,\psi \in \mB_R \sub \mX$ and $t >0$, by \eqref{3.11} we have
\[
\begin{split}
\bigg| \frac{f(u+t\psi)-f(u)}{t} - \frac{f(u+t\phi)-f(u)}{t} \bigg| & \leq  \frac{\big| f(u+t\psi)-f(u+t\phi)\big| }{t} 
\\
& \leq L(R) \|\phi -\psi\|,
\end{split}
\]
from which the conclusion follows by virtue of \eqref{2.1} and Theorem \ref{Theorem2}. Similarly, consider a path $u : [0,t_0) \larrow \mX $ for $t_0>0$, satisfying $u(0)=u$ and $u'(0^+)=\phi$. Then, by \eqref{3.11} we have for $t \in (0,t_0)$ and $u,\phi \in \mB_R \sub \mX$ that
\[
\begin{split}
\bigg| \frac{f(u(t))-f(u(0))}{t} - \frac{f(u+t\phi)-f(u)}{t} \bigg| & \leq  \frac{\big| f\big(u+t\phi + \mathrm o(t)_{t\to 0^+}\big)-f(u+t\phi)\big| }{t} 
\\
& \leq L(R) \frac{\| \mathrm o(t)_{t\to 0^+} \|}{t},
\end{split}
\]
from which the conclusion follows by utilising \eqref{2.2} and Theorem \ref{Theorem2}. Since for $\la$-a.e.\ $x\in \Om$ we have $ (\p \F(\cdot, x))_u \in \mX^*$, \eqref{1.13} together with the subadditivity and homogeneity properties of the essential supremum imply that
\beq 
\label{3.12}
\left\{ \ \ 
\begin{split}
(\mfD f)^+_u(\la \phi +\mu \psi) &\leq \la (\mfD f)^+_u(\phi) + \mu (\mfD f)^+_u (\psi), 
\\
 (\mfD f)^+_u(\la \phi) &=\la (\mfD f)^+_u(\phi), 
 \end{split}
 \right.
\eeq
for any $\la,\mu \geq 0$ and $u,\phi,\psi \in \mX$. Then, \eqref{3.12} implies the desired convexity and homogeneity of the $\L^\infty$ envelope $f$. Finally, appropriate choices of $\la,\mu,\phi,\psi$ in \eqref{3.12} easily yield after a few substitutions the double inequality
\[
-\mu (\mfD f)^+_u( - \psi) \leq  (\mfD f)^+_u(\phi +\mu \psi)  - (\mfD f)^+_u(\phi)  \leq \mu (\mfD f)^+_u( \psi), 
\]
for any $\la \geq 0$ and $u,\phi,\psi \in \mX$, from which the continuity of $(\mfD f)^+_u : \mX \larrow \R$ ensues, for any fixed $u\in\mX$.   \qed
\ms


\section{Measure-theoretic argmax set and approximate supremum}
\label{Section 4}

In this section we present some new measure-theoretic concepts and results of independent interest, which in particular will allow us to recast the representation formula  \eqref{1.13}-\eqref{1.14} of the semi-differential of the envelope \eqref{1.12} from Theorem \ref{Theorem2} in the form \eqref{1.17}. The latter representation aligns with the formulation of more classical Danskin-type results (which typically require continuity in the second variable and compactness of an underlying topological space). As mentioned in the introduction, the formulation \eqref{1.17} requires some more structure on behalf of the measure space involved (Borel regular metric measure spaces with the differentiation property, recall Section \ref{Section 2}), but this additional structure is always available in cases of interest, including applications to the Calculus of Variations in $\L^\infty$, presented in Section \ref{Section 6}. We begin with an auxiliary notion.

\begin{definition}[Limiting supremum] \label{Definition13} Let $(\Om,\mathfrak L, \la)$ be a measure space and $f : \Om \larrow \R$ a $\la$-measurable function. Let also  $(\mF_\e)_{\e>0} \sub \mathfrak L$ be a family of $\mathfrak L$-measurable sets of positive $\la$-measure, ordered by inclusion, namely $\la(\mF_\e)>0$ for $\e>0$, and $\mF_{\e''} \sub \mF_{\e'}$ when $0<\e''<\e'$. We define the $\la$-$\limsup$ of $f$ over the family $(\mF_\e)_{\e>0}$ as
\[
\la\text{-}\underset{ \vv{ \ (\mF_\e)_{\e>0}\ }}{\limsup} f := \lim_{\e\to0^+}\Big( \la\text{-}\underset{\mF_\e}{\ess \sup} f\Big).
\]
\end{definition}

The above concept is a generalisation of the essential limsup (see e.g.\ \cite{C1, K3.5}) from a shrinking family of balls $(\mB_\e(x))_{\e>0}$ in a Radon measure space, to a general shrinking family of $\mathfrak L$-measurable sets. Since $\e \mapsto \la\text{-} {\ess \sup}_{\mF_\e} f$ is monotone, the above limit always exists for any directed family $(\mF_\e)_{\e>0} \sub \mathfrak L$ as above (but could be $+\infty$ unless $f$ is $\la$-essentially bounded above). Even though at first glance it seems that the limit merely equals
\[
\la\text{-}\underset{\underset{\e>0}{\bigcap}\mF_\e}{\ess \sup} f,
\]
actually, this is not the case. Unless additional hypotheses are satisfied, in general this new object depends on the choice of  $(\mF_\e)_{\e>0}$, and not only on the intersection $\cap_{\e>0} \mF_\e$, retaining some memory of the limiting process, as the next simple example shows.

\begin{example} Let $(\Om,\mathfrak L, \la)=(\R,\mL(\R),\mL^1)$, and choose $f:= \chi_{(-\infty,-1)} + 2 \chi_{(1,\infty)}$ $ \in \L^\infty(\R)$. Consider the families of sets $(\mF'_\e)_{\e>0},(\mF''_\e)_{\e>0} \sub \mL(\R)$, where we take $\mF'_\e:=[-1-\e,1]$ and $\mF''_\e:=[-1,1+\e]$. Let also $\mF:=[-1,1]$. Then, we have
\[
 \mF = \underset{\e>0}{\bigcap}\mF'_\e = \underset{\e>0}{\bigcap}\mF''_\e, \ \ \text{ but }\ \   \underset{\mF}{\ess \sup} f <  \underset{ \vv{ \ (\mF'_\e)_{\e>0}\ }}{\limsup} f < \underset{ \vv{ \ (\mF_\e'')_{\e>0}\ }}{\limsup} f,
\]
because we have ${\ess \sup}_{{\mF}}f =0$, ${\ess \sup}_{\mF'_\e} f =1$, and ${\ess \sup}_{\mF''_\e} f =2$. (As it is customary, if $\la=\mL^n$ on $\R^n$ or any subsets thereof, the (Lebesgue) measure modifier will be dropped.)
\end{example}

Next we define an appropriate measure theoretic extension of the classical notion of argmax set \eqref{1.3} of pointwise defined continuous functions on compact sets. A priori, it is not at all obvious how one can do this sensibly. For instance, the intersection of superlevel sets of the form $\{f \geq {\ess \sup}_{{\Om}}f -\e\}$ for all $\e>0$ might be empty, and by taking topological closures of these set we might include too many extraneous points beyond the set we intend to capture. It turns out that by intersecting the measure-theoretic closures of these sets, we capture the correct object. 

For the remainder of this section, we fix a \emph{Borel regular metric measure space $(\Om,d,\mathfrak L, \la)$ with the differentiation property}, abbreviated as BMMSDP (recall Section \ref{Section 2}).

\begin{definition}[Essential argmax set] \label{Definition15} Let $(\Om,d,\mathfrak L, \la)$ be a BMMSDP and let $f : \Om \larrow \R$ be a $\la$-measurable function. We define the \emph{$\la$-essential argmax set of $f$ over $\Om$} as the following intersection of measure-theoretic closures of superlevel sets (recall \eqref{2.6A}-\eqref{2.7A}):
\[
\la \text{-} \Argmax \{ f : \Om \} := \bigcap_{\e>0} \overline{\Big\{f \geq \la \text{-} \underset{\Om}{\ess \sup}\, f -\e\Big\}}^\la.
\]
\end{definition}
Note that if $f$ is not $\la$-essentially bounded above, namely when $\la \text{-}{\ess \sup}_{\Om} f =\infty$, then the notion still makes sense, but $\la \text{-} \Argmax \{ f : \Om \} =\emptyset$. The following proposition lists the main properties of this new notion, and justifies its name as we show it coincides with the classical objects, when the necessary additional continuity and compactness properties are available. 

\begin{proposition}[Properties of the essential argmax set] \label{Proposition16}  Let $(\Om,d,\mathfrak L, \la)$ be a BMMSDP and let $f : \Om \larrow \R$ be a $\la$-measurable function, $\la$-essentially bounded above. Then:

\begin{enumerate}
\item The essential argmax set of $f$ over $\Om$ can be characterised in the following way:
\[
\la \text{-} \Argmax \{ f : \Om \} = \bigg\{ x\in \Om \ : \ \la\text{-}\mathrm{ap}\limsup_{y\to x} f(y) = \la \text{-} \underset{\Om}{\ess \sup}\, f \bigg\},
\]
and the approximate limsup above coincides with the corresponding approximate limit. In particular, if $f$ is $\la$-approximately continuous in an open neighbourhood of the set $\big\{f\geq \la\text{-}{\ess \sup}_{\Om} f -\e_0 \big\}$ for some $\e_0>0$, then the essential argmax set  is non-empty.

\item If $f \in \C(\Om)$ is continuous, then 
\[
\overline{\Big\{f \geq \la \text{-} \underset{\Om}{\ess \sup}\, f -\e\Big\}}^\la = \Big\{x\in \Om : f(x) \geq \underset{\Om}{\sup} \, f -\e\Big\},
\]
and both sets are closed. Further,
\[
\la \text{-} \Argmax \{ f : \Om \} = \Big\{ x\in \Om : f(x) = \sup_{\Om} f \Big\}.
\]
\item If $f \in \C(\Om)$ is continuous and $\Om$ is compact, then we additionally have 
\[
\la \text{-} \Argmax \{ f : \Om \} = \Argmax \{ f : \Om \},
\]
and both sets are non-empty.
\end{enumerate}
\end{proposition}

\BPP \ref{Proposition16}. \emph{(1)} By Definition \ref{Definition15}, we have that $x\in \la \text{-} \Argmax \{ f : \Om \} $ if and only if
\[
\forall\, \e>0, \ \ \ x \in \overline{\{f\geq M -\e \}}^\la, \ \ \ \ M:= \la \text{-} \underset{\Om}{\ess \sup}\, f .
\]
By  \eqref{2.6A}-\eqref{2.7A}, the above is equivalent to
\[
\forall\, \e>0, \ \ \ \limsup_{\rho \to0} \, \mD_\rho\Big(\la, \{f\geq M -\e \}, x \Big) >0.
\]
We can further rewrite the above as
\[
\forall\, t<M, \ \ \ \limsup_{\rho \to0} \, \mD_\rho\big(\la, \{f \geq t \}, x \big) >0.
\]
Note also that by the definition of $M$ as the essential supremum over $\Om$, for any $t > M$ we have that $\la\big(\mB_\rho(x) \cap \{f \geq t \} \big)=0$, which implies that
\[
\forall\, t>M, \ \ \ \limsup_{\rho \to0} \, \mD_\rho\big(\la, \{f \geq t \}, x \big) =0.
\]
The above imply that $M$ is optimal, namely
\[
M =\sup\bigg\{ t\in \R \ : \ \limsup_{\rho \to0} \, \mD_\rho\big(\la, \{f \geq t \}, x \big) >0 \bigg\}.
\]
By comparing the above with \eqref{2.9A}, in view of \eqref{2.6A} and the definition of $M$, we conclude that $x\in \la \text{-} \Argmax \{ f : \Om \} $ if and only if
\[
\la\text{-}\mathrm{ap}\limsup_{y\to x} f(y) = \la \text{-} \underset{\Om}{\ess \sup}\, f.
\]
Further, the approximate limsup above can be replaced by the corresponding approximate limit because $f \leq  \la \text{-}{\ess \sup}_{\Om} f$, $\la$-a.e.\ on $\Om$. Finally, if $f$ is approximately continuous in an open neighbourhood of the set $\big\{f\geq \la\text{-}{\ess \sup}_{\Om} f -\e_0 \big\}$ for some $\e_0>0$, by definition this means that the approximate limit exists at all points thereon. Therefore, $f$ is approximately continuous on $\la \text{-} \Argmax \{ f : \Om \} $. By the uniqueness property of approximate limits, we deduce that the essential argmax set is non-empty.

\smallskip

\noi \emph{(2)} Suppose $f\in \C(\Om)$, and fix $t\in \R$. Since the set $\{ f\geq t\}$ is closed in $\Om$, by \eqref{2.8A}, it follows that
\[
\overline{\{ f\geq t\}}^\la \sub \{ f\geq t\}.
\]
Further, let $\e>0$. Again by continuity, if $x\in \{ f\geq t+\e\} \sub \{ f >t \}$, there exists a $\rho>0$ such that $\mB_\rho(x) \sub \{ f > t\}\sub \{ f\geq t\}$. This implies
\[
\la\big( \mB_\rho(x) \cap \{ f\geq t\}\big) = \la\big( \mB_\rho(x)\big),
\]
whence $\mD_\rho\big(\la, \{ f\geq t\}, x\big) =1$. This yields that $x\in \overline{\{ f\geq t\}}^\la$. In conclusion, we have shown that if $f\in \C(\Om)$, then
\[
\ \ \ \ \{ f\geq t+\e\} \sub \overline{\{ f\geq t\}}^\la \sub \{ f\geq t\} , \ \ \ \forall\, \e>0, \ \forall\, t\in \R.
\]
In particular, it follows that $\overline{\{ f\geq t\}}^\la = \{ f\geq t\}$, when $f\in \C(\Om)$. By Remark \ref{Remark17}, we further infer that
\[
\overline{\Big\{f \geq \la \text{-} \underset{\Om}{\ess \sup}\, f -\e\Big\}}^\la = \Big\{ f  \geq \underset{\Om}{\sup} \, f -\e\Big\},
\]
and therefore both sets are closed. Finally, by the above and Definition \ref{Definition15}, we also have
\[
\begin{split}
\la \text{-} \Argmax \{ f : \Om \} &= \bigcap_{\e>0} \overline{\Big\{f \geq \la \text{-} \underset{\Om}{\ess \sup}\, f -\e\Big\}}^\la
\\
& = \bigcap_{\e>0} \Big\{x\in \Om : f(x) \geq  \underset{\Om}{\sup} \, f -\e \Big\}
\\
&= \Big\{x\in \Om : f(x) = \underset{\Om}{\sup} \,f \Big\}.
\end{split}
\]
\noi \emph{(3)} Follows from Part \emph{(2)} under the assumption of compactness on $\Om$, because if $f\in \C(\Om)$ we have that $\sup_\Om f =\max_\Om f$, and therefore
\[
\Big\{ x\in \Om : f(x) = \sup_{\Om}f \Big\} = \Argmax \{f : \Om\} \neq \emptyset.
\]
The proposition has been established.     \qed
\ms

We conclude this section by introducing our last measure-theoretic notion, which is combination of the previous two given in Definitions \ref{Definition13} and  \ref{Definition15}.

\begin{definition}[Approximate supremum over an essential argmax set] \label{Definition18} Let $(\Om,d,\mathfrak L, \la)$ be a BMMSDP. Given two $\la$-measurable functions $f,F : \Om \larrow \R$, where $f$ is $\la$-essentially bounded above, we define the \emph{approximate supremum of $F$ over the essential argmax set of $f$ over $\Om$} by setting
\[
\underset{ \la \text{-} \Argmax\{f : \,\Om \}}{\textrm{ap-\!}\sup} F \, := \la\text{-}\underset{ \vv{ \ (\Om_\e)_{\e>0}\ }}{\limsup} \, F ,
\]
where the directed by inclusion family of measurable set $(\Om_\e)_{\e>0} \sub \mathfrak L$ is given by
\[
\Om_\e := \Big\{ f \geq \la \text{-} \underset{\Om}{\ess \sup} \, f -\e\Big\}, \ \ \ \e>0.
\]
\end{definition}
\begin{remark}[Reformulations] In the light of Definitions \ref{Definition13} and  \ref{Definition15}, we can rewrite Definition \ref{Definition18} directly without reference to the notion of limiting supremum, as follows:
\beq
\label{4.1}
\underset{ \la \text{-} \Argmax\{f : \,\Om \} }{\textrm{ap-\!}\sup} F  = \lim_{\e\to0^+}\bigg( \la\text{-}\underset{  \{ f \geq \la \text{-} \underset{\Om}{\ess \sup}\, f -\e  \}}{\ess \sup} F\bigg).
\eeq
Note however that we couldn't have just taken \eqref{4.1} as a standalone definition, as we needed to define rigorously what it means to have a measure-theoretic argmax set, which was the content of Definition \ref{Definition15} (and Proposition \ref{Proposition16}). Further, since the measure-theoretic closure differs at most by a nullset from the set itself, \eqref{4.1} is also equivalent to
\[
\underset{ \la \text{-} \Argmax\{f : \,\Om \} }{\textrm{ap-\!}\sup} F  = \lim_{\e\to0^+}\Bigg( \la\text{-}\underset{  \overline{\{ f \geq \la \text{-} \underset{\Om}{\ess \sup}\, f -\e  \}}^\la}{\ess \sup} F\Bigg).
\]
\end{remark}
We close this section by the claimed reformulation of the conclusion of Theorem \ref{Theorem2} by using Definitions \ref{Definition13}, \ref{Definition15}, and \ref{Definition18}.

\begin{corollary}[Reformulation of the semi-differentiability of the generalised $\L^\infty$ envelope]  \label{Corollary19} In the context of  Theorem \ref{Theorem2} and under the same assumptions, suppose also that $(\Om,d,\mathfrak L, \la)$ is a BMMSPD, namely a Borel regular metric measure space with the differentiation property. Then, the Gateaux (right) semi-differential of the envelope function
\[
f(u) = \la \text{\!-} \underset{x\in \Om}{\ess\sup}\, \F(u,x), \ \ \ f \ : \ \mX \larrow \R,
\]
can be represented as
\[
(\mfD f)^+_u(\phi) = \underset{x\in  \la \text{-} \Argmax\{\F(u,\cdot) \,:\, \Om \} }{\mathrm{ap}\text{-\!}\sup} \, \big(\p \F (\cdot,x)\big)_u(\phi) ,
\]
for any $u,\phi \in \mX$.
\end{corollary}
 
 By applying Corollary \ref{Corollary19} to the expression \eqref{1.20} representing the semi-differential of supremal functionals (established in Theorem \ref{Theorem3}), we readily obtain the claimed reformulation \eqref{1.21}, by the properties of the Lebesgue measure on $\R^n$. Further, by utilising the definition \eqref{1.22} of the total semi-differential, we obtain the reformulation \eqref{1.23} as well.


\section{Semi-differentiability of general supremal functionals}
\label{Section 6}

In this section we apply Theorem \ref{Theorem2} to the particular case of supremal functionals.

\BPT \ref{Theorem3}. Consider the general supremal functional given by \eqref{1.6}, where now $\Om \sub \R^n$ is an open set and recall the notations $\D^{\vec k}u= \big(u,\D u, \ldots, \D^k u \big)$ and 
\[
\R^{\vec M} =  \mathbb R^N \times\mathbb R^{N \by n} \times \cdots  \by \mathbb R^{N \by n^{\ot k}}_{\mathrm s}.
\]
By assumption the Carath\'eodory function $\H : \Om \by \R^{\vec M} \larrow \R$ satisfies \eqref{1.18}-\eqref{1.19}. We set
\beq
\label{6.1}
\F \ : \ \W^{k,\infty}(\Om;\R^N) \by \Om \larrow \R, \ \ \ \ \F(u,x) := \H\big(x,\D^{\vec k}u(x)\big).
\eeq
Then, the envelope of $\F$ is the supremal functional $\E_\infty(\cdot,\Om)$:
\beq
\label{6.2}
\E_\infty(u,\Om) = \underset{x\in \Om}{\ess \sup}\,  \F(u,x) ,
\eeq
where the Banach space is $\W^{k,\infty}(\Om;\R^N)$, and the measure space is $(\Om,\mL(\Om),\mL^n)$.  To conclude, it suffices to show that the assumptions \eqref{1.18}-\eqref{1.19} on $\H$ imply that $\F$ defined by \eqref{6.1} satisfies \eqref{1.8}, \eqref{1.10} and \eqref{1.11}. To this aim, by using \eqref{1.18}-\eqref{1.19}, we estimate $\F$ from \eqref{6.1} as follows:
\[
\begin{split}
\| \F(u,\cdot)\|_{\L^\infty(\Om)} & = \big\|  \H\big(\cdot,\D^{\vec k}u\big) \big\|_{\L^\infty(\Om)}
\\
& \leq \|  \H\big(\cdot, \mathbf 0\big) \|_{\L^\infty(\Om)} + \om_{ 1+ \| \D^{\vec k}u  \|_{\L^\infty(\Om)}} \Big( \| \D^{\vec k}u  \|_{\L^\infty(\Om)} \Big),
\end{split}
\]
which implies $ \F(u,\cdot) \in \L^\infty(\Om)$ for any $u\in \W^{k,\infty}(\Om;\R^N)$. Further, note that for a.e.\ $x\in \Om$, the functional $\F(\cdot,x)$ is Gateaux differentiable at any point $u\in \W^{k,\infty}(\Om;\R^N)$, and the partial differential
\beq
\label{6.3}
\p \F(\cdot, x) \ : \ \W^{k,\infty}(\Om;\R^N) \larrow \big( \W^{k,\infty}(\Om;\R^N)\big)^*
\eeq
is given by
\beq
\label{6.4}
\big(\p \F(\cdot, x)\big)_u (\phi) =  \H_{,\bfX}\big(x,\D^{\vec k}u(x)\big) : \D^{\vec k}\phi(x),
\eeq
for any $\phi \in \W^{k,\infty}(\Om;\R^N)$. We now show that under our assumptions, $\F(\cdot, x)$ is actually differentiable in the Fr\'echet sense. Fix $u \in \W^{k,\infty}(\Om;\R^N)$ and take 
\[
R:= 1+ \big\| \D^{\vec k}u\big\|_{\L^\infty(\Om)}. 
\]
Then, for a.e.\ $x\in \Om$ and any $\phi \in \mB_1 \sub \W^{k,\infty}(\Om;\R^N)$, by \eqref{1.18} we may estimate
\[
\begin{split}
\Big| \F(u+\phi,x)- &\F(u,x)-\big(\p \F(\cdot, x)\big)_u (\phi)\Big| 
\\
&= \Big| \H\big(\cdot,\D^{\vec k}u+\D^{\vec k}\phi\big)- \H\big(\cdot,\D^{\vec k}u\big)-  
\H_{,\bfX}\big(\cdot,\D^{\vec k}u\big) : \D^{\vec k}\phi\Big|(x)
\\
&= \bigg| \int_0^1\Big( \H_{,\bfX}\big(\cdot, \D^{\vec k}u+\la \D^{\vec k}\phi\big): \D^{\vec k}\phi\Big)\,\mathrm d \la   -  
\H_{,\bfX}\big(\cdot,\D^{\vec k}u\big) : \D^{\vec k}\phi \bigg|(x)
\\
&\leq  \big\| \D^{\vec k}\phi\big\|_{\L^\infty(\Om)}  \int_0^1 \Big| \H_{,\bfX}\big(\cdot, \D^{\vec k}u+\la \D^{\vec k}\phi\big) -  
\H_{,\bfX}\big(\cdot,\D^{\vec k}u\big) \Big|(x) \,\mathrm d \la  
\\
& \leq  \big\| \D^{\vec k}\phi\big\|_{\L^\infty(\Om)}  \int_0^1\om_R \big(  \la\big\| \D^{\vec k}\phi\big\|_{\L^\infty(\Om)} \big) \,\mathrm d \la   
\\
& \leq  \om_R\big(  \big\| \D^{\vec k}\phi\big\|_{\L^\infty(\Om)} \big)   \big\| \D^{\vec k}\phi\big\|_{\L^\infty(\Om)}.
\end{split}
\]
The above estimate implies that \eqref{6.3}-\eqref{6.4} is actually the Fr\'echet differential of \eqref{6.1}. We now show that $\p \F(\cdot, x) : \W^{k,\infty}(\Om;\R^N) \larrow \big( \W^{k,\infty}(\Om;\R^N)\big)^*$  is actually a (strongly) continuous operator,  for a.e.\ $x\in \Om$. For any $w,v \in \mB_1 \sub \W^{k,\infty}(\Om;\R^N)$, by the representation of the dual operator norm and \eqref{6.4}, we have (using \eqref{1.18} for $R=2$)
\[
\begin{split}
\Big\| \big(\p \F(\cdot, x)\big)_w &-  \big(\p \F(\cdot, x)\big)_v \Big\|_{(\W^{k,\infty}(\Om;\R^N))^*} 
\\
&= \sup_{\phi \in \W^{k,\infty}(\Om;\R^N) : \| \D^{\vec k}\phi \|_{\L^\infty(\Om)} \leq 1} \bigg| \Big[ \big(\p \F(\cdot, x)\big)_w -  \big(\p \F(\cdot, x)\big)_v \Big] (\phi)\bigg|
\\
&= \sup_{\phi \in \W^{k,\infty}(\Om;\R^N) : \| \D^{\vec k}\phi \|_{\L^\infty(\Om)} \leq 1}\bigg|  
\Big[ \H_{,\bfX}\big(\cdot,\D^{\vec k} w \big)- \H_{,\bfX}\big(\cdot,\D^{\vec k}v\big) \Big]: \D^{\vec k} \phi  \bigg|(x)
\\
&\leq \om_2 \Big(  \big\| \D^{\vec k}w -\D^{\vec k}v \big\|_{\L^\infty(\Om)} \Big).
\end{split}
\]
The above estimate establishes that $\F(\cdot,x) \in \C^1\big(  \W^{k,\infty}(\Om;\R^N) \big)$, for a.e.\ $x\in \Om$. Hence, \eqref{1.8} has been established. Again by the definition of the dual operator norm, the above estimate implies additionally that for any $ \phi \in \W^{k,\infty}(\Om;\R^N)$, we have
\[
\begin{split}
\underset{x\in \Om}{\ess\sup}\, \bigg| \Big[ \big(\p \F(\cdot, x)\big)_w -  \big(\p \F(\cdot, x)\big)_v \Big] (\phi)\bigg|
\leq \om_2 \Big(  \big\| \D^{\vec k}w -\D^{\vec k}v \big\|_{\L^\infty(\Om)} \Big)  \| \D^{\vec k}\phi \|_{\L^\infty(\Om)},
\end{split}
\]
for all $w,v \in \mB_1 \sub \W^{k,\infty}(\Om;\R^N)$, which yields that $\F$ as defined by \eqref{6.1} also satisfies \eqref{1.10}. Finally, again by \eqref{1.18} (for $R=2$) we have
\[
\begin{split}
\underset{x\in \Om}{\ess\sup}\,  \big| \F(w,x) - \F(v,x) \big| & = \Big\|  \H\big(\cdot,\D^{\vec k}w\big) -\H\big(\cdot,\D^{\vec k}v\big) \Big\|_{\L^\infty(\Om)}
\\
& \leq \om_2 \Big( \big\| \D^{\vec k}w - \D^{\vec k} v  \big\|_{\L^\infty(\Om)} \Big),
\end{split}
\]
for all $w,v \in \mB_1 \sub \W^{k,\infty}(\Om;\R^N)$. This implies that $\F$ defined in \eqref{6.1} satisfies \eqref{1.11} as well. Therefore, all the hypotheses of Theorem \ref{Theorem2} are met by $\F$. Consequently, its envelope \eqref{6.2}, namely the supremal functional $\E_\infty(\cdot,\Om)$ (recall also \eqref{1.6}) is Gateaux semi-differentiable everywhere on $ \W^{k,\infty}(\Om;\R^N)$, with its semi-differential given by \eqref{1.13}-\eqref{1.14}. In view of \eqref{6.1}, \eqref{6.2}, \eqref{6.3} and \eqref{6.4}, the expression of the semi-differential of $F$ given by \eqref{1.13}-\eqref{1.14} is equivalent to the desired formula \eqref{1.20}. The theorem has been established. \qed
\ms

\begin{remark}[Reformulations of the semi-differential of $\E_\infty$] \
\smallskip 

\noi (i) By utilising the results of Section \ref{Section 4}, the regularity properties of the Lebesgue measure imply that the expression of the semi-differential \eqref{1.20} of the supremal functional \eqref{1.6} given by Theorem \ref{Theorem3}, can be recast in the form \eqref{1.21}, by using the new notions of approximate supremum and essential argmax set (Definitions \ref{Definition13}, \ref{Definition15} and \ref{Definition18}). 
\smallskip 

\noi (ii) Theorem \ref{Theorem3} is clearly valid also for $\E_\infty(\cdot,\mO)$ on every open subset $\mO \subseteq \Om$. This implies that the total semi-differential operator given by \eqref{1.22}, and therefore (by taking $\mO := \{\phi \neq 0\}$ when $\phi \in \smash{\W_0^{k,\infty}(\Om;\R^N)^\varobslash}$), the representation \eqref{1.23} follows. This global notion (in the sense that it refers to $\E_\infty$ without reference to subdomains) will be essential in Section \ref{Section 7} to establish the variational characterisation of absolute minimisers of supremal functionals.
\end{remark}

Utilising Proposition \ref{Proposition11} and Remark \ref{Remark12}, we can show that under a slightly stronger assumption, the semi-differentiability can be interpreted in the stronger sense of Hadamard and Penot. As we discussed earlier, this is significant, as it allows for general results from semi-differential Calculus to be directly applicable.

\begin{proposition}[Hadamard and Penot semi-differentiability of supremal functionals] \label{Proposition22}
In the context of Theorem \ref{Theorem3} and under the same assumptions, suppose additionally that
\beq
\label{6.5}
\big| \H_{,\bfX}(\cdot, \mathbf 0)\big| \in \L^\infty(\Om).
\eeq
Then, the supremal functional $\E_\infty(\cdot, \Om)$ (recall \eqref{1.6}) is Hadamard-semi-differentiable and also Penot-semi-differentiable, with the (right) semi-differential being given again by \eqref{1.20}. Further, the general properties of Hadamard semi-differential Calculus discussed in Remark \ref{Remark12}, are valid for supremal functionals.
\end{proposition}
\ms

\BPP \ref{Proposition22}. We need to confirm that by adding \eqref{6.5} to the list of assumptions for $\H$, then the functional $\F$ defined in \eqref{6.1} satisfies \eqref{3.10}. Then, the conclusion will be a consequence of Proposition \ref{Proposition11}. By \eqref{1.18}, we may estimate
\[
\begin{split}
 \Big|  \H\big(\cdot,\bfX''\big) -   \H\big(\cdot,\bfX''\big)  \Big| &\leq \bigg(\sup_{\bar \bfX \in \mB_R} \big| \H_{,\bfX}(\cdot, \bar\bfX)\big| \bigg) \big| \bfX'' - \bfX' \big|
 \\
 &\leq \bigg(\sup_{\bar \bfX \in \mB_R} \Big[\big| \H_{,\bfX}(\cdot, \mathbf 0)\big| + \om_R(|\bar \bfX |)\Big]\bigg) \big| \bfX'' - \bfX' \big|
\\
 &\leq \Big( \big| \H_{,\bfX}(\cdot, \mathbf 0)\big| + \om_R(R)\Big) \big| \bfX'' - \bfX' \big|,
\end{split}
\]
for any $\bfX'',\bfX' \in \mB_R \sub \R^{\vec M}$ and $R>0$, with the estimate being true a.e.\ on $\Om$. By \eqref{6.1} and \eqref{6.5},  the above allows us to infer that
\[
\begin{split}
\underset{x\in \Om}{\ess\sup}\,  \big| \F(w,x) - \F(v,x) \big| & = \Big\|  \H\big(\cdot,\D^{\vec k}w\big) -\H\big(\cdot,\D^{\vec k}v\big) \Big\|_{\L^\infty(\Om)}
\\
&\leq \Big( \big\| \H_{,\bfX}(\cdot, \mathbf 0)\big\|_{\L^\infty(\Om)} + \om_R(R)\Big) \big\|  \D^{\vec k}w -\D^{\vec k}v \big\|_{\L^\infty(\Om)},
\end{split}
\]
for any $w,v \in \mB_R \sub \W^{k,\infty}(\Om;\R^N)$. Hence, $\F$ satisfies \eqref{3.10} as well, and the desired conclusion ensues as a result of Proposition \ref{Proposition11}.  \qed
\ms


\section{Variational characterisation of minimality via semi-differentials}
\label{Section 7}

In this section we establish our final main result.

\BPT \ref{Theorem5}.  \underline{(i) $\Rightarrow$ (ii)}: Let $u\in \W^{k,\infty}(\Om;\R^N)$ be given and suppose that it is an absolute minimiser of \eqref{1.6}. By Proposition \ref{Proposition20}, this is equivalent to 
\[
\E_\infty \big(u,\{\psi \neq 0\}\big)\leq\E_\infty\big(u+\psi,\{\psi \neq 0\}\big), 
\]
for all $\psi \in \W^{k,\infty}_0(\Om;\R^N)^\varobslash$. Fix such a non-zero $\psi$ and $t>0$ and set $\mO := \{\psi \neq 0\} \sub \Om$, which is an open set. Then minimality implies that the lower right semi-differential is non-negative (recall the preliminaries in Section \ref{Section 2}):
\beq
\label{7.1}
\big(\underline{\mfD} \E_\infty(\cdot, \mO) \big)^+_u(\psi)  = \liminf_{t\to0^+} \frac{  \E_\infty (u+ t \psi, \mO) - \E_\infty (u,\mO)}{t}\geq 0. 
\eeq
Since assumptions \eqref{1.18}-\eqref{1.19} are satisfied by $\H$, by Theorem \ref{Theorem3} we have that $ \E_\infty (\cdot, \mO)$ is semi-differentiable, and the right semi-differential is expressed by formula \eqref{1.20}. In view of the material in Section \ref{Section 4}, especially Corollary \ref{Corollary19}, it is also expressed by \eqref{1.21}. Therefore, \eqref{7.1} implies 
\[
\big(\mfD \E_\infty(\cdot, \mO) \big)^+_u(\psi) \geq 0, 
\]
for any $\psi \in \W^{k,\infty}_0(\Om;\R^N) ^\varobslash$. Due to our choice of $\mO$, in view of \eqref{1.22}, we have
\[
(\mfD \E_\infty)^+_u(\psi) =\big(\mfD \E_\infty(\cdot, \{\psi \neq 0\}) \big)^+_u(\psi) = \big(\mfD \E_\infty(\cdot, \mO) \big)^+_u(\psi)  \geq 0.
\]
By \eqref{1.23}, the above inequality is equivalent to
\[
\underset{\mL^n\text{-}\Argmax \{  \H(\cdot, \D^{\vec{k}}u)\, : \, \{\psi\neq0 \}\}}{\mathrm{ap}\text{-\!}\sup} \! \Big[ \H_{,\bfX} \big(\cdot, \D^{\vec{k}}u \big) \!: \! \D^{\vec{k}}\psi \Big] \geq 0,
\]
for any $\psi \in \W^{k,\infty}_0(\Om;\R^N) ^\varobslash$. The conclusion ensues.

\ms

\noi \underline{(ii) $\Rightarrow$ (i)}: Unsurprisingly, this direction is considerably more complicated. We do however provide a simple direct proof under the additional assumption of convexity of $\H(x,\cdot)$ on $\smash{\R^{\vec M}}$ after the end of the present general proof.

\smallskip

\noi {\bf Step 1.} We begin with establishing two general properties of $\H$ stemming from the level-convexity and non-degeneracy assumptions \eqref{1.24}-\eqref{1.25}.

\begin{claim} \label{Claim26} For a.e.\ $x\in \Om$, we have the implication
\[
\H_{,\bfX}(x, \mathbf Z) \! :\! \mathbf W \geq 0 \ \ \ \Longrightarrow \ \ \ \H(x, \mathbf Z) \leq \H(x, \mathbf Z+ \mathbf W),
\]
for all $\mathbf Z, \mathbf W \in \R^{\vec M}$.
\end{claim}

\BPC \ref{Claim26}. Fix $x\in \Om$ such that $\H(x,\cdot)$ is defined, and fix also a $\mathbf W \in \R^{\vec M}$ satisfying $\H_{,\bfX}(x, \mathbf Z) \! :\! \mathbf W \geq 0$, and a point
\beq
\label{7.2}
\mathbf Z \in \R^{\vec M} \set \mathrm{Argmin} \big\{ \H(x,\cdot) : \R^{\vec M} \big\}. 
\eeq
The choice of $\mathbf Z$ in \eqref{7.2} implies that $\H(x,\mathbf Z) > \inf_{\R^{\vec M}}\H(x,\cdot) $, regardless of whether the argmin set is empty or not. Let $t\in \R$ be such that $\H(x,\mathbf Z)=t  $. By assumptions \eqref{1.24}-\eqref{1.25} and the choice of the point $\mathbf Z$, we have that
\[
\left\{ \ \ \
\begin{split}
(1)\ &\, \mathbf Z \in \big\{\H(x,\cdot)=t\big\} = \p \big\{\H(x,\cdot)<t \big\},  \phantom{\big|}
\\
(2)\ &\big\{\H(x,\cdot)<t \big\} \neq \emptyset  \text{ \& is open convex},  \phantom{\Big|}
\\
 (3)\ &\, \p \big\{\H(x,\cdot)<t \big\} \text{ is a $\C^1$ hypersurface},  \phantom{\big|}
 \\
(4)\ &\H_{,\bfX}(x, \cdot) \text{ is the outer normal vector}  \phantom{\Big|}
 \\
 &\text{field on the boundary }\p\big\{\H(x,\cdot)<t \big\} . 
  \end{split}
 \right. \ \ \
\]
The above imply that $\big\{\H(x,\cdot)<t \big\}$ is contained into the open half-space (See Figure 1):
\[
\mathfrak H_{\mathbf Z,x} : = \Big\{ \mathbf Y \in \R^{\vec M} \ : \ (\mathbf Y- \mathbf Z)\!:\! \H_{,\bfX}(x, \mathbf Z)  < 0\Big\} \sub \R^{\vec M} .
\]
\[
\underset{\text{Figure 1. The idea of the proof of Claim \ref{Claim26}.}}{\includegraphics[scale=0.2]{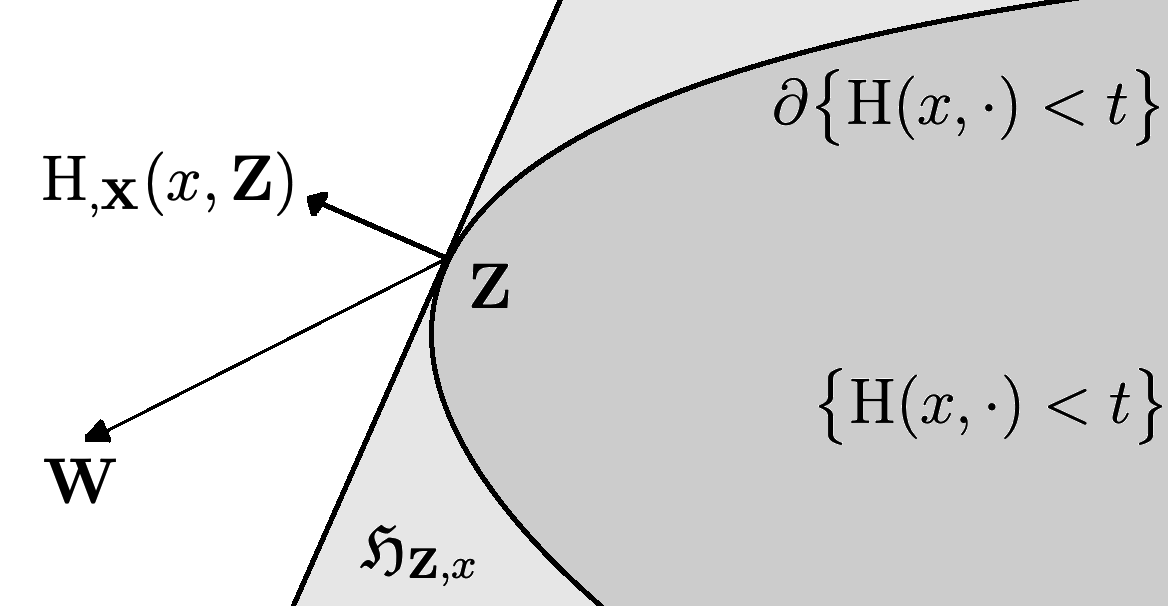}}\ \ \ \ \
\]
Thus, $\big\{\H(x,\cdot) \geq t \big\} \supseteq \R^{\vec M} \set \mathfrak H_{\mathbf Z,x}$, namely
\beq
\label{7.3}
\big\{\H(x,\cdot) \geq t \big\} \supseteq  \Big\{ \mathbf Y \in \R^{\vec M} \  : \ (\mathbf Y- \mathbf Z) \!:\! \H_{,\bfX}(x, \mathbf Z)  \geq 0\Big\}.
\eeq
Since by assumption $\mathbf W$ satisfies $\H_{,\bfX}(x, \mathbf Z) \! :\! \mathbf W \geq 0$, we have
\[
\H_{,\bfX}(x, \mathbf Z) \! :\! \big( (\mathbf W+\mathbf Z)-\mathbf Z \big) \geq 0.
\]
In view of \eqref{7.3}, the above inequality says that $\mathbf W+\mathbf Z \in \big\{\H(x,\cdot) \geq t \big\}$, which by recalling that $\H(x,\mathbf Z)=t  $, it yields
\[
\H(x, \mathbf Z) \leq \H(x, \mathbf Z+ \mathbf W).
\]
This establishes the desired implication, under assumption \eqref{7.2}. If instead we have that $\smash{\mathbf Z \in \mathrm{Argmin} \big\{ \H(x,\cdot) : \R^{\vec M} \big\}}$, then we arrive again at the same conclusion, as this membership implies that $\H(x, \mathbf Z) \leq \H(x, \mathbf U)$, for all $\smash{\mathbf U \in \R^{\vec M}}$.   \qed
\ms

Next we show an approximate less rigid version of the above result.

\begin{claim} \label{Claim27} If $x\in \Om$, $\e>0$ and $\mathbf Z, \mathbf W \in \R^{\vec M}$ are such that
\[
\H_{,\bfX}(x, \mathbf Z) \! :\! \mathbf W \geq -\e,
\]
then it follows that
\[
\H(x, \mathbf Z) \leq 
\left\{
\begin{array}{ll}
\H\bigg( x, \mathbf Z+ \mathbf W +\e\dfrac{\H_{,\bfX}(x, \mathbf Z) }{|\H_{,\bfX}(x, \mathbf Z) |^2}\bigg), & \mathbf Z \in \R^{\vec M} \set \mathrm{Argmin} \big\{ \H(x,\cdot) : \R^{\vec M} \big\},
\\
\H \big( x, \mathbf Z+ \mathbf W \big), & \mathbf Z \in \mathrm{Argmin} \big\{ \H(x,\cdot) : \R^{\vec M} \big\}.
\end{array}
\right.
\]
for all $\mathbf Z, \mathbf W \in \R^{\vec M}$.
\end{claim}

Note that, by our non-degeneracy assumption \eqref{1.24}, the above is well-defined.

\BPC \ref{Claim27}. Fix $x\in \Om$, $\e>0$ and $\mathbf Z, \mathbf W \in \R^{\vec M}$ as in the statement. By Claim \ref{Claim26}, it suffices to consider only the case that 
\[
\mathbf Z \in \R^{\vec M} \set \mathrm{Argmin} \big\{ \H(x,\cdot) : \R^{\vec M} \big\}.
\]
Since by assumption $\H_{,\bfX}(x, \mathbf Z) \! :\! \mathbf W \geq -\e$, which is equivalent to
\[
\H_{,\bfX}(x,\mathbf Z)\! :\! \bigg[ \mathbf W +\e\dfrac{\H_{,\bfX}(x, \mathbf Z) }{|\H_{,\bfX}(x, \mathbf Z) |^2}\bigg] \geq 0.
\]
The conclusion now follows by applying Claim \ref{Claim26} to the above inequality. \qed
\ms

\noi {\bf Step 2.} Now we turn to the main part of the proof. By assumption, we have
\[
\underset{\mL^n\text{-}\Argmax \{  \H(\cdot, \D^{\vec{k}}u)\, : \, \{\phi\neq0 \}\}}{\mathrm{ap}\text{-\!}\sup} \! \Big[ \H_{,\bfX} \big(\cdot, \D^{\vec{k}}u \big) \!: \! \D^{\vec{k}}\phi \Big] \geq 0,
\]
for any $\phi \in \W^{k,\infty}_0(\Om;\R^N)^\varobslash$. Fix such a non-zero $\phi$ and $0<\e<1$. In view of Definition \ref{Definition18} in Section \ref{Section 4}, this means that for any $0<\e<1$ we have
\beq
\label{7.4}
\underset{ \Om_{\e,\phi}(u) }{\ess \sup} \, \Big[ \H_{,\bfX} \big(\cdot, \D^{\vec{k}}u \big) \!: \! \D^{\vec{k}}\phi \Big] \geq 0,
\eeq
where
\beq
\label{7.5}
 \Om_{\e,\phi}(u) : = \Big\{ x\in \Om\, : \, \H \big(x, \D^{\vec{k}}u (x)\big) \geq \E_{\infty}\big(u,\{\phi\neq 0\}\big) -\e\Big\} \bigcap \{\phi\neq 0\}.
\eeq
By the definition of the essential supremum, there exists a measurable set $ \Theta_{\e,\phi}(u) \sub  \Om_{\e,\phi}(u)$ with positive measure $\mL^n ( \Theta_{\e,\phi}(u) )>0$, such that
\beq
\label{7.6}
\left\{ \ \ 
\begin{split}
& \H_{,\bfX} \big(\cdot, \D^{\vec{k}}u \big) \!: \! \D^{\vec{k}}\phi \geq -\e, & \text{ a.e.\ on }\Theta_{\e,\phi}(u),
 \\
& \H\big(\cdot, \D^{\vec{k}}u \big) \geq \E_{\infty}\big(u,\{\phi\neq 0\}\big) -\e, &  \text{ a.e.\ on }\Theta_{\e,\phi}(u).
\end{split}
\right.
\eeq
We now set
\beq
\label{7.7}
 \Om^\de : = \Big\{x\in \Om : \big| \H_{,\bfX}\big(x, \D^{\vec{k}}u (x)\big) \big| > \de\Big\} , \ \ \ \de \geq  0. 
\eeq
By \eqref{7.7}, we have that $(\Om_\de)_{\de>0} \sub \mL(\Om)$ and $\Om^\de \nearrow \Om^0$, as $\de\to0^+$. Further
\beq
\label{6.8A}
\Om \set\Om^0 =  \Big\{x\in \Om : \big| \H_{,\bfX}\big(x, \D^{\vec{k}}u (x)\big) \big| =0\Big\}
\eeq
and we can decompose 
\[
\Theta_{\e,\phi}(u) = \big(\Theta_{\e,\phi}(u) \set \Om^0 \big) \biguplus \big(\Theta_{\e,\phi}(u) \cap \Om^0 \big) .
\]
Since $\mL^n( \Theta_{\e,\phi}(u) )>0$, the disjointness of the union above implies that
\[
\mL^n\big(\Theta_{\e,\phi}(u) \set \Om^0 \big) + \mL^n\big(\Theta_{\e,\phi}(u) \cap \Om^0 \big)>0. 
\]
Hence, at least one of the sets $\Theta_{\e,\phi}(u) \set \Om^0$ and $\Theta_{\e,\phi}(u) \cap \Om^0$ above has non-zero measure.

\ms

\noi {\bf \underline{Case 1}:} $\mL^n\big(\Theta_{\e,\phi}(u) \set \Om^0 \big)>0$. By \eqref{1.24}, for a.e.\ $x\in \Om$ we have that
\beq
\label{6.9A}
\Big\{\mathbf Z \in \R^{\vec M}: \big| \H_{,\bfX}\big(x, \mathbf Z \big) \big| =0\Big\} \sub \mathrm{Argmin} \big\{ \H(x, \cdot )  : \R^{\vec M}\big\}.
\eeq
By \eqref{6.8A}-\eqref{6.9A}, it follows that for a.e.\ $x\in \Theta_{\e,\phi}(u) \set \Om^0$, we have that
\[
\D^{\vec k}u(x) \in \mathrm{Argmin} \big\{ \H(x, \cdot )  : \R^{\vec M}\big\}.
\]
This above therefore implies
\beq
\label{6.10A}
\H\big( \cdot, \D^{\vec{k}}u \big) \leq \H \big( \cdot, \D^{\vec{k}}u + \D^{\vec{k}}\phi \big), \ \ \text{ a.e.\ on } \Theta_{\e,\phi}(u) \set \Om^0.
\eeq
By \eqref{6.10A} and \eqref{7.6} and recalling also that $\Theta_{\e,\phi}(u) \sub \{\phi \neq0\}$, we have 
\[
\begin{split}
 \E_{\infty}\big(u,\{\phi\neq 0\}\big) &\leq \e + \underset{\Theta_{\e,\phi}(u) \set \Om^0} {\ess\sup} \H\big(\cdot, \D^{\vec{k}}u \big) 
\\
&\leq \e + \underset{\Theta_{\e,\phi}(u) \set \Om^0} {\ess\sup} \H \big( \cdot, \D^{\vec{k}}u + \D^{\vec{k}}\phi \big)
\\
&\leq \e + \underset{\{\phi \neq0\}} {\ess\sup}\, \H \big( \cdot, \D^{\vec{k}}u + \D^{\vec{k}}\phi \big)
\\
& = \e + \E_{\infty}\big(u +\phi,\{\phi\neq 0\}\big) ,
\end{split}
\]
for any $0<\e<1$, and any $\phi \in \W^{k,\infty}_0(\Om;\R^N)^\varobslash$. Therefore, in view of Proposition \ref{Proposition20}, by letting $\e\to0^+$ we see that $u$ is an absolute minimiser in this case.

\ms

\noi {\bf \underline{Case 2}:} $\mL^n\big(\Theta_{\e,\phi}(u) \cap \Om^0 \big)>0$. In view of \eqref{7.7},  $(\Om_\de)_{\de>0} \sub \mL(\Om)$ and we have $\Om^\de \nearrow \Om^0$, as $\de\to0^+$. Hence, by the continuity of the Lebesgue measure, we obtain
\[
\lim_{\de\to0^+} \mL^n\big(\Om^\de \cap \Theta_{\e,\phi}(u) \big) = \mL^n\big( \Om^0 \cap \Theta_{\e,\phi}(u) \big)>0 .
\]
This implies that there exists $\de_0>0$ such that
\beq
\label{7.8}
\mL^n\big(\Om^\de \cap \Theta_{\e,\phi}(u) \big) >0 , \ \ \text{ for all } \de \in (0,\de_0).
\eeq
By \eqref{7.6} and \eqref{7.8} we deduce that
\beq
\label{7.9}
\left\{ \ \ 
\begin{split}
& \H_{,\bfX} \big(\cdot, \D^{\vec{k}}u \big) \!: \! \D^{\vec{k}}\phi \geq -\e, & \text{ a.e.\ on }\Om^\de\cap \Theta_{\e,\phi}(u),
 \\
& \H\big(\cdot, \D^{\vec{k}}u \big) \geq \E_{\infty}\big(u,\{\phi\neq 0\}\big) -\e, &  \text{ a.e.\ on }\Om^\de\cap \Theta_{\e,\phi}(u).
\end{split}
\right.
\eeq
By \eqref{7.7} and assumption \eqref{1.24}, for any $\de \in (0,\de_0)$ we have that
\beq
\label{7.10}
 \D^{\vec{k}}u(x) \in  \R^{\vec M} \set \mathrm{Argmin} \big\{ \H(x,\cdot) : \R^{\vec M} \big\}, \ \ \text{ a.e.\ }x\in \Om^\de.
\eeq
By \eqref{7.9} and \eqref{7.10}, we see that $ \D^{\vec{k}}u(x),  \D^{\vec{k}}\phi(x) \in \R^{\vec M} $ satisfy the assumptions of Claim \ref{Claim27} for the given $\e \in (0,1)$ and for a.e.\ $x\in \Om^\de\cap \Theta_{\e,\phi}(u)$. We therefore have
\beq
\label{7.11}
 \H\big(\cdot, \D^{\vec{k}}u \big) \leq \H\bigg( \cdot ,\D^{\vec{k}}u+\D^{\vec{k}}\phi+\e \dfrac{\H_{,\bfX}(\cdot, \D^{\vec{k}}u) }{|\H_{,\bfX}(\cdot, \D^{\vec{k}}u) |^2}\bigg), \ \ \text{ a.e.\ on } \Om^\de\cap \Theta_{\e,\phi}(u).
\eeq
Additionally, by \eqref{7.7} we have the bound
\beq
\label{7.12}
\bigg\| \dfrac{\H_{,\bfX}(\cdot, \D^{\vec{k}}u) }{|\H_{,\bfX}(\cdot, \D^{\vec{k}}u) |^2} \bigg\|_{\L^\infty(\Om_\de)} \leq \frac{1}{\de}.
\eeq
By  \eqref{7.9} and \eqref{7.11} (recall also  \eqref{7.5} and that $\Theta_{\e,\phi}(u) \sub \{\phi\neq0\}$), we infer that
\beq
\label{7.13}
\begin{split}
\E_{\infty}\big(u,\{\phi\neq 0\}\big) & \leq \e + \underset{ \Om^\de\cap \Theta_{\e,\phi}(u) }{\ess \sup} \, \H\big(\cdot, \D^{\vec{k}}u \big)
\\
& \leq \e + \underset{ \Om^\de\cap \{\phi\neq 0\} }{\ess \sup} \,   \H\bigg( \cdot ,\D^{\vec{k}}u+\D^{\vec{k}}\phi+\e \dfrac{\H_{,\bfX}(\cdot, \D^{\vec{k}}u) }{|\H_{,\bfX}(\cdot, \D^{\vec{k}}u) |^2}\bigg).
\end{split}
\eeq
By assumption \eqref{1.18} and the bound \eqref{7.12}, \eqref{7.13} yields for any $R\geq R(\de)$, where we set
\[
R(\de) :=\big\| \D^{\vec{k}}u+\D^{\vec{k}}\phi  \big\|_{\L^\infty(\{\phi\neq 0\} )} + \frac{1}{\de} ,
\]
the estimate
\beq
\label{7.14}
\begin{split}
\E_{\infty}\big(u,\{\phi\neq 0\}\big) & \leq \e + \underset{ \Om^\de\cap \{\phi\neq 0\} }{\ess \sup} \,   \H\big( \cdot ,\D^{\vec{k}}u+\D^{\vec{k}}\phi \big) + \om_R\Bigg( \! \e \bigg\| \dfrac{\H_{,\bfX}(\cdot, \D^{\vec{k}}u) }{|\H_{,\bfX}(\cdot, \D^{\vec{k}}u) |^2} \bigg\|_{\L^\infty(\Om_\de)} \Bigg)
\\
& \leq \e + \underset{ \{\phi\neq 0\} }{\ess \sup} \, \H\big( \cdot ,\D^{\vec{k}}u+\D^{\vec{k}}\phi \big) + \om_R \Big( \frac{\e}{\de} \Big)
\\
&=  \e + \E_{\infty}\big(u +\phi ,\{\phi\neq 0\}\big) + \om_R \Big( \frac{\e}{\de} \Big).
\end{split}
\eeq
By letting $\e\to0^+$ in \eqref{7.14} for $\de \in (0,\de_0)$ fixed, we deduce that
\[
\E_{\infty}\big(u,\{\phi\neq 0\}\big)  \leq \E_{\infty}\big(u +\phi ,\{\phi\neq 0\}\big) ,
\]
for any given $\phi \in \W^{k,\infty}_0(\Om;\R^N)^\varobslash$. In view of Proposition \eqref{Proposition20}, it follows that $u$ is an absolute minimiser, as desired.
\qed
\ms

\begin{example}[About level-convexity and non-degeneracy] \label{Example28} As this example shows, the cornerstone of the proof in which both assumptions \eqref{1.24}-\eqref{1.25} are used essentially, Claim \ref{Claim26}, fails if \eqref{1.24} is violated even at a single point outside the argmin set. Let $\H :\R \larrow [0,\infty)$ be any function that satisfies $\H(X)>\H(0)=0$ for $X\neq 0$, with $\H$ strictly decreasing on $(-\infty,0]$ and  strictly increasing on $[0,\infty)$, but such that $\H'(X_0)=0$ at sone $X_0>0$. Then $\H$ is level convex on $\R$ and satisfies $\H'(X_0)W=0$ for all $W\in\R$, but 
\[
\text{$\H(X_0) > \H(X_0 +W)$, \ \ for\ \ $W:= -X_0$}.
\]
\end{example}

\noi {\bf A special proof of Theorem \ref{Theorem5} (ii) $\Rightarrow$ (i) }[\emph{Under additional convexity assumptions on $\H$}].

\smallskip

\noi In the context of Theorem \ref{Theorem5}  and under the same assumptions, suppose additionally  that $\H(x,\cdot)$ is convex on $\smash{\R^{\vec M}}$, for a.e.\ $x\in \Om$. In this case we can fashion out a simple direct proof of the sufficiency direction (ii) $\Rightarrow$ (i). To this aim, suppose that
\[
\underset{\mL^n\text{-}\Argmax \{  \H(\cdot, \D^{\vec{k}}u)\, : \, \{\psi\neq0 \}\}}{\mathrm{ap}\text{-\!}\sup} \! \Big[ \H_{,\bfX} \big(\cdot, \D^{\vec{k}}u \big) \!: \! \D^{\vec{k}}\psi \Big] \geq 0,
\]
for any $\psi \in \W^{k,\infty}_0(\Om;\R^N)^\varobslash$. Fix such a non-zero $\psi$; by Theorem \ref{Theorem3} and the definition of the total semi-differential \eqref{1.22}-\eqref{1.23}, this implies
\[
\lim_{t\to0^+} \frac{  \E_\infty \big(u+ t \psi, \{\psi \neq 0\} \big) - \E_\infty \big(u, \{\psi \neq 0\} \big) }{t} = \big(\mfD \E_\infty(\cdot, \{\psi \neq 0\}) \big)^+_u(\psi)  \geq 0. 
\]
Let us define $h : \R \larrow \R$ by setting
\[
h(t) : = \E_\infty \big(u+ t \psi, \{\psi \neq 0\} \big) - \E_\infty \big(u, \{\psi \neq 0\} \big).
\]
Then, $h$ is convex on $\R$, $h(0)=0$, and we also have $h'(0^+) \geq 0 \geq h'(0^-)$ (the last inequality follows by replacing $\psi$ with $-\psi$ in the semi-differential inequality). The conclusion ensues by Lemma \ref{Lemma26} that follows and Proposition \ref{Proposition20}.    \qed
\ms

\begin{lemma} \label{Lemma26} Let  $ : \R \larrow \R$ be a convex function, and suppose   $h'(0^-) \leq 0 \leq h'(0^+)$. Then, we have that $h(t)\geq h(0)$, for all $t\in \R$.
\end{lemma}

\BPL \ref{Lemma26}. Since $h$ is convex, for all $t,s\in \R$ and $\la\in(0,1)$, we have
\[
h(\la t + (1-\la)s) \leq \la h(t) + (1-\la) h(s).
\]
For $s=0$, this gives
\[
t \bigg[ \frac{h( \la t ) - h(0)}{\la t} \bigg] \leq h(t) - h(0), \ \ \text{ for all }t\neq 0.
\]
By letting $\la\to 0^+$ for $t>0$ yields $t h'(0^+) \leq h(t) - h(0)$. By letting $\la\to 0^+$ for $t<0$ yields $t h'(0^-) \leq h(t) - h(0)$. Since by assumption $h'(0^-) \leq 0 \leq h'(0^+)$, we deduce that $h(t)\geq h(0)$, for all $t\in \R$.
\qed
\ms

Now we state a variant of Theorem \ref{Theorem5} involving general classes of variations, instead of absolute minimisers. This result might be useful for extensions of the theory to constrained or other type of problems, and has essentially the same proof as Theorem \ref{Theorem5} (merely replace non-zero sets $\{\psi \neq 0\}$ of variations by arbitrary open subset $\mO\sub \Om$, as we are not requiring that these variations preserve the boundary values on $\p\mO$).

\begin{theorem}[Variational characterisation of general minimisers via the semi-differential]  \label{TheoremXX} In the context of Theorem \ref{Theorem3} and under the same assumption \eqref{1.18}-\eqref{1.19}, let $\mO \sub \Om$ be open, and let also $\mathfrak C \sub \W^{k,\infty}(\mO;\R^N)$ be a closed subspace. Given  $u \in \mathfrak C$ fixed, consider the following statements:
\begin{itemize}

\item[\emph{(i)}] The mapping $u \in \mathfrak C $ is a minimiser of the functional $\E_\infty(\cdot, \mO)$ (given by  \eqref{1.6}) within the class $\mathfrak C $, namely
\[
\E_\infty(u,\mO)\leq\E_\infty(u,\mO), \text{ for all }  v \in \mathfrak C .
\]
\item[\emph{(ii)}] The semi-differential of $\E_\infty(\cdot, \mO)$ (recall \eqref{1.20}-\eqref{1.21}) is non-negative within $\mathfrak C $ when evaluated at $u$: 
\[
\big(\mfD \E_\infty(\cdot, \mO) \big)^+_u(u-v)  = \underset{\mL^n\text{-}\Argmax \{  \H(\cdot, \D^{\vec{k}}u)\, : \, \mO\}}{\mathrm{ap}\text{-}\sup} \! \Big[ \H_{,\bfX} \big(\cdot, \D^{\vec{k}}u \big) \!: \! \D^{\vec{k}}(u-v) \Big] \geq 0,
\]
for any $v \in  \mathfrak C $.
\end{itemize}
Then, we have that \emph{(i)} implies \emph{(ii)}. Conversely, \emph{(ii)} implies \emph{(i)} if in addition the supremand $\H$ is non-degenerate and level-convex, namely \eqref{1.24}-\eqref{1.25} hold true.
\end{theorem}

\begin{remark}[Derivation of semi-differentials from the Euler-Lagrange PDEs in $\L^p$ as $p\to\infty$]

Even though a priori it is not obvious, it is indeed possible to derive the non-negativity of the semi-differential of $\E_\infty$ asserted in Theorems \ref{Theorem3} and \ref{Theorem5} from the Euler-Lagrange equations of the corresponding $\L^p$ functional as $p\to\infty$. For the sake of illustration we provide a \emph{formal argument}. Consider the $k$-th order functional
\[
{\E_p(u,\Om) := \left(\, \av_\Om  \H\big(\D^{\vec k} u \big) ^p\,\mathrm d\mL^n \! \right)^{\!\!1/p} },
\]
placed in $\W^{k,p}(\Om;\R^N)$, where we assume $\Om\sub \R^n$ is open bounded, and $\H \geq 0$. Let $u_p$ be a global minimiser for some Dirichlet data on $\p\Om$. For a fixed $\phi \in \W^{k,\infty}_0(\Om;\R^N)^\varobslash$ and $\e>0$, we set
\[
\Om_{\e,\phi}(u): = \{\phi \neq 0\} \bigcap \Big\{  \H\big(\D^{\vec k} u \big) \geq \E_{\infty}(u, \{\phi \neq 0\}) -\e \Big\}.
\]
Then, the weak form of the Euler-Lagrange equations reads
\[
\int_{\{\phi \neq 0\} }  \H\big(\D^{\vec k} u_p \big)^{p-1} \H_{,\bfX}\big(\D^{\vec k} u_p \big)   \!:\!   \D^{\vec k} \phi =0.
\]
Suppose $u_p \weak u_\infty$ in $\W^{k,q}(\Om;\R^N)$ as $p\to\infty$ for all $q\in(1,\infty)$, and $\E_{\infty}(u_\infty, \{\phi \neq 0\})>0$, which gives that $\E_{\infty}(u_p, \{\phi \neq 0\})>0$ for large $p$. From the above, by using the definition of $\Om_{\e,\phi}(u)$ for $\e>0$ small enough and rescaling, we rewrite as
\[ 
\begin{split}
0 &= \int_{\Om_{\e,\phi}(u_p)} \bigg( \frac{\H\big(\D^{\vec k} u_p \big) }{\E_{\infty}(u_p, \{\phi \neq 0\})}\bigg)^{p-1} \H_{,\bfX}\big(\D^{\vec k} u_p \big)  \!:\!   \D^{\vec k} \phi \,\mathrm d\mL^n
\\
& \ \ \  + \int_{\{\phi \neq 0\} \set \Om_{\e,\phi}(u_p)}  \bigg( \frac{\H\big(\D^{\vec k} u_p \big)}{\E_{\infty}(u_p, \{\phi \neq 0\})}\bigg)^{p-1}  \H_{,\bfX}\big(\D^{\vec k} u_p \big)  \!:\!   \D^{\vec k} \phi \,\mathrm d\mL^n.
\end{split}
\]
This gives the estimate
\[ 
\begin{split}
0 &\leq \bigg( \underset{\Om_{\e,\phi}(u_p)}{\ess\sup}  \, \H_{,\bfX}\big(\D^{\vec k} u_p \big) \!:\!   \D^{\vec k} \phi \bigg) \int_{\Om_{\e,\phi}(u_p)} \bigg(\frac{\H\big(\D^{\vec k} u_p \big)}{\E_{\infty}(u_p, \{\phi \neq 0\})}\bigg)^{p-1} \,\mathrm d\mL^n
\\
& \ \ \  +   \bigg( 1-\frac{ \e}{\E_{\infty}(u_p, \{\phi \neq 0\})} \bigg)^{p-1} \int_{\{\phi \neq 0\} } \Big|\H_{,\bfX}\big(\D^{\vec k} u_p \big)   \!:\!   \D^{\vec k} \phi \Big| \,\mathrm d\mL^n,
\end{split}
\]
from where, formally, we deduce as $p\to\infty$ that
\[
\underset{\Om_{\e,\phi}(u_\infty)}{\ess\sup}   \, \H_{,\bfX}\big(\D^{\vec k} u_\infty \big) \!:\!   \D^{\vec k} \phi  \geq 0,
\]
for $\e>0$ sufficiently small and $\phi \in \W^{k,\infty}_0(\Om;\R^N)^\varobslash$, which is equivalent to $(\mfD \E_\infty)^+_u(\phi) \geq 0$.
\end{remark}


\section{Relation of the semi-differential to the Aronsson equations}
\label{Section 8}

In this section we provide a direct proof that non-negativity of the semi-differential of supremal functionals (which by Theorem \ref{Theorem5} follows from absolute minimality, and characterises it if the functional is level-convex), implies the satisfaction of a conventional Aronsson-type equation. For simplicity, we restrict ourselves to the general scalar-valued ($k$-th order) case, as this removes the additional difficulties presented in the vector-valued case with the additional orthogonal PDE systems arising, which have discontinuous coefficients even for smooth solutions (see e.g.\ \cite{K7, K8, K9, K10, KS} and Section \ref{Section 2}). Additionally, to avoid technical nuances arsing from concepts of generalised solutions for Aronsson-type systems (see e.g.\ \cite{CKP, K4}), we will restrict our attention to the smooth case. 

\begin{theorem}[Derivation of the Aronsson-type $k$-th order PDE from the semi-differential] \label{Theorem26}
Consider the supremal functional \eqref{1.6} for $\Om \sub \R^n$ open, and suppose that $N=1$ and that 
\beq
\label{8.1}
\H \in C^1\Big(\Om \by \R \by \R^n \by \R^{n^{\ot 2}}_{\mathrm s}\cdots  \by \R^{n^{\ot k}}_{\mathrm s}\Big),
\eeq
where $n,k\in\N$, $n\geq 2$. Let $u\in \C^{k+1}(\Om)$, and suppose that the total semi-differential $( \mfD \E_\infty)^+_u$ of \eqref{1.6}  at $u$ $($given by \eqref{1.22}-\eqref{1.23}$)$ is non-negative on the subspace of compactly supported variations:
\beq
\label{8.2}
( \mfD \E_\infty)^+_u(\phi) \geq 0, \ \ \ \forall \ \phi \in \W^{k,\infty}_c(\Om).
\eeq
Then, $u$ is a solution to the following Aronsson-type $(k+1)$-th order fully nonlinear PDE:
\beq
\label{8.3}
\H_{,\bfX_k} \big(\cdot, \D^{\vec k}u \big) \!:\! \Big(\D \big( \H\big(\cdot, \D^{\vec k}u \big)\big)\Big)^{\!\ot k} =0, \ \ \ \text{ in }\Om.
\eeq
\end{theorem}

Recall from Section \ref{Section 2} that $\bfX_j$ is the variable corresponding to the $\R^{n^{\ot j}}_{\mathrm s}$-argument for the $j$-th order derivative, and ``$(\cdot)^{\ot j}$" is the $j$-th order symmetric tensor product, $j=0,\ldots,k$. The following examples, taken from Yu \cite[p.\ 175-176]{Yu}, show that the opposite direction above is not true even if $n=k=1$ and the solution to the Aronsson ODE is $\C^\infty$ smooth. The first one assumes level convexity in $u'$ but has $x$-dependence, and the second one assumes joint convexity in $(u,u')$, and in both cases the supremand is smooth. However, under these assumptions, by Theorem \ref{Theorem5} we have that (absolute) minimisers are characterised by the non-negativity of the semi-differential.

\begin{example}[The Aronsson PDE in general not sufficient for minimality in $\L^\infty$, cf.\ Yu \cite{Yu}] \,
\label{Example36} 
\begin{enumerate}
\item Suppose $\H : (0,1) \by \R \larrow \R$ is given by $\H(x,X) := H(X) + h(x)$, where $H \in \C^\infty(\R)$ is the level-convex function
\[
H(X) := (X^2-3X)^3,
\]
and $h \in \C^\infty[0,1]$ is any function satisfying 
\[
\text{$\max_{[0,1]}h=\max_{[0,1/2]}h=0$, \ \ $\max_{[1/2,1]}h=-100$.}
\]
Then, $u \equiv 0$ is a $\C^\infty[0,1]$ smooth solution to the Aronsson equation \eqref{8.3}, which reads
\[
\H_X(\cdot,u') \big(\H(\cdot,u')\big)'=0, \ \ \text{ in } (0,1),
\]
that is not a minimiser of the supremal functional $u\mapsto \underset{(0,1)}{\ess \sup} \,\H(\cdot,u')$ in $\W^{1,\infty}_0(0,1)$.

\item Suppose $\H : \R \by \R \larrow \R$ is the convex function $\H(X_0,X_1):=|X_1|^2-X_0$. Then, $u\in \C^\infty[0,2]$ given by 
\[
 u(x):= \frac{1}{4}(x-1)^2, \ \ \ x\in (0,2),
 \]
is a smooth solution to the Aronsson equation \eqref{8.3}, which reads
\[
\H_X(u,u') \big(\H(u,u')\big)'=0, \ \ \text{ in } (0,2),
\]
whilst additionally it satisfies that $\H(u,u')\equiv 0$ everywhere on $(0,2)$, but $u$ is not a minimiser of the supremal functional $u\mapsto \underset{(0,2)}{\ess \sup} \,\H(u,u')$ in $\W^{1,\infty}_{1/4}(0,2)$.
\end{enumerate}
\end{example}

\begin{remark} \label{remark28}  (i) Special cases of Theorem \ref{Theorem26} have previously been established in  \cite{K10, DK, KP} in the first and second order case, but except for the more general context, this result also has a streamlined proof. The latter contains two crucial ideas: firstly, every arbitrary open set satisfies interior sphere conditions at a dense set of boundary points (without any regularity hypotheses), and secondly even though supremal functionals are typically non-differentiable, when restricted to the $\C^k$ subspace, one can select \emph{special subsets} $\mO \Subset \Om$ such that $\E_\infty(\cdot,\mO)$ actually \emph{is differentiable}. (Interestingly this happens only for $n\geq 2$. For $n=1$, Example \ref{Example23} shows that this is not possible and we always have non-differentiability!). The former fact is based on viscosity solution type ideas of ``touching sets", whilst the latter is not really essential for the method of proof to work, but it is interesting nonetheless, shortening the argument. A slight reworking of the method as in \cite{DK} works for all $n\in \N$, but the one-dimensional case is not especially interesting anyway.

\smallskip

\noi (ii) Assumption \eqref{8.1} implies that \eqref{1.18}-\eqref{1.19} are valid on any $\mO \Subset \Om$, and $\E_\infty(\cdot, \mO)$ is semi-differentiable in $\W^{k,\infty}(\mO)$ (in fact \eqref{6.5} is satisfied as well on $\mO$, and hence $\E_\infty$ is actually semi-differentiable in the stronger sense of Hadamard and Penot). These hypotheses are weaker than those required by Theorem \ref{Theorem3} to be semi-differentiable on $\W^{k,\infty}(\Om)$, and can be restated as that $\E_\infty$ is semi-differentiable in $\smash{\W_{\textrm{loc}}^{k,\infty}(\Om)}$.

\smallskip

\noi (iii) To derive the fully nonlinear PDE \eqref{8.3} we actually only need to evaluate $( \mfD \E_\infty)^+_u$ to $\C^k$-test functions which are supported on balls, and not on all the test functions of \eqref{8.2}. Namely, by inspecting the proof that follows, we will only need $( \mfD \E_\infty)^+_u(\phi) \geq 0$ only for the following class of functions $\phi$:
\[
\Big\{\phi \in \C^k (\overline{\mB})\cap \W_0^{k,\infty}(\mB)\, :\, \mB \Subset \Om \text{ is a ball} \Big\}.
\]
\noi (iv) Finally, by Theorem \ref{Theorem5} we have that \eqref{8.2} is equivalent to $u$ being a (localised) absolute minimiser of \eqref{1.6} on $\Om$, if additionally $\H(x,\cdot)$ is level-convex and non-degenerate.  
\end{remark}

\BPT \ref{Theorem26}. Fix $x\in \Om$ and consider the sublevel set
\beq
\label{8.4}
\mathfrak S_x := \Big\{y\in \Om\ :\  \H\big(y, \D^{\vec k}u(y) \big) <  \H\big(x, \D^{\vec k}u(x) \big) \Big\}.
\eeq
Then, $\mathfrak S_x$ is open in $\Om$.

\smallskip

\noi {\bf \underline{Case 1.}}  If $\mathfrak S_x =\emptyset$, then we have $\H\big(\cdot, \D^{\vec k}u \big) \geq  \H\big(x, \D^{\vec k}u(x)\big)$ everywhere on $\Om$, and $x$ is a global interior minimum of $\H\big(\cdot, \D^{\vec k}u \big)$. This implies that
\beq
\label{8.5}
\D \big( \H\big(\cdot, \D^{\vec k}u \big)(x) =0,
\eeq
and therefore we deduce that
\beq
\label{8.12}
\H_{,\bfX_k} \big(x,\D^{\vec k}u(x) \big) \!:\! \Big(\D \big( \H\big(\cdot, \D^{\vec k}u \big)\big)(x)\Big)^{\!\ot k} =0.
\eeq
\noi {\bf \underline{Case 2.}}  If $\mathfrak S_x \neq \emptyset$, fix $\e \in \big(0,\dist(x,\p\Om)\big)$ and consider the largest ball contained inside $\mathfrak S_x \cap \mB_\e(x)$. Let us denote its centre by $z_\e$ and its radius by $r_\e$ (See Figure 2):
\[
\mB_{r_\e}(z_\e) \sub \mathfrak S_x \cap \mB_\e(x).
\]
Then, $ \mathfrak S_x$ satisfies an interior sphere condition at any of the points that $\p \mB_{r_\e}(z_\e)$ touches $\p \mathfrak S_x $. Fix any of these points, say $x_\e$, and consider the ball with half the radius $r_\e/2$ and centre the middle of the segment $[x_\e, z_\e]$, and let us denote it by $\mB^\e$:
\[
\mB^\e := \mB_{\frac{r_\e}{2}}\Big(\frac{z_\e +x_\e}{2} \Big).
\]
\[
\underset{\text{Figure 2. The idea of the proof of Theorem \ref{Theorem26}.}}{\includegraphics[scale=0.2]{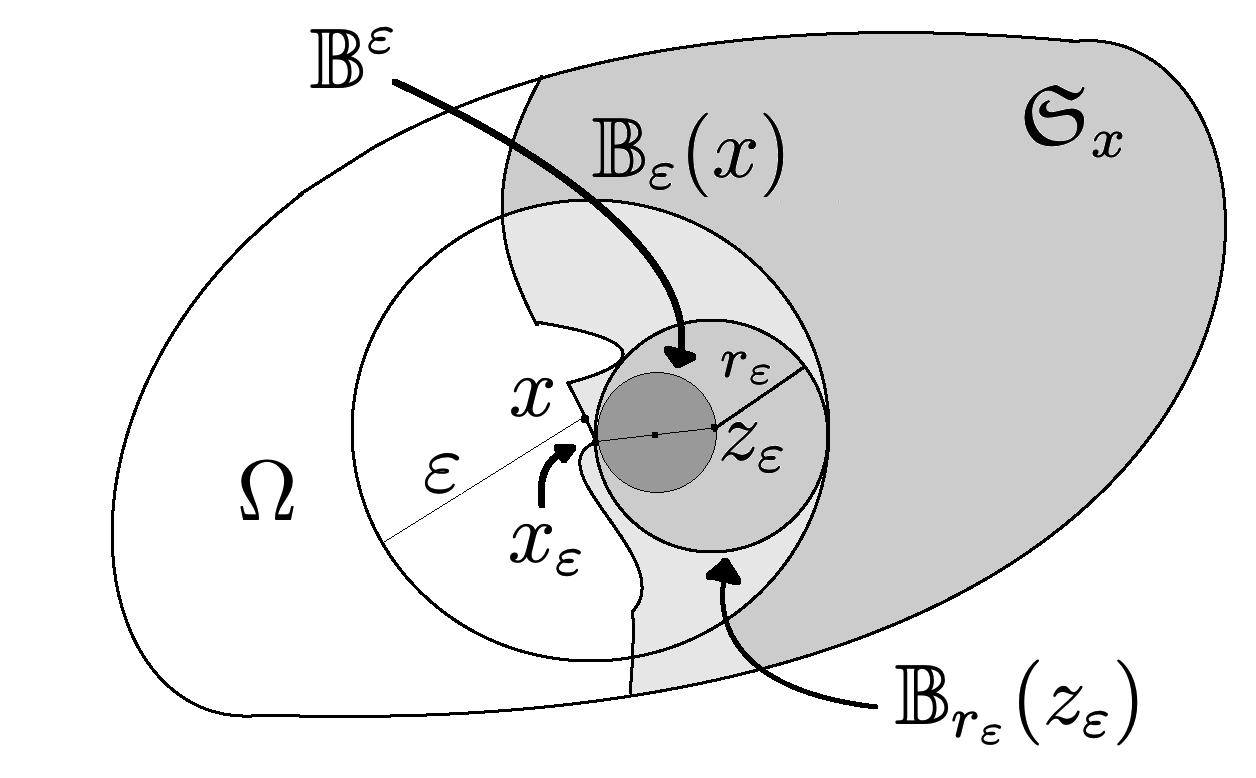}}
\]
Then, we have that $\mB^\e \sub \mathfrak S_x $ and also $(\p \mB^\e) \cap (\p \mathfrak S_x ) = \{x_\e\}$, namely the ball $\mB^\e$ is contained entirely into the sublevel set $\mathfrak S_x$, and the only intersection of their boundaries is the point $x_\e$ itself. This means that
\beq
\label{8.6}
\Argmax \big\{ \H\big(\cdot, \D^{\vec k}u \big)   :  \overline{\mB^\e}\big\} = \{x_\e\},
\eeq
as the maximum of $\H\big(\cdot, \D^{\vec k}u \big)$ over $\overline{\mB^\e}$ is attained only at $x_\e$. Additionally, $\mathfrak S_x $ still satisfies an interior sphere condition at $x_\e$ with respect to $\mB^\e$. Now we have to consider two sub-cases:

\smallskip

\noi {\bf \underline{Case 2(a).}} We have $\D \big( \H\big(\cdot, \D^{\vec k}u \big)\big)(x_\e) \neq 0$. In this eventuality, the implicit function theorem yields that $\p \mathfrak S_x $ is a $\C^1$ hypersurface locally near $x_\e$. Thus, the sphere $\p \mB_\e$ and $\p \mathfrak S_x $ share the same normal direction at $x_\e$, which implies
\beq
\label{8.7}
\D \big( \H\big(\cdot, \D^{\vec k}u \big)\big)(x_\e) = \al_\e \frac{2(x_\e-z_\e)}{r_\e}, \ \ \text{ for some } \al_\e \neq 0.
\eeq
Fix now $\phi \in \C^k (\overline{\mB^\e})\cap \W_0^{k,\infty}(\mB^\e)$. By assumption \eqref{8.2}, the regularity of $\H,u,\phi$ (recall also Theorem \ref{Theorem3} and Remark \ref{Remark17}) allows us to infer that
\beq
\label{8.8}
\max_{\Argmax \{ \H (\cdot, \D^{\vec k}u )  \, :\,  \overline{\mB^\e}\}} \H_{,\bfX} \big(\cdot, \D^{\vec k}u \big) : \D^{\vec k}\phi \geq 0.
\eeq
By \eqref{8.6} and \eqref{8.8}, since the argmax set of $\H\big(\cdot, \D^{\vec k}u \big)$ over $\mB^\e$ is the singleton set $\{x_\e\}$, by replacing $\phi$ with $-\phi$, we infer that
\beq
\label{8.9}
\H_{,\bfX} \big(x_\e, \D^{\vec k}u(x_\e) \big): \D^{\vec k}\phi(x_\e) = 0.
\eeq
Choose now any function $\ze \in \C^k[0,1]$ such that
\[
\begin{split}
&\ze^{(j)}(0)=0 \text{ for } j =1,\ldots, k,  
\\
 &\ze^{(j)}(1)=0  \text{ for } j =0,\ldots, k-1, \ \ \ \ze^{(k)}(1)=1,
\end{split}
\] 
and set 
\[
\phi_\e (x) : = \ze\bigg( \frac{2}{r_\e} \Big|   x - \frac{z_\e +x_\e}{2}  \Big| \bigg).
\]
Then, $\phi_\e \in \C^k (\overline{\mB^\e})\cap \W_0^{k,\infty}(\mB^\e)$, and additionally we have that its $k$-th order derivative on the boundary $\p\mB^\e$ is proportional to the $k$-fold tensor product of the normal vector to the boundary of $\mB^\e$. Therefore, in particular
\beq
\label{8.10}
\D^{j} \phi_\e (x_\e) = \left\{ 
\begin{array}{ll}
\be_{\e,k} \bigg( \dfrac{2(x_\e-z_\e)}{r_\e}\bigg)^{\! \ot k}, & j=k,\ \ \text{ for some } \be_{\e,k} \neq 0, \ms
\\
\mathbf 0, & j=0,\ldots,k-1.
\end{array}
\right.
\eeq
By substituting \eqref{8.7} into \eqref{8.10} we obtain
\beq
\label{8.11}
\D^{\vec k} \phi_\e (x_\e) = \bigg(\mathbf 0, \ldots, \mathbf 0, \frac{\be_{\e,k}}{(\al_\e)^k}\Big(  \D\big( \H\big(\cdot, \D^{\vec k}u\big) \big)(x_\e) \Big)^{\!\! \ot k}\bigg).
\eeq
By substituting  \eqref{8.11} into \eqref{8.9} and using the homogeneity of the equation, we deduce that
\[
\H_{,\bfX_k} \big(x_\e ,\D^{\vec k}u(x_\e) \big) \!:\! \Big(\D \big( \H\big(\cdot, \D^{\vec k}u \big)\big)(x_\e)\Big)^{\!\ot k} =0.
\]
By letting $\e\to0$, the regularity assumption \eqref{8.1} on $\H$ allows to infer that \eqref{8.12} is satisfied in this case as well.

\smallskip

\noi {\bf \underline{Case 2(b).}} We have $\D \big( \H\big(\cdot, \D^{\vec k}u \big)\big)(x_\e) =0$. By letting $\e\to0$, the regularity assumption \eqref{8.1} on $\H$ once again allows to infer that  \eqref{8.12} is satisfied in this case too.

\smallskip

\noi In conclusion, \eqref{8.12} holds at any $x\in \Om$ in all cases, which is \eqref{8.3}. The  theorem has been established.
\qed


\ms\ms

\noi{\bf Acknowledgements.} The author would like to thank Roger Moser (Bath, UK) and Simone Carano (Reading, UK),  for various fruitful scientific discussions relevant to the results presented herein.

\ms
\ms

\noi{\bf Declarations.} 

\noi $\bullet$ The author has no conflict of interest to disclose.

\noi $\bullet$ There are no data associated with this work.

\noi $\bullet$ This work contains no AI-generated content.

\ms



\bibliographystyle{amsplain}

\end{document}